\documentclass[10pt]{amsart}
\numberwithin{equation}{section}

\usepackage{times,amsmath}
\usepackage[T1]{fontenc}
\usepackage{dsfont}
\usepackage{mathrsfs}
\usepackage[colorlinks,citecolor=blue]{hyperref}
\usepackage{xcolor}
\usepackage[colorlinks]{hyperref}
\usepackage{xcolor}
\usepackage[a4paper,asymmetric,margin=1in]{geometry}
\usepackage{latexsym}
\usepackage{amsthm}
\usepackage{amssymb}
\usepackage{amsfonts}
\usepackage{amsmath}
\usepackage{longtable}
\usepackage{graphicx}
\usepackage{multirow}
\usepackage{multicol}
\usepackage[utf8]{inputenc}
 \usepackage[T1]{fontenc}
 \usepackage{mathtools}
  \usepackage{bm}
 \usepackage{enumerate}

\newcommand{\R}[0]{\mathbb{R}}

\newcommand{\Z}[0]{\mathbb{Z}}
\newcommand{\N}[0]{\mathbb{N}}

\newcommand{\sF}[0]{\mathcal{F}}

\newcommand{\sA}[0]{\mathcal{A}}
\newcommand{\sB}[0]{\mathcal{B}}
\newcommand{\sC}[0]{\mathcal{C}}

\newcommand{\sS}[0]{\mathcal{S}}
\newcommand{\sT}[0]{\mathcal{T}}
\newcommand{\sL}[0]{\mathcal{L}}

\newcommand{\sO}[0]{\mathcal{O}}
\newcommand{\sR}[0]{\mathcal{R}}
\newcommand{\sP}[0]{\mathcal{P}}
\newcommand{\sW}[0]{\mathcal{W}}
\newcommand{\sX}[0]{\mathcal{X}}
\newcommand{\sI}[0]{\mathcal{I}}

\newcommand{\sV}[0]{\mathcal{V}}
\newcommand{\sZ}[0]{\mathcal{Z}}
\newcommand{\sJ}[0]{\mathcal{J}}
\newcommand{\il}[0]{\langle}
\newcommand{\ir}[0]{\rangle}

\newcommand{\dbl}{d_{\tiny{\mbox{bl}}}}

\newcommand{\eps}[0]{\varepsilon}

\newcommand{\clp}{\mathcal{P}}

\newcommand{\RR}{\mathbb{R}}
\newcommand{\NN}{\mathbb{N}}
\newcommand{\ZZ}{\mathbb{Z}}

\newtheorem{theorem}{Theorem}[section]
\newtheorem{corollary}[theorem]{Corollary}
\newtheorem{lemma}[theorem]{Lemma}
\newtheorem{prop}[theorem]{Proposition}
\newtheorem{definition}{Definition}[section]
\newtheorem{remark}{Remark}[section]
\newtheorem{condition}{Condition}[section]

\definecolor{amet}{rgb}{0.8, 0.2, 0.8}

\definecolor{ques}{rgb}{0.8, 0.2, 0.2}

\allowdisplaybreaks

\begin{document}
\title[LDP for Mean Field Systems with Vanishing Noise]{Asymptotic Behavior of Stochastic Currents under Large Deviation Scaling with Mean Field Interaction and Vanishing Noise.}
\author[A. Budhiraja]{Amarjit Budhiraja}
\address{Department of Statistics and Operations Research,
University of North Carolina,
Chapel Hill, NC 27599, United States}
\email{budhiraj@email.unc.edu}
\author[M. Conroy]{Michael Conroy}
\address{Department of Statistics and Operations Research,
University of North Carolina,
Chapel Hill, NC 27599, United States}
\email{mconroy@live.unc.edu}
\begin{abstract}
We study the large deviation behavior of a system of diffusing particles with a mean field interaction, described through a 
collection of stochastic differential equations, in which each particle is driven by a vanishing independent Brownian noise.
An important object in the description of the asymptotic behavior, as the number of particles approach infinity and the noise intensity approaches zero, is the stochastic current associated with the interacting particle system in the sense of Flandoli et al. (2005). We establish a joint large deviation principle (LDP) for the path empirical measure for the particle system and the associated stochastic currents in the simultaneous large particle and small noise limit. Our work extends recent results of Orrieri (2018), in which the diffusion coefficient is taken to be identity, to a setting of a state dependent and possibly degenerate noise with the mean field interaction influencing both the drift and diffusion coefficients, and allows for a stronger topology on the space of stochastic currents in the LDP. Proof techniques differ from Orrieri (2018) and rely on methods from stochastic control, theory of weak convergence, and representation formulas for Laplace functionals of Brownian motions.
 \\



\noindent{\em Acknowledgements:} Research supported in part by the National Science Foundation (DMS-1814894 and DMS-1853968). AB is grateful for the support from Nelder Fellowship from Imperial College, London, where part of this research was completed.

\end{abstract}

\maketitle

\section{Introduction}

Consider the interacting particle system described through a collection of stochastic differential equations (SDEs) on $\R^d$ given as
\begin{equation}\label{eq:eq435r}
	dX_j^N(t) = b\left(X_j^N(t), V^N(t)\right)\,dt + \eps_N \sigma\left(X_j^N(t), V^N(t)\right)\,dW_j(t), \qquad 1\le j \le N,\qquad N\in\N, 
\end{equation}
on some finite time horizon $0\le t \le T$, where $\eps_N \downarrow 0$ as $N \to \infty$ and $\{W_j, j \in \N\}$ are independent $m$-dimentional Brownian motions on $[0,T]$. 
Here $V^N(t)$ is the empirical measure of the particle states at time $t$, namely
\[
	V^N(t) = \frac{1}{N}\sum_{j=1}^N \delta_{X_j^N(t)}, \qquad 0\le t \le T,
\]
and thus the interaction among the particles is of the mean-field type and influences both the drift and
diffusion coefficients of each particle. 
The law of large numbers (LLN) and fluctuation results for such mean-field systems have been widely studied, see for instance \cite{BraHep, Daw, KurXio, McK, Mel, Oel, ShiTan}. In particular, when $N\to \infty$, under conditions on the coefficients and the initial data, $\{V^N(t), 0 \le t \le T\}$ converges to the solution of the Vlasov equation
 \begin{equation*}
	\frac{\partial}{\partial t} V + \nabla \cdot b(\cdot, V)V = 0,
\end{equation*}
which can be formally written as
 \begin{equation}\label{eq:eqvj}
	\frac{\partial}{\partial t} V + \nabla \cdot \sJ = 0,
\end{equation}
where $\sJ \doteq   b(\cdot, V)V$ is the nonlinear current given as the limit of the {\em stochastic currents} 
\begin{equation}\label{eq:eqstcu}
	J^N(\varphi) = \frac{1}{N}\sum_{j=1}^N \int_0^T \varphi\left(t, X_j^N(t)\right) \circ dX_j^N(t),\end{equation}
defined for arbitrary smooth and compactly supported $\varphi: (0,T)\times \RR^d \to \RR^d$, where $\circ$ denotes the
Stratonovich integral.
Currents and their stochastic counterparts are key objects in geometric measure theory and play an important role in the theory of rough paths (cf. \cite{FlaTud, GMS, gub, lyo}). In the current context they  provide a convenient way to describe the asymptotics of the empirical measure process $V^N$.

In this work we are interested in studying the asymptotics of probabilities of significant deviations of the
empirical measure $V^N$, for the  $N$-particle microscopic stochastic evolution described by \eqref{eq:eq435r},
from its macroscopic hydrodynamic limit described by the first order Vlasov equation in \eqref{eq:eqvj}. A common approach to such a study is by establishing a general large deviation principle (LDP) on an appropriate abstract space from which the information on probabilities of deviations for specific events involving the $N$-particle system \eqref{eq:eq435r} can be obtained by a suitable application of the contraction principle. In view of the representation of the hydrodynamic limit of $V^N$ in terms of the nonlinear current functional $\sJ$, a natural candidate for an LDP are the pairs 
$(V^N, J^N)$ regarded as random elements of an appropriate space. 
Under the conditions on the coefficients considered in this work (see Condition \ref{maincondition}), $V^N$
will take values in $\sV  \doteq \sC([0,T], \sP_1(\R^d))$, namely, the space of continuous functions from $[0,T]$ to the space $\sP_1(\RR^d)$ of  probability measures on $\RR^d$ with finite first moment, equipped with the Wasserstein-1 distance (see Section \ref{sec:prelim} for precise definitions). The identification of an appropriate space for $J^N$ requires a bit more work (cf. \cite{flagubgiator, Orr}). In particular, note that \eqref{eq:eqstcu} describes an uncountably infinite collection of identities in which the right side is defined in an almost sure sense for each fixed $\varphi$. Thus a basic problem is to provide a pathwise representation for the collection 
\begin{equation}\label{eq:eq600r}
	\left\{\varphi \mapsto \frac{1}{N}\sum_{j=1}^N \int_0^T \varphi\left(t, X_j^N(t)\right) \circ dX_j^N(t)\right\}, \end{equation}
which defines a continuous, linear map on a suitable function space. This problem was studied in \cite{flagubgiator} (see also \cite{Orr}) where it was shown that 
there is a random variable $\sJ^N$ with values in a certain negative Sobolev space $\mathbf{H}^{-\mathbf{s}}$ of distributions (see Section \ref{sec:stochcurrents}), which gives a pathwise representation for the collection in \eqref{eq:eq600r} in the sense that 
$$\il \sJ^N, \varphi \ir = \frac{1}{N}\sum_{j=1}^N \int_0^T \varphi\left(t, X_j^N(t)\right) \circ dX_j^N(t)\qquad\mbox{ a.s.,}$$
for every smooth $\varphi$ with compact support. Thus the stochastic currents $\sJ^N$ can be viewed as random elements of  the Hilbert space $\mathbf{H}^{-\mathbf{s}}$, and the basic problem of interest is then to establish a large deviation principle for $(V^N, \sJ^N)$ in $\sV \times \mathbf{H}^{-\mathbf{s}}$.

This large deviation problem in the setting where $m=d$ and $\sigma = \mbox{Id}$
was studied in \cite{Orr} by direct change of measure arguments.  Specifically, \cite{Orr} treats the large deviation upper bound by first establishing an estimate for compact sets by considering an explicit tilt of the measure and then extends the estimate to all closed sets by establishing certain exponential tightness estimates. The lower bound is proved by exploiting connections between large deviations and $\Gamma$-convergence from \cite{mar}, in particular the key idea is to construct a suitable `recovery sequence' using results from \cite{FLOS}. One important aspect of the results and proof methods in \cite{Orr}
is that the LDP is established with the weak topology on the Hilbert space $\mathbf{H}^{-\mathbf{s}}$. Indeed, both the proofs of the upper and lower bounds rely on the use of the weak topology in important ways, e.g. since bounded sets are relatively compact under the weak topology in $\mathbf{H}^{-\mathbf{s}}$, in proving exponential tightness it suffices to estimate the probability that $\sJ^N$ takes values in the complement of a bounded ball.

In the current work we take a different approach to the study of the large deviation principle that is based on methods from stochastic control, the theory of weak convergence of probability measures, and Laplace asymptotics. This approach allows us to avoid establishing exponential tightness estimates of the form in \cite{Orr} and enables us to treat diffusion coefficients that are state dependent and possibly degenerate (see Section \ref{sec:maincond}). In addition,  since in this approach one needs to establish ordinary tightness rather than exponential tightness,  by appealing to certain compact embedding results for Sobolev spaces, we are able to establish an LDP with the norm topology on $\mathbf{H}^{-\mathbf{s}}$ instead of the weak topology considered in \cite{Orr}.  
In fact, we establish a somewhat more general large deviation principle than the one considered in \cite{Orr} from which the LDP for $(V^N, \sJ^N)$ can be deduced by the contraction principle. Specifically, we consider path empirical measures $\mu^N$ associated with the interacting particle system in \eqref{eq:eq435r} defined as 
\begin{equation*}
	\mu^N = \frac{1}{N}\sum_{j=1}^N \delta_{X_j^N}.
\end{equation*}
Under the conditions of this work it follows that $\mu^N$ is a random variable with values in
$\sP_1(\sC([0,T], \RR^d))$, namely the space of probability measures, on the Banach space of $\RR^d$-valued continuous trajectories on $[0,T]$, with integrable norm (equipped with the Wasserstein-1 metric). Our main result, Theorem \ref{main}, gives an LDP for $(\mu^N, \sJ^N)$ in $\sP_1(\sC([0,T], \RR^d)) \times \mathbf{H}^{-\mathbf{s}}$. Using the continuity of the map $ \nu \mapsto  \{ t\mapsto \nu \circ \pi_t^{-1}\}$  from
$\sP_1(\sC([0,T], \RR^d))$ into $\sV$, where $\pi_t$ is the projection map on $\sC([0,T], \RR^d)$ giving the evaluation at time $t$, we then deduce an LDP for the sequence $(V^N, \sJ^N)$ in $\sV\times \mathbf{H}^{-\mathbf{s}}$ in
Corollary \ref{cor:conpri}.
The rate function, in the general setting of a state dependent diffusion coefficient, is given as a value function of a certain deterministic mean field control problem with a quadratic cost (see \eqref{eq:RateFunction} and \eqref{eq:ratefncon}). In Proposition \ref{lem:sigeqi} we show that in the 
  special case where $\sigma = \mbox{Id}$, this representation of the rate function simplifies to a more explicit form given in terms of certain controlled Vlasov equations (see \eqref{eq:eq11004th}) which was obtained in \cite{Orr}.
  
As noted previously, proof techniques here are quite different from \cite{Orr}. The starting point of our analysis is a certain variational representation for exponential functionals of finite dimensional Brownian motions (see \cite{BouDup, BudDup}), using which the proof of the large deviation principle reduces to a 
study of tightness and convergence properties of certain controls and controlled analogues of the state processes $\{X^N_j, 1 \le j \le N\}$, state empirical measures $V^N$, path occupation measures $\mu^N$, and stochastic currents $\sJ^N$,  denoted as $\{\bar X^N_j, 1 \le j \le N\}$,  $\bar V^N$, $\bar\mu^N$, and $\bar\sJ^N$, respectively. For the upper bound proof we introduce certain joint empirical measures, denoted as $Q^N$ (see \eqref{eq:occupation}), of particle trajectories and associated control processes. The main step in the proof of the upper bound is to establish the tightness of the sequence $\{(\bar \mu^N, Q^N, \bar \sJ^N), N \in \N\}$ and to provide a suitable characterization of the weak limit points of this sequence. In particular, the tightness of the controlled stochastic currents $\{\bar \sJ^N\}$ is established with the norm topology on $\mathbf{H}^{-\mathbf{s}}$ and relies on approximations of $\{\bar\sJ^N\}$ by distributions with compact support as well as certain compact embedding results for Sobolev spaces (see Lemma \ref{compactembedding}).  The lower bound proof is constructive in that, given
a near optimal measure $\mu$ on $\sC([0,T], \RR^d)$ and a near optimal current $\sJ$ in a certain variational problem associated with the rate function, we construct a sequence of controls and controlled variables $(\bar \mu^N,  \bar \sJ^N)$ that converge to $(\mu, \sJ)$ in a suitable manner. The key ingredients in the proof here are a weak uniqueness (i.e. uniqueness in probability laws) property of certain equations associated with the controlled versions of the Vlasov equation \eqref{eq:eqvj} (see Lemma \ref{lem:weakuniq}) and certain infinite product space constructions.

Large deviation principles for weakly interacting diffusions as in \eqref{eq:eq435r} with non-vanishing noise (i.e. $\eps_N=1$) have been studied in \cite{DawGar}.   A different approach, based on weak convergence methods of the form used in the current work, was taken in \cite{BudDupFis}. The latter paper, in contrast to \cite{DawGar}, allowed for degenerate diffusion coefficients and for a mean field interaction in the diffusion coefficient. There have also been several works (in addition to the paper \cite{Orr} discussed above) that have studied large deviation problems for weakly interacting diffusions with small noise.
In particular,  see \cite{HIP}, \cite{RST}, and references therein, for large deviations results for McKean-Vlasov equations in the small noise limit; and see \cite{HerTug}  for an analysis of  interchanging of mean-field limit with the small noise limit at the level of rate function convergence. 
In a related direction, the paper \cite{BudCon} studied large deviation properties of a system of interacting diffusions in which each particle is driven by an independent individual source of noise and also by a vanishing amount of noise that is common to all particles. Different levels of intensity of the small common noise lead to different types of large deviation behavior, and the paper  \cite{BudCon} provided precise characterization of the various regimes.

\subsection{Organization.} The paper is organized as follows. In Section \ref{sec:prelim}, we specify our model, describe the space on 
which the large deviation principle will hold, define the rate function, and present our main large deviation result. 
Section \ref{sec:laplace} provides the proof of this result, with the proofs of its key lemmas given in Section \ref{sec:keylemmas}. 
The proofs of some auxiliary results  are given in the Appendix.

\subsection{Notation.}
\label{sec:notat}
The following notation will be used throughout. We use $\sC(R, S)$, $\sC_c(R, S)$, and $\sC^k(R, S)$, $k\in \N\cup\{\infty\}$, to 
denote the spaces of continuous, continuous and compactly supported, and $k$-times continuously differentiable functions 
from $R$ into $S$, respectively. Also, $\sC_c^k(R, S) = \sC_c(R, S) \cap \sC^k(R, S)$ for $k \in \N\cup\{\infty\}$. We denote by $L^2(\mu, R, S)$ the space of $\mu$-square integrable functions from 
$R$ into $S$. When $\mu$ is the Lebesgue measure, we will occasionally suppress it in the notation
and write $L^2(\mu, R, S)$ as $L^2(R, S)$.
The evaluation of a distribution $F$ on a test function $\varphi$ will be denoted by 
$\il F, \varphi\ir$, and integration of a function $f$ with 
respect to a measure $\mu$ will be denoted by $\il \mu , f \ir$. 
$\sB(S)$ denotes the collection of all Borel sets on $S$. For a Polish space $(S,d_S)$, $\sP(S)$ denotes the space of probability measures on $S$, endowed with the topology of weak convergence. A convenient metric on this space is the bounded Lipschitz metric given as
\begin{align*}
\dbl(\mu, \nu) &\doteq \sup_{f\in \sL_{\text{b}}(S)}  \left| \il\mu, f \ir - \il \nu, f\ir \right|,  \qquad \mu, \nu \in \sP(S),\qquad \mbox{where}\\
\sL_{\tiny\mbox{b}}(S) &\doteq \left\{ f \in \sC(S, \R) : \sup_{x\ne y} \frac{|f(x) - f(y)|}{d_S(x,y)} \le 1,  \; \sup_{x}|f(x)| \le 1\right\}.
\end{align*}
When $\theta \in \sP(S)$, the notation $E_\theta$ will be used to denote expectation on the probability space $(S, \sB(S), \theta)$.
For two spaces $S_1$ and $S_2$ and $\theta \in \sP(S_1\times S_2)$, $\theta_{(1)}$ and $\theta_{(2)}$
will denote the marginal distributions on $S_1$ and $S_2$, respectively. Similar notation will be used when more than two spaces are involved.
Euclidean norms 
will be denoted by $|\cdot|$. 
For a Polish space $(S,d_S)$, the space $\sC([0,T], S)$ will be equipped with the metric
$$d(x,y) = \sup_{0\le t \le T} d_S(x(t), y(t)), $$
under which it is a Polish space as well.
On $\sC([0,T], \R^d)$, we define the norm
$\|x\|_{\infty} \doteq \sup_{0\le t \le T} |x(t)|$, and the metric above becomes $d(x,y) = \|x-y\|_\infty$. 
We will use $\Rightarrow$ to denote convergence in distribution, and $\overset{P}{\to}$ to denote convergence in $P$-probability.
Infimum over an empty set, by convention, is taken to be $+\infty$. For a metric space $S$, a function $I: S\to [0,\infty]$ is 
called a rate function if $\{x\in S: I(x)\le l\}$ is a compact set for every $l <\infty$.

\section{Preliminaries and Main Result}\label{sec:prelim}

Let $(\Omega, \sF, P, \{\sF(t), 0\le t\le T\})$ be a filtered probability space where the filtration satisfies the {\em usual conditions} (see \cite[Definition 21.22]{Kle}). Fix $m \in \N$, and let $\{W_j, j \in \N\}$ be a sequence of independent $m$-dimensional $\{\sF(t)\}$-Brownian motions on the time horizon $0\le t\le T$. For each $N \in \N$, we consider the following system of stochastic differential equations in $\R^d$:
\begin{equation}\label{eq:model}
	X_j^N(t) = X_j^N(0) + \int_0^tb\left(X_j^N(s), V^N(s) \right)\,ds + \eps_N \int_0^t\sigma\left(X_j^N(s), V^N(s) \right)\,dW_j(s), \qquad 1\le j \le N, 
\end{equation}
where $V^N(t)$ denotes the $\sP(\R^d)$-valued empirical measure 
\begin{equation}\label{eq:EmpMzrV}
	V^N(t) \doteq \frac{1}{N}\sum_{j=1}^N \delta_{X^N_j(t)}, \qquad 0\le t \le T, 
\end{equation}
and $\{\eps_N, N \in \N\}$ is some sequence in $\R_+$ such that $\eps_N \downarrow 0$ as $N \to \infty$. Without loss of generality, we will assume that $\sup_N \eps_N \le 1$ throughout. 
Denote $\sX \doteq \sC([0,T], \R^d)$, and
define $\sP(\sX)$-valued random variables, given as the empirical measure of $(X_1^N, \ldots, X_N^N)$, by 
\begin{equation}\label{eq:EmpMzrmu}
	\mu^N \doteq \frac{1}{N}\sum_{j=1}^N \delta_{X_j^N}. 
\end{equation}
Note that the marginal of $\mu^N$ at time $t$ is $V^N(t)$, that is, defining $\pi_t: \sC([0,T], \R^d) \to \R^d$  as the projection map $\pi_t(x) = x(t)$, we have
\[
	\mu^N \circ \pi_t^{-1} = V^N(t), \qquad 0\le t \le T.
\]
We will view each $\mu^N$ as a random variable taking values in the Wasserstein-1 space which is defined as follows. 
For a Polish  space $(S, d_S)$, define the space $\sP_1(S)$ by 
\[
	\sP_1(S) \doteq \left\{ \mu \in \sP(S) : \int_{S} d_S(x,x_0)\,\mu(dx) < \infty \right\}, 
\]
for some choice of $x_0 \in S$ (the space does not depend on the choice of $x_0$). 
Then $\sP_1(S)$ is a Polish space under the Wassertstein-1 distance given by 
\begin{equation}\label{eq:eq1256r}
	d_{1}(\mu, \nu) \doteq \sup_{f\in \sL(S)}  \left| \il \mu, f \ir - \il \nu, f \ir \right|, \qquad \sL(S) \doteq \left\{ f \in \sC(S, \R) : \sup_{x\ne y} \frac{|f(x) - f(y)|}{d_S(x,y)} \le 1 \right\}. 
\end{equation}
For further details on Wassertstein spaces, we refer to \cite{Vil}. 
The particular cases of interest here are the spaces $\sP_1(\R^d)$ and $\sP_1(\sX)$, and the notation $d_1$ will be used for the metric on both spaces, with the distinction being clear
 from context. 
Noting that (under Condition \ref{maincondition} given below)
\begin{align*}
	\int_{\sX} d_{\sX}(x,0)\, \mu^N(dx) = \int_{\sX} \|x\|_{\infty}\, \mu^N(dx) = \frac{1}{N}\sum_{j=1}^N \left\|X^N_j\right\|_{\infty} <\infty \qquad \mbox{a.s.,}
\end{align*}
we see that indeed $\mu^N$ is a $\sP_1(\sX)$-valued random variable. Similarly, it can be checked that
$V^N$ is a $\sC([0,T], \sP_1(\RR^d))$-valued random variable. Throughout, we will denote $\sV \doteq \sC([0,T], \sP_1(\R^d))$.

\subsection{Main Conditions.}\label{sec:maincond}
The following is our main assumption on the coefficients. 
\begin{condition}\label{maincondition} There is some $L < \infty$ such that for all $x, y \in \R^d$ and $\mu, \nu \in \sP_1(\R^d)$, 
\[
	|b(x, \mu) - b(y, \nu)| + |\sigma(x, \mu) - \sigma(y, \nu)| \le L \left( |x-y| + d_{1}(\mu, \nu)\right),
\]
and $| \sigma(x, \mu)| \le L$. 
\end{condition}
Note that the above condition implies in particular that
 for all $x \in \R^d$ and $\mu \in \sP_1(\R^d)$, 
\begin{equation}\label{eq:eq208}
	|b(x, \mu)| \le L \left(1 + |x| + \int_{\R^d} |y|\,\mu(dy) \right). 
\end{equation}
with possibly a larger choice of $L$ than in Condition \ref{maincondition}.
By standard arguments, Condition~\ref{maincondition} implies that there exists a unique pathwise solution to \eqref{eq:model} for each $N \in \N$. 

\begin{remark}
	The boundedness of $\sigma$ is used in an important way at several places in the proof. It is a key ingredient in the proof of Lemma \ref{L2bound} which in turn is key to Lemmas \ref{lem:tightjn} and \ref{keytightness}. The last two lemmas are used in both the upper and lower bound proofs. For the upper bound proof one can relax the assumption on boundedness of $\sigma$ by using localization arguments of the form used in \cite{BudDup} (see e.g. \cite[Theorem 8.4]{BudDup_book}), however these localization arguments do not work in a simple manner for the proof of the lower bound.  Relaxing the condition on the boundedness of $\sigma$ remains an interesting open problem.
\end{remark}

We assume the following on the initial conditions of \eqref{eq:model}. 
\begin{condition}\label{initialcondition} For each $N \in \N$ and $1\le j \le N$, $X_j^N(0) = x_j^N \in \R^d$ is deterministic. The collection of initial conditions satisfies the following. 
\begin{enumerate}[(i)]
\item There exists some $\mu_0 \in \sP(\R^d)$ such that, $\dbl\left(V^N(0), \mu_0\right) \to 0$.

\item $\displaystyle \sup_{N \ge 1} \frac{1}{N} \sum_{j=1}^N \left| x_j^N \right|^2 < \infty$. 
\end{enumerate}
\end{condition}
Note that (i) and (ii) above imply that $\int_{\R^d} |x|^2 \,\mu_0(dx) < \infty$ from the observation 
\[
	\int_{\R^d} \left(|x|^2\wedge K\right)\,\mu_0(dx) = \lim_{N \to \infty} 
	\frac{1}{N}\sum_{j=1}^N \left(\left|x_j^N\right|^2\wedge K\right) \le \sup_{N \ge 1} \frac{1}{N} \sum_{j=1}^N \left| x_j^N \right|^2
\] 
for any $K \in (0, \infty)$, and applying Fatou's lemma. 
The above condition  also gives that,
as  $N \to \infty$,
\[
	d_1\left(V^N(0), \mu_0 \right) \to 0. 
\]

In order to prove the Laplace lower bound, we will make a stronger assumption given below on the diffusion coefficient $\sigma$  which says that it depends on the state of the system only through the empirical measure.  We will also require the convergence of the initial data in a somewhat stronger sense.
\begin{condition}\label{sigmacondition} 
	\begin{enumerate}[(i)]
	\item For each $x \in \R^d$ and $\mu \in \sP_1(\R^d)$, $\sigma(x, \mu) = \sigma(\mu)$. 
	
	\item For all $\mu_0$-integrable $f:\R^d \to \R$, 
	\[
		\left\il V^N(0), f \right\ir \to \il \mu_0, f \ir \qquad \mbox{as}\; N \to \infty
	\]
	\end{enumerate}
\end{condition}
\begin{remark}
	Part (i) of Condition \ref{sigmacondition}  is used in the proof of the weak uniqueness result in Lemma \ref{lem:weakuniq}. Relaxing this condition is a challenging open problem. The second part of Condition \ref{sigmacondition} is used in obtaining the convergence stated in \eqref{eq:PNLambdaN}.
\end{remark}
We are interested in the large deviations behavior of $\mu^N$ and $V^N$ as well as a 
collection of random linear functionals, referred to as {\em stochastic currents}, 
associated with the sequence of processes $\{X_j^N(t)\}$. We now introduce these objects. 
For each $N$ and $\varphi \in \sC_c^\infty([0,T] \times \R^d, \R^d)$ define
\begin{equation}\label{eq:JNdef}
	J^N(\varphi) \doteq \frac{1}{N}\sum_{j=1}^N \int_0^T \varphi\left(t, X_j^N(t)\right) \circ dX_j^N(t), 
\end{equation}
where the above is a Stratanovich stochastic integral. The relationship between Stratanovich and It\^o integrals gives the following formula for $J^N(\varphi)$:
\[
	J^N(\varphi) = \frac{1}{N}\sum_{j=1}^N \left( \int_0^T \varphi\left(t, X_j^N(t)\right)\cdot dX_j^N(t) + \frac{1}{2}\left\il \varphi\left(\cdot, X_j^N(\cdot)\right), X_j^N(\cdot) \right\ir_T \right), 
\]
where $\il Y, Z \ir_t$ denotes the quadratic variation at time $t$ of two continuous semimartingales $Y$ and $Z$. 
From results in \cite{flagubgiator}, $J^N$ can be viewed as a random linear functional on a suitable Sobolev space. We now briefly describe these results and make precise the space in which these random linear functionals take values.

\subsection{Stochastic Currents}\label{sec:stochcurrents}

Recall that for $k \in \N$, $H^k(\R^d, \R^d)$ is the Hilbert space of functions $f \in L^2(\R^d, \R^d)$ such that the distributional derivatives $D^\alpha f$ are also $L^2$ functions for all $|\alpha| \le k$, where $\alpha = (\alpha_1, \ldots, \alpha_d)$ denotes a multi-index. 
More generally, for any $s \in \R_+$, $H^s(\R^d, \R^d)$ is defined as the space of functions $f \in L^2(\R^d, \R^d)$ such that
\begin{equation}\label{eq:FourierSobolevnorm}
	\|f\|^2_{s} \doteq \int_{\R^d} |\hat f(\xi)|^2 (1+ |\xi|^2)^s\,d\xi  <\infty, 
\end{equation}
where $\hat f(\xi) = \int e^{-2\pi i \xi\cdot x} f(x)\,dx$ is the Fourier transform on $\R^d$. 
We refer the reader to \cite{adafou, folland, NezPalVal} for details on these spaces.

In order to describe the linear space associated with 
the map $\varphi \mapsto J^N(\varphi)$, 
we will need to consider a suitable Sobolev space of functions of time and space. Following \cite{BerButPis, flagubgiator, Orr}, a natural choice in this regard is the space
\[
	H^{s_1}\left((0,T), H^{s_2}\left(\R^d, \R^d\right)\right), 
\]
where $\mathbf{s} = (s_1, s_2) \in \left(\frac{1}{2}, 1\right) \times \left(\frac{d}{2} + 1, \infty\right)$ (see \cite{Orr} for a precise description of the space).
However in order to apply certain compact embedding results (see e.g. the proof of Lemma \ref{keytightness}) we will consider a slight modification of these spaces defined as follows.

Fix $a,b \in \R$ such that $a<0<T<b$ and define $U\doteq(a,b)$ and $\sO_d \doteq \left(\frac{1}{2}, 1\right) \times \left(\frac{d}{2} + 1, \infty\right)$. 
Then define
\[
	\mathbf{H}^{\mathbf{s}} \doteq H^{s_1}\left(U, H^{s_2}\left(\R^d, \R^d\right)\right), \qquad \mathbf{s} \in \sO_d,
\]
as the space of functions $f:U\times \RR^d \to \RR^d$ satisfying
\begin{equation}\label{eq:HstNorm}
\begin{aligned}
	\|f\|^2_{\mathbf{s}} 
	&\doteq \|f\|^2_{L^2(U, H^{s_2}(\R^d, \R^d))} + [f]^2_{\mathbf{s}} \\
	&\doteq  \int_{U} \|f(u, \cdot)\|^2_{s_2}\,du + \int_{U} \int_{U} \frac{\|f(u, \cdot) - f(v, \cdot)\|^2_{s_2}}{|u-v|^{1+2s_1}}\,du\,dv < \infty, 
\end{aligned}
\end{equation}
where $\|\cdot\|_{s_2}$ is as in \eqref{eq:FourierSobolevnorm}. The norm $\|\cdot\|_{\mathbf{s}}$
is usually referred to as a
{\it Gagliardo norm}, and in fact corresponds to an inner product which makes $\mathbf{H}^{\mathbf{s}}$ a separable Hilbert space (see \cite[Section 3]{NezPalVal}). The topological dual of the Hilbert space $\mathbf{H}^{\mathbf{s}}$
will be denoted as $\mathbf{H}^{-\mathbf{s}}$, namely
\[
	\mathbf{H}^{-\mathbf{s}} \doteq \left(\mathbf{H}^{\mathbf{s}} \right)'.
\]
The norm on this space is given as
\[
	\|F\|_{-\mathbf{s}} \doteq \sup_{\varphi \in \sC_c^\infty (U \times \R^d, \R^d)} \frac{|\il F, \varphi \ir|}{\|\varphi\|_{\mathbf{s}}}.
\]

For  $\varphi \in \sC_c^\infty(U\times \R^d, \R^d)$,  abusing notation, we let
$$J^N(\varphi) \doteq \frac{1}{N}\sum_{j=1}^N \int_0^T \varphi\left(t, X_j^N(t)\right) \circ dX_j^N(t).$$
Note that if $\varphi_{\mbox{\tiny{res}}}$ denotes the restriction of $\varphi$ to $[0,T]\times \R^d$, then
$J^N(\varphi)= J^N(\varphi_{\mbox{\tiny{res}}})$. Also, any $\varphi \in \sC_c^\infty([0,T]\times \R^d, \R^d)$ can be extended to a $\varphi_{\mbox{\tiny{ext}}} \in \sC_c^\infty(U\times \R^d, \R^d)$ where once more
$J^N(\varphi)= J^N(\varphi_{\mbox{\tiny{ext}}})$.
By a {\it pathwise realization} of the collection $\{\varphi \mapsto J^N(\varphi)\}$ on $\sC_c^\infty([0,T]\times \R^d, \R^d)\}$, we mean a random variable $\sJ^N$ with values in $\mathbf{H}^{-\mathbf{s}}$ such that for any
$\varphi \in \sC_c^\infty([0,T]\times \R^d, \R^d)$ and any extension $\varphi_{\mbox{\tiny{ext}}} $
of $\varphi$ in $\sC_c^\infty(U\times \R^d, \R^d)$, 
$\il\sJ^N, \varphi_{\mbox{\tiny{ext}}}\ir = J^N(\varphi)$ a.s.

The following result, giving the existence of a pathwise realization, follows along the lines of \cite{Orr} . The proof is an immediate consequence of Lemma \ref {JbarReal} below (on taking $u^N_j=0$ in the lemma), the proof of which is given in the Appendix. 

\begin{theorem}\label{JNpathwiserealization}  Suppose Conditions \ref{maincondition} and \ref{initialcondition} hold.
	Then for each $N \in \N$  and $\mathbf{s} \in \sO_d$, there is an $\mathbf{H}^{-\mathbf{s}}$-valued random variable $\sJ^N$ on $(\Omega, \sF, P)$
	such that for every $\varphi \in \sC_c^\infty(U \times \R^d, \R^d)$,
	$\il\sJ^N(\omega), \varphi\ir = [J^N(\varphi)](\omega)$ for a.e. $\omega \in \Omega$. Namely, $\sJ^N$ is a pathwise realization
	of $\{\varphi \mapsto J^N(\varphi)\}$.
\end{theorem}

Note that the pathwise realizations $\{\sJ^N\}$ are a.s. compactly supported in the first coordinate. Namely, if $U_0 \subset U$ is an open set such that $U_0\cap [0,T]=\emptyset$, then for all $\varphi$ with compact support in $U_0\times \R^d$, $\il\sJ^N, \varphi \ir=0$ a.s. In particular,
$\sJ^N$ is a distribution a.s. supported in $[0,T]\times \R^d$.

In this work we will prove a large deviation principle for the pair $(\mu^N, \sJ^N)$ in the space $\sP_1(\sX) \times \mathbf{H}^{-\mathbf{s}}$ for each $\mathbf{s} \in \sO_d$, from which a LDP describing the asymptotics of $V^N$ will follow by the contraction principle. 
We begin by introducing the rate function that will govern the large deviation behavior.

\subsection{Rate Function}
Let 
 $\sR$ denote the set of positive measures $r$ on $\sB([0,T]\times\R^m)$ such that $r([0,t]\times\R^m) = t$ for all $0\le t \le T$, and define 
 \[
 	\sR_1 \doteq \left\{ r \in \sR : \int_{[0,T]\times\R^m} |y|\,r(dt,dy) < \infty \right\}. 
\]
The space $\sR_1$  is a Polish space under the Wasserstein-1 metric (defined as in \eqref{eq:eq1256r} with $S=[0,T]\times \RR^m$). Each $r \in \sR_1$ can be decomposed as $r(dt, dy) = r_t(dy)\,dt$, where $r_t \in \sP(\R^m)$. For an $\sR_1$-valued random variable $\rho$, consider the McKean-Vlasov equation
 \begin{equation}\label{eq:LimitContSDE}
\begin{aligned}
	dX(t) &= b(X(t), V(t))\,dt + \int_{\R^m} \sigma(X(t), V(t))y \,\rho_t(dy)\,dt, \\
	V(t) &= P\circ X(t)^{-1}, \qquad V(0) = \mu_0, 
\end{aligned}
\end{equation}
where $X$ is stochastic process with sample paths in $\sX$, 
 $\rho(dt,dy) = \rho_t(dy)\,dt$ is the disintegration of $\rho$, and $\mu_0$ is the measure in Condition \ref{initialcondition}(i). 
 The distribution of a pair $(X, \rho)$ that solves \eqref{eq:LimitContSDE}, which is a probability measure on $\sZ \doteq \sX \times \sR_1$,  is called a weak solution of \eqref{eq:LimitContSDE}. 
 Let $\sS(\sZ) \subset \sP(\sZ)$ denote the set of all such weak solutions.
 With an abuse of notation, we will denote   the canonical coordinate maps on $(\sZ, \sB(\sZ))$ 
 by $(X, \rho)$ once more. That is, 
 \[
 	X(\xi, r) = \xi, \qquad \rho(\xi, r) = r, \qquad (\xi, r) \in \sZ. 
\] 
Note that if $\Theta\in \sS(\sZ)$, then $(X, \rho)$ satisfy \eqref{eq:LimitContSDE} $\Theta$-a.s. For each $\Theta \in \sP(\sZ)$ and $0\le t \le T$, define the measure
\[
	\nu_\Theta(t) \doteq \Theta \circ X(t)^{-1}, 
\]
which is an element of $\sP(\R^d)$. 
When $\Theta \in \sS(\sZ)$, it is easy to check that Condition \ref{maincondition} and Gronwall's lemma imply that $E_\Theta\left[ |X(t)| \right] < \infty$, and hence $\nu_\Theta(t) \in \sP_1(\R^d)$ for each $0\le t \le T$. Letting $\nu_\Theta$ denote the map $t \mapsto \nu_\Theta(t)$,  in fact we have that $\nu_\Theta \in \sV$. 
 For each 
$\varphi \in \sC_c^\infty(U \times \R^d, \R^d)$, define the map $G_{\varphi}: \sS(\sZ) \to \R$ by 
\begin{equation}\label{eq:eq457}
\begin{aligned}
	G_\varphi(\Theta) & \doteq E_\Theta\left[ \int_0^T \varphi\left(t, X(t) \right)\cdot dX(t) \right] \\
	&=  E_\Theta\left[ \int_0^T \varphi\left(t, X(t) \right) \cdot b(X(t), \nu_{\Theta}(t))dt \right] \\
	&\qquad +
	E_\Theta\left[ \int_{[0,T]\times\R^m}\varphi\left(t, X(t) \right) \cdot \sigma(X(t), \nu_{\Theta}(t)) y \,\rho(dt,dy) \right].
\end{aligned}
\end{equation}
Now let
 $$\sP_2(\sZ) \doteq \left\{ \Theta \in \sP(\sZ) : E_\Theta\left[ \int_{[0,T]\times\R^m}|y|^2\,\rho(dt,dy) \right]<\infty \right\},$$
and for $\sJ \in \mathbf{H}^{-\mathbf{s}}$, define 
\[
	\sP^*(\sJ) \doteq \left\{ \Theta \in \sS(\sZ)\cap \sP_2(\sZ) :   \il\sJ, \varphi\ir=G_\varphi(\Theta) \; \mbox{for all}\; \varphi \in \sC_c^\infty\left(U \times \R^d, \R^d\right) \right\}. 
\]
Define $I: \sP_1(\sX) \times \mathbf{H}^{-\mathbf{s}} \to [0,\infty]$ as
\begin{equation}\label{eq:RateFunction}
	I(\mu, \sJ) \doteq \inf \left\{ E_\Theta \left[ \frac{1}{2} \int_{[0,T]\times\R^m} |y|^2\,\rho(dt,dy) \right] : \Theta_{(1)} = \mu,  \Theta \in \sP^*(\sJ)\right\}, 
\end{equation}
where we recall that $\Theta_{(1)}$ denotes the marginal of $\Theta$ on $\sX$.

\begin{remark}
Note that the domain of the function $I$ depends on $\mathbf{s} \in \sO_d$.  However, it turns out (see Lemma \ref{Gpathwisereal}) that if $I(\mu, \sJ)<\infty$
for some $\mathbf{s} \in \sO_d$ and $(\mu, \sJ) \in \sP_1(\sX) \times \mathbf{H}^{-\mathbf{s}}$,
	then $\sJ \in \mathbf{H}^{-\mathbf{s}'}$ for all
	$\mathbf{s}' \in \sO_d$, and the value of $I(\mu, \sJ)$ is independent of $\mathbf{s}$.
\end{remark}

\subsection{Main Results}

In this section we present the main results.  For each $N \in \N$, let $\mu^N$, $V^N$ and $\sJ^N$ be as in \eqref{eq:EmpMzrmu} , \eqref{eq:EmpMzrV}, and Theorem \ref {JNpathwiserealization} respectively. Our first main result is a law of large numbers for $(\mu^N, V^N, \sJ^N)$. 

By using the Lipschitz property of $b$ it can be checked that for $\mu_0$ as in Condition \ref{initialcondition} and any $\RR^d$ valued random variable $\xi_0$ on $(\Omega, \sF, P)$ with distribution $\mu_0$, there is an a.s. unique solution $\xi$, with sample paths in $\sX$, to the equation
\begin{equation}\label{eq:deteq}
	\xi(t) = \xi_0 + \int_0^t b\left(\xi(s), V^*(s)\right)\, ds, \qquad V^*(t) = P\circ \xi(t)^{-1}, \; 0 \le t \le T.
\end{equation}
Let 
\begin{equation}\label{eq:LLNmu}
	\mu^*  = P\circ \xi^{-1}.
\end{equation}
Using the linear growth of $b$ and Condition \ref{initialcondition}(ii) it can be checked that $\mu^* \in \sP_1(\sX)$.

The following theorem gives the law of large numbers. Its proof is given in Section \ref{section:LLN}.
\begin{theorem}[LLN]\label{LLN}
Assume Conditions \ref{maincondition} and \ref{initialcondition} hold and let $\mathbf{s} \in \sO_d$. %
Then,
\[
	\left(\mu^N, V^N, \sJ^N\right) \overset{P}{\to} \left(\mu^*, V^*, \sJ^* \right) \qquad \mbox{as}\;N \to \infty, 
\]
in $\sP_1(\sX) \times \sV\times  \mathbf{H}^{-\mathbf{s}}$, where $V^*$ and $\mu^*$ are as in \eqref{eq:deteq} and \eqref{eq:LLNmu} and
$\sJ^*$ is characterized as
 \begin{equation}
 	\il \sJ^*, \varphi \ir = \int_0^T \left\il  V^*(t) ,\varphi(t, \cdot)\cdot b\left(\cdot, V^*(t)\right)\right\ir\,dt, 
 \end{equation}
 for $\varphi \in \sC_c^\infty(U\times \R^d, \R^d)$.
 \end{theorem}
 
 \begin{remark}
The pair $(V^*, \sJ^*)$ can alternatively be characterized as the unique solution of the equation
 \begin{equation}\label{eq:distlimita}
	\frac{\partial}{\partial t} V + \nabla \cdot b(\cdot, V)V = 0, \qquad \sJ = b(\cdot, V)V,  \qquad V(0) = \mu_0, 
\end{equation}
in the distributional sense on $(0,T) \times \R^d$, by which we mean that for all $\varphi \in \sC_c^\infty((0,T)\times \R^d, \R)$, 
\[
	\int_0^T \left\il V(t), \frac{\partial}{\partial t}\varphi(t, \cdot) \right\ir \,dt + \int_0^T \left\il  V(t) ,\nabla \varphi(t, \cdot)\cdot b(\cdot, V(t)) \right\ir\,dt = 0, 
\]
and for all $\varphi \in \sC_c^\infty((0,T)\times \R^d, \R^d)$, 
\begin{align*}
	\il \sJ, \varphi \ir &= \int_0^T \left\il  V(t), \varphi(t, \cdot)\cdot b(\cdot, V(t)) \right\ir \,dt.
\end{align*}
\end{remark}

Recall the function $I$ defined in \eqref{eq:RateFunction}, and for each $N \in \N$ let $a_N \doteq N/\eps_N^2$. Our main large deviation result is as follows. 

\begin{theorem}[LDP]\label{main} Assume Conditions \ref{maincondition} and \ref{initialcondition} hold. For each $\mathbf{s} \in \sO_d$, $I$ is a rate function on $\sP_1(\sX) \times \mathbf{H}^{-\mathbf{s}}$. Furthermore, 
\begin{enumerate}[(i)]
\item The sequence $\{(\mu^N, \sJ^N), N \in \N\}$ satisfies the large deviation upper bound on $\sP_1(\sX)  \times \mathbf{H}^{-\mathbf{s}}$ with speed $a_N$ and rate function $I$. Namely, for all closed sets $F$ in 
$\sP_1(\sX)  \times \mathbf{H}^{-\mathbf{s}}$, 
$$\limsup_{N\to \infty} \frac{1}{a_N} \log P\left(\left(\mu^N, \sJ^N\right) \in F\right) \le - \inf_{(\mu, \sJ)\in F} I(\mu, \sJ).$$

\item If in addition Condition \ref{sigmacondition} holds, then $\{(\mu^N, \sJ^N), N \in \N\}$ satisfies the large deviation lower bound on $\sP_1(\sX) \times \mathbf{H}^{-\mathbf{s}}$ with speed $a_N$ and rate function $I$. Namely, for all open sets $G$ in $\sP_1(\sX)  \times \mathbf{H}^{-\mathbf{s}}$,
$$\liminf_{N\to \infty} \frac{1}{a_N} \log P\left(\left(\mu^N, \sJ^N\right) \in G \right) \ge - \inf_{(\mu, \sJ)\in G} I(\mu, \sJ).$$
\end{enumerate} 
\end{theorem}

The proof of Theorem \ref{main}(i) is in Section \ref{section:upper}, and the proof of Theorem \ref{main}(ii) is in Section \ref{section:lower}. The rate function property of $I$ is proved in Section \ref{section:ratefunction}. The proof of Theorem \ref{LLN} is saved for Section \ref{section:LLN}, since it follows along the lines of the proof of the large deviation upper bound. 

It is easy to verify that the map $ \nu \mapsto  \{ t\mapsto \nu \circ \pi_t^{-1}\}$ is a continuous map from
$\sP_1(\sX)$ into $\sV$, and recall from above that each $\Theta \in \sS(\sZ)$ induces $\nu_\Theta \in \sV$. From this and the contraction principle we immediately have a large deviation principle for $\{(\mu^N, V^N, \sJ^N)\}$. In particular, we have the  following corollary.
Define
$\tilde I :\sV \times   \mathbf{H}^{-\mathbf{s}} \to [0,\infty]$ as
\begin{equation}\label{eq:ratefncon}
	\tilde I (V,\sJ) \doteq \inf \left\{ E_\Theta \left[ \frac{1}{2} \int_{[0,T]\times\R^m} |y|^2\,\rho(dt,dy) \right] : \nu_{\Theta} =  V, \Theta \in \sP^*(\sJ) \right\}.\end{equation}
\begin{corollary}\label{cor:conpri}
Assume Conditions \ref{maincondition} and \ref{initialcondition} hold. For each $\mathbf{s} \in \sO_d$, $\tilde I$ is a rate function on $\sV \times \mathbf{H}^{-\mathbf{s}}$. Furthermore, 
\begin{enumerate}[(i)]
\item The sequence $\{(V^N, \sJ^N), N \in \N\}$ satisfies the large deviation upper bound on $\sV  \times \mathbf{H}^{-\mathbf{s}}$ with speed $a_N$ and rate function $\tilde I$.
\item If in addition Condition \ref{sigmacondition} holds, then $\{(V^N, \sJ^N), N \in \N\}$ satisfies the large deviation lower bound on $\sV \times \mathbf{H}^{-\mathbf{s}}$ with speed $a_N$ and rate function $\tilde I$.
\end{enumerate}	
\end{corollary}
When $m=d$ and $\sigma(\mu)$ is invertible, one can give a more explicit representation for the rate function $\tilde I$ as follows. (A similar representation can be found in \cite{Orr} for the case $\sigma = \mbox{Id}$.)
For $\Theta \in \sS(\sZ) \cap \sP_2(\sZ)$ and $V \in \sV$ with $V= \nu_{\Theta}$, define
$$\eta^{\Theta}_t \doteq \Theta \circ (X(t), \sigma(V(t)) v(t) + b(X(t), V(t)))^{-1},$$
where
$v(t) = \int_{\RR^d} y \,\rho_t(dy)$ and $\rho_t$ is obtained from the disintegration of $\rho$ as
$\rho(dt,dy) = \rho_t(dy)\, dt$.
Note that, since $V= \nu_{\Theta}$,  $\eta^{\Theta}_t$ can be disintegrated as $\eta^{\Theta}_t(dx,dy) = \hat \eta^{\Theta}_t(x,dy) \,V(t, dx)$ for some $\hat \eta_t^{\Theta}$.
Define $\tilde I_0: \sV \times   \mathbf{H}^{-\mathbf{s}} \to [0,\infty]$ as
\begin{equation}
	\label{eq:eq11004th}
	\tilde I_0 (V, \sJ) \doteq \inf \left\{ \frac{1}{2} \int_0^T  \left\il V(t), 
	\int_{\RR^d} \left| \sigma^{-1}(V(t))(y- b(\cdot, V(t))\right|^2 \hat\eta^{\Theta}_t(\cdot, dy)
	 \right \ir   dt\right\}, 
\end{equation}
 where the infimum is taken over all $\Theta \in \sS(\sZ) \cap \sP_2(\sZ)$ such that $V= \nu_{\Theta}$ and
 with
 $h(t,x) = \int_{\RR^d} y\, \hat \eta^{\Theta}_t(x,\, dy),$
   $(V, \sJ)$ is  a distributional-sense solution of the equation
 \begin{equation}\label{eq:distlimit}
	\frac{\partial}{\partial t} V + \nabla \cdot h V = 0, \qquad \sJ = hV,  \qquad V(0) = \mu_0, 
\end{equation}
on $(0,T)\times \R^d$. 
Namely, for all $\varphi \in \sC_c^\infty((0,T)\times \R^d, \R)$, 
\[
	\int_0^T \left\il V(t), \frac{\partial}{\partial t}\varphi(t, \cdot) \right\ir \,dt + \int_0^T \left\il  V(t) ,\nabla \varphi(t, \cdot)\cdot h(t, \cdot) \right\ir\,dt = 0, 
\]
and for all $\varphi \in \sC_c^\infty((0,T)\times \R^d, \R^d)$, 
\begin{align}\label{eq:eq602}
	\il \sJ, \varphi \ir &= \int_0^T \left\il  V(t), \varphi(t, \cdot)\cdot h(t, \cdot) \right\ir \,dt.
\end{align}
The following result shows that $\tilde I = \tilde I_0$. The proof is given in Section \ref{sec:ieqtili}. 
\begin{prop}
	\label{lem:sigeqi}
	Suppose that $m=d$, Conditions \ref{maincondition} and \ref{sigmacondition}(i) are satisfied, and $\sigma(\mu)$ is invertible for every $\mu \in \clp_1(\RR^d)$. 
		Then $\tilde I = \tilde I_0$.
\end{prop}

\section{Laplace Asymptotics and Variational Representation}\label{sec:laplace}
Using the well-known equivalence (cf. \cite{BudDup_book, DupEll}) between the large deviation upper bound (resp. lower bound) and the Laplace upper bound (resp. lower bound), we will prove Theorem \ref{main} by establishing a Laplace principle on the space $\sP_1(\sX) \times \mathbf{H}^{-\mathbf{s}}$. Specifically, Theorem \ref{main}(i) will follow from the upper bound 
\begin{equation}\label{eq:Upper}
	\liminf_{N \to \infty} -\frac{1}{a_N} \log E\left[ e^{-a_N F\left(\mu^N, \sJ^N\right)} \right] \ge \inf_{(\mu, \sJ) \in \sP_1(\sX)  \times \mathbf{H}^{-\mathbf{s}}} \left( F( \mu, \sJ) + I(\mu, \sJ) \right), 
\end{equation}
and Theorem \ref{main}(ii) will follow from the lower bound 
\begin{equation}\label{eq:Lower}
	\limsup_{N \to \infty} -\frac{1}{a_N} \log E\left[ e^{-a_N F\left(\mu^N, \sJ^N\right)} \right] \le \inf_{(\mu, \sJ) \in \sP_1(\sX) \times \mathbf{H}^{-\mathbf{s}}} \left( F(\mu, \sJ) + I(\mu, \sJ) \right), 
\end{equation}
where $F$  is any  bounded, continuous function on $\sP_1(\sX) \times \mathbf{H}^{-\mathbf{s}}$. 

The inequality \eqref{eq:Upper} will be proved in Section \ref{section:upper} (under Conditions \ref{maincondition} and \ref{initialcondition}), and the inequality \eqref{eq:Lower} will be proved in Section \ref{section:lower} (under Conditions \ref{maincondition}, \ref{initialcondition}, and \ref{sigmacondition}). The rate function property of $I$ is shown in Section \ref{section:ratefunction}.
The starting point for both upper and lower bounds is the following variational representation.

\subsection{Variational Representation}
\label{sec:varrep}

Let $\sA_N$ denote the class of $\R^{Nm}$-valued  $\sF(t)$-progressively measurable processes $u$ such that $E\left[\int_0^T |u(t)|^2\,dt\right] < \infty$. 
For  $u^N = (u_1^N, \ldots, u_N^N) \in \sA_N$, with each $u_j^N(t)$ taking values in $\R^m$, consider the controlled version of \eqref{eq:model} given as
\begin{equation}\label{eq:controlledmodel}
	d\bar{X}_j^N(t) = b\left(\bar{X}_j^N(t), \bar{V}^N(t) \right)\,dt + \eps_N \sigma\left(\bar{X}_j^N(t), \bar{V}^N(t) \right)\,dW_j(t) + \sigma\left(\bar{X}_j^N(t), \bar{V}^N(t) \right)u_j^N(t)\,dt, 
\end{equation}
where $\bar X_j^N(0) = x_j^N$ and 
\[
	\bar{V}^N(t) \doteq \frac{1}{N}\sum_{j=1}^N \delta_{\bar{X}_j^N(t)}, \qquad 0 \le t \le T. 
\]
Analogous to \eqref{eq:EmpMzrmu}, $\bar \mu^N$ will denote the empirical measure of $(\bar X_1^N, \ldots, \bar X_N^N)$, so that $\bar \mu^N \circ \pi^{-1}_t = \bar V^N(t)$ for each $0\le t \le T$. 
We will also need a controlled analogue of the stochastic current in Theorem \ref{JNpathwiserealization}.
 For $\varphi \in \sC_c^\infty(U\times\R^d, \R^d)$, define
\begin{equation}\label{eq:Jbarj}
	\bar J_j^N(\varphi) \doteq \int_0^T \varphi\left(t, \bar X_j^N(t)\right)\circ d \bar X_j^N(t), \qquad 
	\bar J^N(\varphi)  \doteq \frac{1}{N}\sum_{j=1}^N \bar J_j^N(\varphi).
\end{equation}
The proof of the following result, which is given in the Appendix, is similar to that of Theorem \ref{JNpathwiserealization}.
\begin{lemma}\label{JbarReal}
 Suppose that Conditions \ref{maincondition} and \ref{initialcondition} hold.
 Then, for each $N\in \N$, $1\le j \le N$, and $\mathbf{s} \in \sO_d$, there is a nonnegative square-integrable random variable $C^N_{j,\mathbf{s}}$ such that for all $\varphi \in \sC_c^\infty(U \times\R^d, \R^d)$, 
\[
	\left| \bar J_j^N(\varphi) \right| \le C^N_{j,\mathbf{s}} \|\varphi\|_{\mathbf{s}} \qquad \mbox{a.s.}
\]
In particular, the collection $\{\varphi \mapsto \bar J^N(\varphi)\}$ has a pathwise realization $\bar \sJ^N$ on $(\Omega, \sF, P)$, namely
$\bar \sJ^N$ is an $\mathbf{H}^{-\mathbf{s}}$-valued random variable such that 
 $\il\bar \sJ^N(\omega), \varphi\ir = [\bar J^N(\varphi)](\omega)$ for a.e. $\omega \in \Omega$ and all $\varphi \in \sC^\infty_c(U \times \R^d, \R^d)$. 
Furthermore, if 
\begin{equation}\label{eq:eq525}
	\sup_{N \ge 1} E\left[ \frac{1}{N}\sum_{j=1}^N \int_0^T \left|u_j^N(t)\right|^2\,dt \right] < \infty, 
\end{equation}
then $\sup_{N \ge 1} E\left[ \frac{1}{N}\sum_{j=1}^N \left(C^N_{j, \mathbf{s}} \right)^2 \right] < \infty. $
In particular, if 
$C^N_{\mathbf{s}} \doteq \frac{1}{N}\sum_{j=1}^N C^N_{j, \mathbf{s}}$, then
$\sup_{N \ge 1} E\left[ \left( C^N_{\mathbf{s}} \right)^2 \right] < \infty$.
\end{lemma}

The following variational representation follows from  \cite{BouDup, BudDup} (see also \cite{BudDupFis}).
Specifically, the case where $\{\sF(t)\}$ is the filtration generated by the $m$-dimensional Brownian motions $\{W_j\}$ is covered in \cite{BouDup}, while the setting of a general filtration is treated in
\cite{BudDup}. Recall that  $a_N = N /\eps_N^2$. 
\begin{theorem}[Variational Representation]\label{representation} 
	Suppose that Conditions \ref{maincondition} and \ref{initialcondition} hold. Let $\mathbf{s} \in \sO_d$ and let $F$ be a real-valued, bounded, continuous function on $\sP_1(\sX)  \times \mathbf{H}^{-\mathbf{s}}$.
Then for each $N \in \N$,
\begin{equation}
	-\frac{1}{a_N} \log E\left[ e^{-a_NF\left(\mu^N, \sJ^N \right)} \right] = \inf_{u^N \in \sA_N} E\left[ \frac{1}{2N} \sum_{j=1}^N \int_0^T \left|u_j^N(t)\right|^2\,dt + F\left(\bar \mu^N, \bar \sJ^N \right) \right]. 
\end{equation}
\end{theorem}

\subsection{Tightness Properties.}
\label{sec:tightprop}
The following lemma gives a key tightness property that will be needed in the proofs of both upper and lower Laplace bounds. The proof is given in Section \ref{sectiontightnessproof}.
\begin{lemma}\label{keytightness} 
	Suppose  Conditions \ref{maincondition} and \ref{initialcondition} hold.
	Fix $\mathbf{s} \in \sO_d$, and let $\{u^N, N \in \N\}$ with $u_N \in \mathcal{A}_N$ for each $N$ be such that 
	\[
		\sup_{N\ge1} E\left[ \frac{1}{N}\sum_{j=1}^N \int_0^T \left|u_j^N(t) \right|^2\,dt \right] < \infty.
	\]
	Let $\bar X_j^N$,  $\bar \mu^N$, and $\bar \sJ^N$ be the controlled sequences corresponding to sequence of controls $\{u^N\}$ as defined in Section \ref{sec:varrep}. 
	For each $j$ and $N$, let $\rho_j^N$ be the $\sR_1$-valued random variable given as
	\[
		\rho_j^N(dt, dy) \doteq \delta_{u_j^N(t)}(dy)\,dt, 
	\]
	and consider the sequence of $\sP(\sZ)$-valued random variables defined as
	\begin{equation}\label{eq:occupation}
		Q^N \doteq \frac{1}{N} \sum_{j=1}^N \delta_{\left(\bar X_j^N,  \rho_j^N \right) }, \qquad N \in \N.
	\end{equation}
Then, 
\begin{enumerate}[(i)]

\item The sequence $\{(\bar \mu^N, Q^N, \bar \sJ^N), N \in \N\}$ is tight in $\sP_1(\sX)\times\sP(\sZ) \times \mathbf{H}^{-\mathbf{s}}$, 

\item If $(\bar \mu^N, Q^N, \bar\sJ^N) \Rightarrow (\bar \mu, Q, \bar\sJ)$ as $N \to \infty$ in $\sP_1(\sX)\times \sP(\sZ) \times \mathbf{H}^{-\mathbf{s}}$, 
then $Q_{(1)} = \bar \mu$ and 
$Q \in \sP^*(\bar\sJ)$ a.s. 
\end{enumerate}
\end{lemma}

\subsection{Proof of the Upper Bound}\label{section:upper}
In this section we prove part (i) of Theorem \ref{main} by showing that \eqref{eq:Upper} holds. Assume Conditions \ref{maincondition} and \ref{initialcondition}.
Fix  $\mathbf{s} = (s_1, s_2) \in \sO_d$, and a  real-valued, bounded, continuous function 
$F$ on $\sP_1(\sX) \times \mathbf{H}^{-\mathbf{s}}$. 
Let $\eps \in (0,1)$, and  using Theorem \ref{representation} choose $\{u^N, N \in \N\}$ with $u_N \in \mathcal{A}_N$ for each $N$ such that 
\begin{equation}\label{eq:upperbndeps}
	-\frac{1}{a_N} \log E\left[ e^{-a_NF\left(\mu^N, \sJ^N \right)} \right] \ge E\left[ \frac{1}{2N} \sum_{j=1}^N \int_0^T \left|u_j^N(t)\right|^2\,dt + F\left(\bar \mu^N,  \bar \sJ^N \right) \right] - \eps, 
\end{equation}
where $(\bar \mu^N,  \bar \sJ^N)$ are controlled variables corresponding to the control $u^N$ as defined in Section \ref{sec:varrep}.
From the boundedness of $F$ it follows that
\[
	\sup_{N\ge 1} E\left[ \frac{1}{2N} \sum_{j=1}^N \int_0^T \left| u_j^N(t) \right|^2\,dt \right] \le 2\sup_{(\mu, \sJ)\in \sP_1(\sX) \times \mathbf{H}^{-\mathbf{s}}} \left|F(\mu, \sJ)\right| + 1 < \infty. 
\]
By Lemma \ref{keytightness}, $(\bar \mu^N, Q^N, \bar\sJ^N)$ is tight in $\sP_1(\sX)\times \sP(\sZ) \times \mathbf{H}^{-\mathbf{s}}$.
Thus the sequence $(\bar \mu^N, Q^N, \bar\sJ^N)$ has a weak limit point $(\bar \mu, Q, \bar\sJ)$ along some subsequence, and once again by  Lemma \ref{keytightness}, $Q \in \sP^*(\sJ)$ and  $Q_{(1)} = \bar \mu$ a.s. Assume without loss of generality that $(\bar \mu^N, Q^N, \bar \sJ^N) \Rightarrow (\bar \mu, Q, \bar\sJ)$ along the full sequence. 
Noting that $Q_{(1)}^N = \bar \mu^N$,   
we have, by \eqref{eq:upperbndeps}, 
\[
	-\frac{1}{a_N} \log E\left[ e^{-a_NF\left(\mu^N, \sJ^N \right)} \right] \ge E\left[ \frac{1}{2} \int_{\sR_1} \int_{[0,T]\times \R^m} |y|^2\,r(dt,dy)\, Q_{(2)}^N(dr) + F\left(Q_{(1)}^N, \bar \sJ^N \right) \right] - \eps.
\]
By Fatou's lemma and lower semicontinuity of the map $r \mapsto \int_{[0,T]\times\R^m}  |y|^2\,r(dt,dy)$ on $\sR_1$,
\begin{align*}
	&\liminf_{N \to \infty} -\frac{1}{a_N} \log E\left[ e^{-a_NF\left(\mu^N, \sJ^N \right)} \right] \\
	&\ge E\left[ \frac{1}{2} \int_{\sR_1} \int_{[0,T]\times\R^m} |y|^2\,r(dt,dy)\, Q_{(2)}(dr) + F\left(Q_{(1)}, \bar\sJ \right) \right] - \eps \\
	&= E\left[ E_Q\left[\frac{1}{2}  \int_{[0,T]\times\R^m} |y|^2\,\rho(dt,dy)\right] + F\left(\bar \mu, \bar\sJ \right) \right] - \eps \\
		&\ge \inf_{(\mu, \sJ) \in \sP_1(\sX) \times \mathbf{H}^{-\mathbf{s}}} \left( I(\mu, \sJ) + F(\mu, \sJ) \right) - \eps,
\end{align*}
where the last line follows on recalling the definition of $I$ and the facts that $Q \in \sP^*(\sJ)$ and  $Q_{(1)} = \bar \mu$ a.s.
Since $\eps \in (0,1)$ is arbitrary, this completes the proof of the upper bound in \eqref{eq:Upper}
and thus that of Theorem \ref{main}(i). \hfill\qed

\subsection{Proof of the Lower Bound}\label{section:lower}
In this section we prove part (ii) of Theorem \ref{main} by showing \eqref{eq:Lower}. 
Fix  $\mathbf{s} = (s_1, s_2) \in \sO_d$. We assume Conditions  \ref{maincondition}, \ref{initialcondition}, and \ref{sigmacondition} hold.
Let $\eps \in (0,1)$ and choose $(\Theta_0, \sJ_0) \in \sP(\sZ) \times \mathbf{H}^{-\mathbf{s}}$ such that $\Theta_0 \in \sP^*(\sJ_0)$
and 
\begin{equation}\label{eq:eq609}
	E_{\Theta_0} \left[ \frac{1}{2} \int_{[0,T]\times\R^m} |y|^2\,\rho(dt,dy) \right] + F\left((\Theta_0)_{(1)}, \sJ_0\right) \le \inf_{(\mu, \sJ) \in \sP_1(\sX) \times \mathbf{H}^{-\mathbf{s}}}\left( I(\mu, \sJ) + F(\mu, \sJ) \right) + \eps. 
\end{equation}
To prove the lower bound we will construct a sequence $\{u^N\}$ of controls on some filtered probability space such that $u^N \in \sA_N$ for each $N$ and 
\begin{equation}\label{eq:specialuN}
\begin{aligned}
	&\limsup_{N\to\infty} E\left[ \frac{1}{2N} \sum_{j=1}^N \int_0^T \left|u_j^N(t)\right|^2\,dt + F\left(\bar \mu^N, \bar \sJ^N\right) \right] \\
	&\le E_{\Theta_0} \left[ \frac{1}{2} \int_{[0,T]\times \R^m} |y|^2\,\rho(dt,dy) \right] + F\left((\Theta_0)_{(1)}, \sJ_0\right),
\end{aligned}
\end{equation}
where $\bar \mu^N$ and $\bar \sJ^N$ are the controlled processes corresponding to $\{u^N\}$.
It will then follow by Theorem \ref{representation} and \eqref{eq:eq609} that
\begin{align*}
	\limsup_{N \to\infty} -\frac{1}{a_N} \log E\left[ e^{-a_N F\left(\mu^N, \sJ^N\right)} \right] &\le \limsup_{N \to\infty} E\left[ \frac{1}{2N} \sum_{j=1}^N \int_0^T \left| u_j^N(t) \right|^2\,dt + F\left(\bar \mu^N, \bar \sJ^N \right) \right] \\
	&\le \inf_{(\mu, \sJ) \in \sP_1(\sX)  \times \mathbf{H}^{-\mathbf{s}}}\left( I(\mu, \sJ) + F(\mu, \sJ) \right) + \eps.
\end{align*}
 Since $\eps > 0$ is arbitrary, 
the lower bound follows. 

The construction of a sequence $\{u^N\}$ such that
 the inequality in \eqref{eq:specialuN} holds will need the following uniqueness property.
\begin{definition} Let $\theta : \sZ \to \R^d\times \sR_1$ denote the map $\theta(\xi, r) = (\xi(0), r)$. We say that weak uniqueness of solutions of \eqref{eq:LimitContSDE} holds if $\Theta_1, \Theta_2 \in \sS(\sZ) \cap \sP_2(\sZ)$ and $\Theta_1 \circ \theta^{-1} = \Theta_2 \circ \theta^{-1}$ implies that $\Theta_1 = \Theta_2$. 
\end{definition}

The following lemma is key to the proof of the lower bound. The proof is provided in Section \ref{sectionweakuniq}. 
Recall that in this section we assume that Conditions  \ref{maincondition}, \ref{initialcondition}  and \ref{sigmacondition} hold.
\begin{lemma}\label{lem:weakuniq} Weak uniqueness of solutions holds for \eqref{eq:LimitContSDE}. 
\end{lemma}

We now construct the sequence $\{u^N\}$ that satisfies \eqref{eq:specialuN}. 
Because $\Theta_0 \in \sS(\sZ)$, we can disintegrate
\[
	\Theta_0 \circ \theta^{-1}(dx \, dr) = \mu_0(dx)\,\Lambda_0(x, dr), 
\]
for some measurable map $\Lambda_0: \RR^d \to \sP(\sR_1)$. 
Let $\sW \doteq \sC([0,T], \R^m)$, and let $\gamma$ be the standard Wiener measure on $\sW$.
Define a measurable map $\Lambda: \R^d \to \sP(\sR_1\times \sW)$ as
\[
	\Lambda(x, dr, dw) \doteq \Lambda_0(x, dr) \otimes \gamma(dw), \qquad x \in \R^d. 
\]
Define the measurable space $(\tilde \Omega, \tilde \sF)$ by 
\[
	\tilde \Omega = (\sR_1 \times \sW)^{\infty}, \quad \tilde \sF = \sB\big(\tilde \Omega\big),
\]
where an element $(r,w) \in \tilde \Omega$ has the coordinates $r = (r_1, r_2, \ldots)$ and $w = (w_1, w_2, \ldots)$ 
with $r_j \in \sR_1$ and $w_j \in \sW$ for each $j$. 
Consider  the canonical  filtration $\{\tilde \sF(t)\}$ on $(\tilde \Omega, \tilde \sF)$ defined as
$$\tilde \sF(t) \doteq \sigma \left(w_j(s), r_j([0,s]\times A), \, j \in \NN, \,A \in \sB(\R^m), \,s \le t\right), \qquad 0 \le t \le T,$$
and define the sequence $\{P^N, N \in \N\}$ of probability measures on $(\tilde \Omega, \tilde \sF)$ by 
\[
	P^N(dr,dw) = \bigotimes_{j \le N} \Lambda\left(x_j^N, dr_j, dw_j\right) \, \bigotimes_{j > N} \left((\Theta_0)_{(2)}\otimes\gamma\right)(dr_j, dw_j), 
\]
where $\{x_j^N\}$ are as in Condition \ref{initialcondition}. Next define the sequence 
$\{\Lambda^N, N \in \N\}$ of $\sP(\R^d \times \sR_1)$-valued random variables on $(\tilde \Omega, \tilde \sF)$ by 
\[
	\Lambda^N \doteq \frac{1}{N} \sum_{j=1}^N \delta_{\left(x_j^N, \rho_j\right)}, 
\]
where for each $j \in \N$, $\rho_j$ is the $\sR_1$-valued random variable on $(\tilde \Omega, \tilde \sF)$ defined as
 $\rho_j(r,w) = r_j$. Using Condition \ref{sigmacondition}(ii), we see by a standard argument that
\begin{equation}\label{eq:PNLambdaN}
	P^N \circ (\Lambda^N)^{-1} \to \delta_{\Theta_0 \circ \theta^{-1}} \qquad \mbox{as}\; N \to \infty,  
\end{equation}
in $\sP(\sP(\R^d \times \sR_1))$. 

Now, for each $j \in \N$, disintegrating $\rho_j$ as $\rho_j(dt,dy) = (\rho_j)_t(dy)\, dt$, define
\[
	u_j(t) \doteq \int_{\R^m} y\,(\rho_j)_t(dy), \qquad 0 \le t \le T, 
\]
and define $u^N \doteq (u_1, \ldots, u_N)$ for each $N \in \N$. Furthermore, for each $j$ and $(r,w) \in \tilde \Omega$, let
\[
	W_j(t, (r,w)) \doteq w_j(t), \qquad 0 \le t \le T.
\]
Then for each $N$, $W_1, \ldots, W_N$ are mutually independent $\{\tilde \sF(t)\}$-Brownian motions on 
$(\tilde \Omega, \tilde \sF, P^N)$. 
Recall that in this section we are assuming Condition \ref{sigmacondition}, and so $\sigma(x,\nu) = \sigma(\nu)$ for
$(x, \nu) \in \RR^d\times \sP_1(\RR^d)$.
Let $(\bar{X}_1^N, \ldots, \bar{X}_N^N)$ be the unique pathwise solution (which is guaranteed due to Conditions  \ref{maincondition} and \ref{initialcondition})
on $(\tilde \Omega, \tilde \sF, P^N)$ of the system
\begin{align*}
	\bar{X}_j^N(t) &= x_j^N + \int_0^t b\left(\bar{X}_j^N(s), \bar{V}^N(s) \right)\,ds + \eps_N \int_0^t \sigma\left(\bar{V}^N(s) \right)\,dW_j(s) + \int_0^t\sigma\left(\bar{V}^N(s) \right)u_j(s)\,ds, \\
	\bar{V}^N(t) &= \frac{1}{N} \sum_{j=1}^N \delta_{\bar{X}_j^N(t)}, \qquad 0\le t \le T, \qquad 1 \le j \le N.
\end{align*}
Also let $\bar \mu^N = \frac{1}{N}\sum_{j=1}^N \delta_{\bar X_j^N}$. Now define the sequence $\{Q^N\}$
of $\sP(\sZ)$-valued random variables as
\[
	Q^N \doteq \frac{1}{N} \sum_{j=1}^N \delta_{\left(\bar{X}_j^N, \rho_j \right)}, \qquad N \in \N.
\]
 Letting $E^N$ denote expectation on $(\tilde \Omega, \tilde \sF, P^N)$, 
 we note that for a measurable $f : \sR_1 \to \R_+$, 
 \begin{equation}\label{eq:eq630}
 \int_{\sR_1} f(r) \,(\Theta_0)_{(2)}(dr) < \infty\qquad  \mbox{implies}\qquad  E^N\left[ \frac{1}{N} \sum_{j=1}^N f\left(\rho_j\right)\right] \to \int_{\sR_1} f(r) \, (\Theta_0)_{(2)}(dr). 
 \end{equation}
Indeed, if $g(x) = \int_{\sR_1} f(r)\,\Lambda_0(x,dr)$ for $x \in \R^d$, then 
 $$E^N\left[ \frac{1}{N} \sum_{j=1}^N f\left(\rho_j\right)\right] = \frac{1}{N} \sum_{j=1}^N  \int_{\sR_1} f(r)\,\Lambda_0\left(x_j^N, dr\right) = \frac{1}{N} \sum_{j=1}^N g\left(x_j^N\right),$$
and 
 \begin{align*}
 	\int_{\R^d} g(x)\, \mu_0(dx) &= \int_{\R^d} \int_{\sR_1} f(r)\, \Lambda_0(x, dr)\,\mu_0(dx)\\
	&= \int_{\R^d \times \sR_1} f(r) \,\Theta_0 \circ \theta^{-1}(dx, dr) = \int_{\sR_1} f(r)\, (\Theta_0)_{(2)}(dr) <\infty.
 \end{align*}
 Thus, from Condition \ref{sigmacondition}(ii),
 \begin{equation}\label{eq:condneeded}\lim_{N \to \infty}E^N\left[ \frac{1}{N} \sum_{j=1}^N f\left(\rho_j\right) \right] = \lim_{N \to \infty} \frac{1}{N} \sum_{j=1}^N g\left(x_j^N\right) = \int_{\R^d} g(x)\, \mu_0(dx) = \int_{\sR_1} f(r)\, (\Theta_0)_{(2)}(dr),
 \end{equation}
 which proves \eqref{eq:eq630}.
Now, we have
\begin{equation}\label{eq:eq424}
\begin{aligned}
	\limsup_{N \to \infty} E^N \left[ \frac{1}{N}\sum_{j=1}^N \int_0^T \left| u_j(t) \right|^2\,dt \right] 
	&\le \limsup_{N \to \infty} E^N \left[ \frac{1}{N}\sum_{j=1}^N \int_{[0,T]\times \R^m}  \left| y \right|^2 \rho_j(dt, dy) \right] \\
	&= E_{\Theta_0} \left[ \int_{[0,T]\times \R^m} |y|^2\,\rho(dt,dy) \right] < \infty, 
\end{aligned}
\end{equation}
where the convergence on the second line follows from \eqref{eq:eq630} on observing that, since $\Theta_0 \in \sP_2(\sZ)$,
$$f(r) = \int_{[0,T]\times\R^m}  \left| y \right|^2 r(dt,dy), \qquad r \in \sR_1,$$ satisfies
$$\int_{\sR_1} f(r)\, (\Theta_0)_{(2)}(dr) = E_{\Theta_0} \left[ \int_{[0,T]\times\R^m} |y|^2\,\rho(dt,dy) \right] < \infty.$$
Next, for each $\varphi \in \sC_c^\infty(U\times \R^d, \R^d)$ define
\[
	\bar J^N(\varphi) \doteq \frac{1}{N} \sum_{j=1}^N \int_0^T \varphi\left(t, \bar X_j^N(t)\right) \circ d\bar X_j^N(t), \qquad N \in \N.
\]
From Lemma \ref{JbarReal}, the collection $\{\varphi \mapsto \bar J^N(\varphi)\}$ has a pathwise realization $\bar \sJ^N$
 in $\mathbf{H}^{-\mathbf{s}}$.
Using Lemma \ref{keytightness} and the moment bound in \eqref{eq:eq424}, we now see that
 $\{(\bar \mu^N, Q^N, \bar \sJ^N), N \in \N\}$ is tight in $\sP_1(\sX) \times \sP(\sZ) \times \mathbf{H}^{-\mathbf{s}}$. 
Suppose, without loss of generality, that $(\bar \mu^N, Q^N,\bar\sJ^N) \Rightarrow (\bar \mu, Q, \bar\sJ)$ 
in $\sP_1(\sX)\times \sP(\sZ) \times \mathbf{H}^{-\mathbf{s}}$.
By Lemma \ref{keytightness} again,  $Q \in \sP^*(\bar\sJ)$ and $Q_{(1)} = \bar \mu$ a.s. 
Since $Q^N\circ \theta^{-1} = \Lambda^N$, \eqref{eq:PNLambdaN} implies that $Q\circ \theta^{-1} = \Theta_0 \circ \theta^{-1}$ a.s., 
and hence by the weak uniqueness established in Lemma \ref{lem:weakuniq}, $Q = \Theta_0$ a.s. 
Furthermore, from the definition of $\sP^*(\bar \sJ)$,   
\[
	\left\il\bar\sJ, \varphi \right\ir = G_\varphi(Q) = G_\varphi(\Theta_0) = \left \il\sJ_0, \varphi \right\ir
\]
for every $\varphi$, a.s., and hence $\bar\sJ = \sJ_0$ a.s. by separability of $\sC_c^\infty(U\times \R^d, \R^d)$  and its denseness in $\mathbf{H}^{-\mathbf{s}}$.

It follows that $(Q^N, \bar\sJ^N) \Rightarrow (\Theta_0, \sJ_0)$. Finally, 
\begin{align*}
	&\limsup_{N\to\infty} E^N \left[ \frac{1}{2N} \sum_{j=1}^N \int_0^T \left|u_j(t)\right|^2\,dt + F\left(\bar \mu^N, \bar \sJ^N\right) \right] \\
	&= \limsup_{N \to \infty} E^N \left[ \frac{1}{2N} \sum_{j=1}^N \int_0^T \left|u_j(t)\right|^2\,dt 
	+ F\left( Q^N_{(1)}, \bar\sJ^N \right) \right] \\
	&\le E_{\Theta_0}\left[ \frac{1}{2} \int_{[0,T]\times\R^m} |y|^2\,\rho(dt,dy) \right] + F\left((\Theta_0)_{(1)}, \sJ_0 \right), 
\end{align*}
where the last inequality is from \eqref{eq:eq424} and since $F$ is a bounded continuous function. This shows \eqref{eq:specialuN} and completes the proof of the 
lower bound in \eqref{eq:Lower},  and part (ii) of Theorem \ref{main} follows. \hfill\qed

\subsection{Rate Function Property}\label{section:ratefunction}
In this section we show that the function $I: \sP_1(\sX) \times \mathbf{H}^{-\mathbf{s}} \to [0,\infty]$ defined in \eqref{eq:RateFunction} has compact sublevel sets for every 
$\mathbf{s} \in \sO_d$.
Fix $\mathbf{s}$, and for each $l < \infty$ consider the level set 
$\Gamma_l \doteq \{(\mu, \sJ) \in \sP_1(\sX) \times \mathbf{H}^{-\mathbf{s}} : I(\mu, \sJ) \le l \}$. 
The proof of the following lemma is given in Section \ref{sec:pflem908}. 
\begin{lemma}
	\label{lem:lem908}
	Suppose  Conditions \ref{maincondition} and \ref{initialcondition} hold.
Let $\mathbf{s} \in \sO_d$ and let
$\{(\mu_k,  \Theta_k, \sJ_k), k \in \N\}$ be a sequence in 	$\sP_1(\sX)\times \sP(\sZ) \times \mathbf{H}^{-\mathbf{s}}$ such that for each $k$, $\Theta_k \in \sP^*(\sJ_k)$,  $(\Theta_k)_{(1)} = \mu_k$, and
\begin{equation}\label{eq:eq100}
	\sup_{k\ge 1} E_{\Theta_k} \left[ \frac{1}{2} \int_{[0,T]\times\R^m} |y|^2\,\rho(dt,dy) \right] <\infty.\end{equation}
Then the sequence $\{(\mu_k,  \Theta_k, \sJ_k), k \in \N\}$ is relatively compact in $\sP_1(\sX)\times \sP(\sZ) \times \mathbf{H}^{-\mathbf{s}}$.
\end{lemma}

Now we prove the compactness of $\Gamma_l$. Let $\{(\mu_k, \sJ_k), k \in \N\}$ be a sequence in $\Gamma_l$. From the definition of $I$, for each $k \in \N$ there is a 
$\Theta_k \in \sP^*(\sJ_k)$ with $(\Theta_k)_{(1)} = \mu_k$ such that 
\begin{equation}\label{eq:Thetak}
	E_{\Theta_k} \left[ \frac{1}{2} \int_{[0,T]\times\R^m} |y|^2\,\rho(dt,dy) \right] \le l + \frac{1}{k}. 
\end{equation}
From Lemma \ref{lem:lem908},
$\{(\mu_k, \Theta_k, \sJ_k)\}$ is relatively compact in $\sP_1(\sX)\times \sP(\sZ) \times \mathbf{H}^{-\mathbf{s}}$.
It is easily checked that if $(\mu, \Theta, \sJ)$ is a limit point along some subsequence, then $\Theta_{(1)} = \mu$ and along the same subsequence
$G_{\varphi}(\Theta_k) \to G_{\varphi}(\Theta)$  and $\il\sJ_k, \varphi\ir \to \il\sJ, \varphi\ir$
for every $\varphi \in \sC_c^\infty(U \times \R^d, \R^d)$.
This shows that
$\Theta \in \sP^*(\sJ)$.  
Sending $k \to \infty$ in \eqref{eq:Thetak} and using lower semicontinuity
of the map $r \mapsto \int_{[0,T]\times\R^m}  |y|^2\,r(dt,dy)$ on $\sR_1$, we obtain 
\[
	E_{\Theta} \left[ \frac{1}{2} \int_{[0,T]\times\R^m} |y|^2\,\rho(dt,dy) \right] \le l, 
\]
and hence $(\mu, \sJ)$  lies in $\Gamma_l$.  Compactness of $\Gamma_l$ follows. \hfill \qed

\subsection{Law of Large Numbers}\label{section:LLN}
Here we prove Theorem \ref{LLN}. 
The model \eqref{eq:model} can be viewed as the controlled equation \eqref{eq:controlledmodel} with the controls 
taken to be $u_j^N \equiv 0$ for all $1\le j \le N$ and $N \in \N$. 
From Lemma \ref{keytightness} it then follows that $(\mu^N, Q^N, \sJ^N)$ is tight in $\sP_1(\sX)\times\sP(\sZ) \times \mathbf{H}^{-\mathbf{s}}$. 
Suppose that along some subsequence $(\mu^N, Q^N, \sJ^N) \Rightarrow (\mu, Q, \sJ)$. Then, once again from Lemma 
\ref{keytightness}, $Q_{(1)} =  \mu$ and 
$Q \in \sP^*(\sJ)$ a.s.
Furthermore, since $u_j^N \equiv 0$ for all $1\le j \le N$ and $N \in \N$ we see that the second coordinate variable on $\sZ$ satisfies $Q(\rho=0)=1$ a.s., and thus, under $Q$, the first coordinate variable on $\sZ$ satisfies
\begin{equation*}
	X(t) =  X(0) + \int_0^t b(X(s),  V(s))\, ds, \qquad  V(t) = Q\circ X(t)^{-1}, \qquad V(0) = \mu_0, 
\end{equation*}
for all $0\le t \le T$. Then, from the unique solvability of \eqref{eq:deteq}, it follows that  $\mu = \mu^*$ a.s., and hence we have that $\mu^N$ converges in probability in $\sP_1(\sX)$
(along the full sequence) to $\mu^*$. Since $V^N(t) = \mu^N \circ \pi_t^{-1}$ and $V^*(t) = \mu^*\circ \pi_t^{-1}$ for each $0\le t\le T$, we also have that
$V^N \to V^*$ in probability in $\sV$. 
Finally, since $Q \in \sP^*(\sJ)$ a.s., 
$$G_\varphi(Q) = \il\sJ, \varphi \ir$$
for all $\varphi \in \sC_c^\infty(U \times \R^d, \R^d)$, a.s., and
note that
\begin{align*}
	G_\varphi(Q) &= E_{Q}\left[ \int_0^T \varphi\left(t, X(t) \right)\cdot dX(t) \right] \\
	&= E_{Q}\left[ \int_0^T \varphi\left(t, X(t) \right)\cdot b(X(t), V^*(t))\,dt \right] \\
	&= \int_0^T \left\il V^*(t), \varphi(t, \cdot ) \cdot b\left(\cdot, V^*(t)\right) \right\ir \, dt.
\end{align*}
Thus $\il \sJ, \varphi \ir$ is (a.s.) uniquely characterized for all $\varphi \in \sC_c^\infty(U \times \R^d, \R^d)$. From the separability
of $\sC_c^\infty(U\times \R^d, \R^d)$  and its denseness in $\mathbf{H}^{-\mathbf{s}}$ we now see that
$\sJ^N$ converges (along the full sequence) in probability, in $\mathbf{H}^{-\mathbf{s}}$, to the nonrandom limit
$\sJ^*$ characterized as
$$\left\il \sJ^*, \varphi \right\ir = \int_0^T \left\il V^*(t), \varphi(t, \cdot )\cdot b\left(\cdot, V^*(t)\right) \right\ir\, dt . $$
The result follows.
\hfill\qed

\subsection{Equivalent Formulation of the Rate Function}
\label{sec:ieqtili}
In this section we give the proof of Proposition \ref{lem:sigeqi}. Let $m = d$, and suppose that for every $\mu \in \clp_1(\RR^d)$, $\sigma(\mu)$ is invertible. 
We first argue that $\tilde I_0  \le \tilde I$.  Fix $(V, \sJ) \in \sV  \times \mathbf{H}^{-\mathbf{s}}$ such that $\tilde I(V, \sJ) < \infty$.
Fix $\delta > 0$ and 
let $\Theta \in \sP^*(\sJ)$ with $\nu_{\Theta} =  V$ be $\delta$-optimal for $ \tilde I(V, \sJ)$, namely
\begin{equation}
	\label{eq:eq210r}
	E_\Theta \left[ \frac{1}{2} \int_{[0,T]\times\R^d} |y|^2\,\rho(dt,dy) \right] \le  \tilde I(V, \sJ) + \delta.\end{equation}
Disintegrate $\rho(dt,dy) = \rho_t(dy)\, dt$ and define
\begin{equation}
\label{eq:urho}
	v(t) \doteq \int_{\RR^d} y \, \rho_t(dy), \qquad \mbox{ a.e. }  t \in [0, T]. 
\end{equation}
Also let 
$\eta^{\Theta}_t \doteq \Theta \circ (X(t), \sigma(V(t))v(t) + b(X(t), V(t)))^{-1} \in \sP(\R^{2d})$. Then, since $\nu_{\Theta} =  V$,  $\eta^{\Theta}_t$ can be disintegrated as
$\eta^{\Theta}_t(dx, dy) = \hat \eta^{\Theta}_t(x,dy)\, V(t,dx)$ for some $\hat \eta^{\Theta}_t: \R^d \to \sP(\R^d)$. 
Define the function  $h$ on $[0,T]\times \R^d$ by 
\begin{align}\label{eq:eq1259}
	h(t,x) &\doteq \int_{\RR^d} y \,\hat \eta^{\Theta}_t(x, dy), 
\end{align}
and note that Condition \ref{maincondition} ensures that this is well-defined. 
Under $\Theta$, $V(0) = \mu_0$ and 
 \begin{equation}\label{eq:eq1256}
X(t) = X(0) + \int_0^t b(X(s), V(s))\,ds + \int_{[0,t]\times \R^d} \sigma(V(s)) y \,\rho_s(dy)\,ds, \qquad \mbox{a.s.}, 
\end{equation}
for each $t$, and so for $\varphi \in \sC_c^\infty((0,T)\times \R^d, \R)$, 
\begin{align*}
	0 &= \varphi(T, X(T)) - \varphi(0, X(0)) \\
	&= \int_0^T  \left( \frac{\partial}{\partial t}\varphi(t, X(t)) + \nabla \varphi(t, X(t)) \cdot  b(X(t), V(t)) + 
	\nabla \varphi(t, X(t)) \cdot \sigma(V(t)) v(t)\right)\,dt,
\end{align*}
where $v$ is as in \eqref{eq:urho}.
Taking expectations with respect to $\Theta$,  
\begin{align}
	0&=
	\int_0^T  \int_{\R^{2d}} \left[\frac{\partial}{\partial t}\varphi(t, x) +  \nabla \varphi(t, x) \cdot  y\right] \, \eta^{\Theta}_t(dx,dy)\,dt\nonumber \\
	%
	&= \int_0^T  \left\il  V(t), \frac{\partial}{\partial t}\varphi(t, \cdot) +   \nabla \varphi(t, \cdot) \int_{\R^{d}} y\, \hat \eta^{\Theta}_t(\cdot, dy)\right\ir  \,dt\label{eq:eq937}\\
	&= \int_0^T  \left\il  V(t), \frac{\partial}{\partial t}\varphi(t, \cdot) + \nabla \varphi(t, \cdot) \cdot  h(t, \cdot)  \right\ir  \,dt.\nonumber
\end{align}
Similarly, since $\il \sJ, \varphi \ir = G_\varphi(\Theta)$, it is seen that 
for $\varphi \in \sC_c^\infty((0,T)\times \R^d, \R^d)$, 
$$\il\sJ, \varphi\ir = \int_0^T \left\il  V(t), \varphi(t, \cdot)\cdot h(t,\cdot) \right\ir \,dt.$$
Since $V = \nu_{\Theta}$, we now see from the above two identities that
\begin{align*}	
\tilde I_0(V, \sJ) 
 &\le \frac{1}{2} \int_0^T  \left\il V(t), 
	\int_{\RR^d} \left| \sigma^{-1}(V(t))(y- b(\cdot, V(t))\right|^2\, \hat\eta^{\Theta}_t(\cdot, dy)
	 \right \ir   \,dt\\
&=\frac{1}{2} \int_0^T 
	\int_{\RR^{2d}} \left| \sigma^{-1}(V(t))(y- b(x, V(t)))\right|^2\, \eta^{\Theta}_t(dx, dy)  \, dt\\
&= \frac{1}{2} \int_0^T E_{\Theta}\left[ \left|v(t)\right|^2 \right]\, dt 
\le \frac{1}{2} E_{\Theta}\left[\int_{[0,T]\times \RR^d} |y|^2\, \rho(dt,dy)\right] \le \tilde I(V, \sJ)+ \delta,
\end{align*}
where the last inequality is from \eqref{eq:eq210r}.
 Since $\delta>0$ is arbitrary, the inequality
$\tilde I_0(V, \sJ) \le \tilde I(V, \sJ)$ follows.

We now prove the reverse inequality, namely $\tilde I(V, \sJ) \le \tilde I_0(V, \sJ)$. Once more fix $\delta > 0$ and $(V, \sJ) \in \sV \times \mathbf{H}^{-\mathbf{s}}$ such that $\tilde I_0(V, \sJ) < \infty$,
and let $\Theta \in \sS(\sZ) \cap \sP_2(\sZ)$ be $\delta$-optimal for $\tilde I_0(V, \sJ)$, namely
\begin{equation}
\frac{1}{2} \int_0^T  \left\il V(t), 
	\int_{\RR^d} \left| \sigma^{-1}(V(t))(y- b(\cdot, V(t))\right|^2\, \hat\eta^{\Theta}_t(\cdot, dy)
	 \right \ir   dt \le \tilde I_0(V, \sJ) + \delta, \label{eq:eq218r}
\end{equation}
$V= \nu_{\Theta}$, and  $(V, \sJ)$ solves
\eqref{eq:distlimit} with $h(t,x) = \int_{\RR^d} y\, \hat \eta^{\Theta}_t(x, dy)$.
In particular, for all $\varphi \in \sC_c^\infty((0,T)\times \R^d, \R^d)$, \eqref{eq:eq602} holds.
Now define an $\sR_1$-valued random variable $\tilde \rho$ on $(\sZ, \sB(\sZ))$ as
$$\tilde \rho(dt, dy) = \delta_{v(t)}(dy)\, dt,$$
where $v$ is defined in terms of the coordinate variable $\rho$ as in \eqref{eq:urho}.
Defining $\tilde \Theta \in \sP(\sZ)$ as
$\tilde \Theta \doteq \Theta \circ (X, \tilde \rho)^{-1}$, we have that $\nu_{\tilde \Theta} = \nu_\Theta = V$, and it can be seen from \eqref{eq:eq1256} that
$\tilde \Theta \in \sS(\sZ)$.
Also, since  \eqref{eq:eq602} holds for any $\varphi \in \sC_c^\infty((0,T)\times \R^d, \R^d)$, 
\begin{align*}
\il\sJ, \varphi\ir &= \int_0^T \left\il  V(t), \varphi(t, \cdot)\cdot h(\cdot, t) \right\ir \,dt\\
&= \int_0^T \left\il  V(t), \varphi(t, \cdot)\cdot \int_{\RR^d} y\, \hat \eta^{\Theta}_t(\cdot, dy)\right\ir \,dt
= \int_0^T\int_{\RR^{2d}} \varphi(t, x)y\, \eta^{\Theta}_t(dx, dy) \,dt\\
&= E_{\Theta}\left[ \int_{[0,T]\times \RR^d} \varphi(t, X(t)) [\sigma(V(t))y + b(X(t), V(t))]\,\rho(dt,dy)\right]\\
&= E_{\tilde\Theta} \left[\int_{[0,T]\times \RR^d} \varphi(t, X(t)) [\sigma(V(t))y + b(X(t), V(t))]\,\rho(dt,dy)\right] = G_{\varphi}(\tilde \Theta),
\end{align*}
where the last line uses the fact that $\int y\, \rho_t(dy) = v(t) =  \int y\, \delta_{v(t)}(dy) = \int y \,\tilde \rho_t(dy)$.
Thus, $\tilde \Theta \in \sP^*(\sJ)$.
Finally,
\begin{align*}
	\tilde I(V, \sJ) &\le E_{\tilde \Theta} \left[ \frac{1}{2} \int_{[0,T]\times \R^d} |y|^2\,\rho(dt,dy) \right] = E_{\Theta} \left[ \frac{1}{2} \int_0^T |v(t)|^2\, dt \right] \\
	&=  E_{\Theta} \left[ \frac{1}{2} \int_0^T \left|\sigma^{-1}(V(t)) [ \sigma(V(t)) v(t) + b(X(t), V(t)) - b(X(t), V(t))]\right|^2\, dt\right]\\
	&= \frac{1}{2} \int_0^T \int_{\RR^{2d}} \left|\sigma^{-1}(V(t)) [ y - b(x, V(t))]\right|^2\, \eta_t(dx,dy)\,dt\\
	&=  \frac{1}{2} \int_0^T  \left\il V(t),
	\int_{\RR^d} \left| \sigma^{-1}(V(t))(y- b(\cdot, V(t))\right|^2\, \hat\eta^{\Theta}_t(\cdot, dy)
	 \right \ir   dt
	  \le \tilde I_0(V, \sJ) + \delta,
\end{align*}
where we used \eqref{eq:eq218r}.
Since $\delta>0$ is arbitrary, the inequality $\tilde I(V, \sJ) \le \tilde I_0(V, \sJ)$ follows and completes the proof of the lemma. \hfill \qed

\section{Proofs of Key Lemmas}\label{sec:keylemmas}
In this section we provide proofs of the results used in showing the Laplace upper and lower bounds. First we establish two estimates that will be used in subsequent sections. 

\begin{lemma}\label{L2bound} 
Suppose Conditions \ref{maincondition} and \ref{initialcondition} are satisfied. Let $u^N = (u_1^N, \ldots, u_N^N) \in \sA_N$ and let $\bar X^N$ be as defined in  \eqref{eq:controlledmodel}.
	Then, for each $N \in \N$,
\begin{equation}\label{eq:firstestimate}
	\frac{1}{N}\sum_{j=1}^NE\left[\left\|\bar X_j^N\right\|_\infty^2 \right] \le c\left( 1 + \frac{1}{N}\sum_{j=1}^N\left|x_j^N\right|^2 + E\left[ \frac{1}{N}\sum_{j=1}^N\int_0^T \left|u_j^N(t)\right|^2\,dt \right] \right), 
\end{equation}
and for any $\eps > 0$ and any $\{\sF(t)\}$-stopping time $\tau$ taking values in $[0, T-\eps]$, 
\[
	\frac{1}{N}\sum_{j=1}^N E\left[ \left| \bar X_j^N(\tau + \eps) - \bar X_j^N(\tau) \right|^2\right] \le c\eps \left( 1 + \frac{1}{N}\sum_{j=1}^N \left|x_j^N\right|^2 + E\left[ \frac{1}{N}\sum_{j=1}^N \int_0^T \left|u_j^N(t)\right|^2\,dt \right] \right). 
\]
where $c<\infty$ does not depend on $N$, $u^N$, or $\eps$. 
\end{lemma}

\begin{proof}
Condition \ref{maincondition} (see \eqref{eq:eq208}) implies 
\[
	\left| b\left(\bar X_j^N(t), \bar V^N(t) \right) \right|^2 \le 3L^2\left(1 + \left|\bar X_j^N(t) \right|^2 + \frac{1}{N}\sum_{j=1}^N \left|\bar X_j^N(t) \right|^2 \right), 
\]
and so from \eqref{eq:controlledmodel} and since $|\sigma| \le L$ and $\eps_N \le 1$, we have 
\begin{align*}
	\left| \bar X_j^N(t) \right|^2 
	&\le 4 \left| x_j^N \right|^2 +4 \left| \int_0^t b\left(\bar X_j^N (s), \bar V^N(s)\right)\,ds \right|^2 + 4\left| \eps_N \int_0^t \sigma\left(\bar X_j^N(s), \bar V^N(s) \right)\,dW_j(s) \right|^2  \\
	&\qquad  + 4\left| \int_0^t \sigma \left( \bar X_j^N(s), \bar V^N(s) \right) u_j^N(s)\,ds \right|^2 \\
	&\le 4 \left| x_j^N \right|^2 + 12 L^2T \left( 1 + \int_0^t \sup_{0\le r \le s}\left| \bar X_j^N(r) \right|^2\,ds + \frac{1}{N}\sum_{j=1}^N \int_0^t \sup_{0\le r \le s}\left| \bar X_j^N(r) \right|^2\,ds \right) \\
	&\qquad + 4\sup_{0\le r \le t} \left| \int_0^r \sigma\left(\bar X_j^N(s), \bar V^N(s) \right)\,dW_j(s) \right|^2 + 4L^2T \int_0^T \left| u_j^N(s) \right|^2\,ds. 
\end{align*}
Hence by The Burkholder-Davis-Gundy inequality, and using boundedness of $\sigma$ once more, 
\begin{align*}
	\frac{1}{N}\sum_{j=1}^N E\left[ \sup_{0\le s \le t} \left| \bar X_j^N(s) \right|^2 \right] &\le \frac{4}{N}\sum_{j=1}^N \left| x_j^N \right|^2 + 24 L^2T \left( 1 + \int_0^t\frac{1}{N}\sum_{j=1}^N E\left[\sup_{0\le r \le s}\left| \bar X_j^N(r) \right|^2\right]\,ds \right)  \\
	&\qquad + 16L^2T + 4L^2T E\left[\frac{1}{N}\sum_{j=1}^N\int_0^T \left|u_j^N(s) \right|^2\,ds \right]. 
\end{align*}
The first statement in the lemma then follows by Gronwall's inequality (see \cite[Theorem A.5.1]{EthKur} ) with $c = 24(L^2T+1)e^{24 L^2T^2}$.

Next, for any $t \in [0, T-\eps]$, the linear growth of $b$, boundedness of $\sigma$, and the Cauchy-Schwarz inequality give 
\begin{align*}
	\left| \bar X_j^N(t + \eps) - \bar X_j^N(t) \right|^2 &\le 4\left| \int_t^{t+\eps} b\left(\bar X_j^N(s), \bar V^N(s) \right)\,ds \right|^2 + 4\left| \eps_N \int_t^{t+\eps} \sigma \left( \bar X_j^N(s), \bar V^N(s)\right)\,dW_j(s) \right|^2 \\
	&\qquad + 4\left| \int_t^{t+\eps} \sigma \left( \bar X_j^N(s), \bar V^N(s) \right) u_j^N(s)\,ds \right|^2 \\
	&\le 12 TL^2 \eps \left( 1 + \sup_{0\le s \le T} \left| \bar X_j^N(s) \right|^2 + \frac{1}{N}\sum_{j=1}^N \sup_{0\le s \le T} \left| \bar X_j^N(s) \right|^2 \right)  \\
	&\qquad + 4  \left| \int_t^{t+\eps} \sigma \left( \bar X_j^N(s), \bar V^N(s)\right)\,dW_j(s) \right|^2 + 4L^2 \eps \int_0^T \left|u_j^N(s) \right|^2\,ds. 
\end{align*}
Since $\tau$ is a bounded stopping time, the optional sampling theorem gives
$$E\left| \int_{\tau}^{\tau+\eps} \sigma \left( \bar X_j^N(s), \bar V^N(s)\right)\,dW_j(s) \right|^2  \le L^2 \eps,$$
 and so 
\begin{align*}
	&\frac{1}{N}\sum_{j=1}^NE\left[ \left| \bar X_j^N(\tau + \eps) - \bar X_j^N(\tau) \right|^2 \right] \\
	&\le 24(T+1)L^2  \eps \left( 1 + E\left[\frac{1}{N}\sum_{j=1}^N \left\| \bar X_j^N \right\|_\infty^2 \right] + E\left[\frac{1}{N}\sum_{j=1}^N \int_0^T \left|u_j^N(s) \right|^2\,ds \right] \right) . 
\end{align*}
The second estimate in the lemma now follows (with a possibly larger choice of $c$).
\end{proof}

\subsection{Proof of Lemma \ref{keytightness}}
\label{sectiontightnessproof}

The following general lemma will be useful in proving the tightness of $\{\bar\sJ^N\}$. The proof is standard (see e.g. \cite[Exercise 3.11.18]{EthKur}) and is therefore omitted.
\begin{lemma}\label{tightnessdecomp} Let $\{Z_k, k \in \N\}$ be a sequence of random variables taking values in a separable Banach space with norm $\|\cdot \|$. Suppose that for each $\eps > 0$ we can write $Z_k = Z_k^\eps + R_k^\eps$ for each $k \in \N$, where $\{Z_k^\eps, k \in \N\}$ is tight and 
$\sup_{k \ge 1} E\left[ \left\|R_k^\eps\right\|  \right] \le \eps.$
Then $\{Z_k\}$ is tight. 
\end{lemma}

To prove tightness for the controlled stochastic currents, we will make use of a collection of test functions $\{g_M, M < \infty\}$ defined as follows. 

\begin{definition}\label{mollifiers} Let $\{g_M, M <\infty\}$ be a collection of functions in $\sC_c^\infty(\R^d, \R)$ that satisfy $0\le g_M(x) \le 1$ for all $M<\infty$ and $x \in \R^d$, and have the following properties
\begin{enumerate}[(i)]
\item For each $M$, $g_M(x) = 1$ on $|x| \le M$,

\item For each $M$, $g_M(x) = 0$ on $|x| \ge M+1$, and

\item For every $k \in \N$, there is a constant $B(k) < \infty$ such that $\left|D^\alpha g_M(x) \right| \le B(k)$ for all $x \in \R^d$, all $M < \infty$, and all $|\alpha| \le k$.
\end{enumerate}
\end{definition}
Note that if $\{g_M, M <\infty\}$ is a collection as in Definition \ref{mollifiers} then
for every $k \in \N$, there is a constant $L(k) < \infty$ such that 
\begin{equation}\label{eq:eq928}
	\left|D^\alpha g_M(x) - D^\alpha g_M(y) \right| \le L(k) |x-y|
\end{equation}
for all $x, y \in \R^d$, all $M <\infty$, and all $|\alpha| \le k$. 
We will need the following property of the collection $\{g_M, M <\infty\}$. Proof of the lemma is given in the Appendix. 
\begin{lemma}\label{gMnormbound} For any $s > 0$, there is a constant $K = K(s)<\infty$ such that for any $f \in H^s(\R^d, \R^d)$, 
\[
	\sup_{M<\infty} \left\|g_M f \right\|_{s} \le K \|f\|_{s}. 
\]
\end{lemma}

The following is a simple extension of the well-known compact embedding result for Sobolev spaces on $\R^d$ known as Rellich's Theorem (see \cite[Theorem 9.22]{folland}). 
Although the proof is standard, we provide details in the Appendix.
For $\mathbf{s} \in \sO_d$, $F \in \mathbf{H}^{-\mathbf{s}}$, and open $U_0 \subset U$, we say $F=0$ on $U_0$ if for all $\varphi \in \sC_c^\infty(U \times \R^d, \R^d)$
with support in $U_0$, $\il F, \varphi \ir=0$. The support of $F$ is the complement of the union of all open sets in $U$ on which $F=0$.

\begin{lemma}\label{compactembedding} Let $\mathbf{s} = (s_1, s_2)$ and $\mathbf{s}' = (s'_1, s'_2)$ in $\sO_d$ be such that $s'_1 < s_1$ and $s'_2 < s_2$.
	Suppose  $A\subset\mathbf{H}^{-\mathbf{s}'}$ is such that for some
	compact $K \subset U \times \R^d$, every $F\in A$ has support contained in $K$. Suppose also that
	$\sup_{F\in A}  \|F\|_{-\mathbf{s}'} < \infty$. Then $A$ is relatively compact in $\mathbf{H}^{-\mathbf{s}}$.
\end{lemma}

Finally, the lemma below establishes the required tightness for the controlled currents. 

\begin{lemma} \label{lem:tightjn}
	Suppose Conditions \ref{maincondition} and \ref{initialcondition} are satisfied.
	Let $\{g_M, M <\infty\}$ be the collection of functions in $\sC_c^\infty(\R^d, \R)$ as in Definition \ref{mollifiers}. For each $N \in \N$, $M<\infty$, and $\varphi \in \sC_c^\infty(U \times \R^d, \R^d)$, define 
\[
	\bar J^{N,M}(\varphi) \doteq \bar J^N(g_M\varphi), \qquad \bar J_c^{N,M}(\varphi) \doteq \bar J^N(\varphi) - \bar J^{N,M}(\varphi). 
\]
Then, the collections $\{\varphi \mapsto \bar J^{N,M}(\varphi)\}$ and 
$\{\varphi \mapsto \bar J^{N,M}_c(\varphi)\}$ 
have pathwise realizations $\bar\sJ^{N,M}, \bar\sJ_c^{N,M}$ in $\mathbf{H}^{-\mathbf{s}}$  for all $\mathbf{s} \in \sO_d$. Furthermore, if 
\[
	\sup_{N \ge 1} E\left[ \frac{1}{N}\sum_{j=1}^N \int_0^T \left| u_j^N(t) \right|^2\,dt \right] < \infty, 
\]
then for all $\mathbf{s} \in \sO_d$,
\[
	\sup_{M<\infty}\sup_{N\ge1} E \left[\left\| \bar\sJ^{N,M} \right\|_{-\mathbf{s}} \right] < \infty 
\]
and 
\[
	\lim_{M\to\infty} \sup_{N\ge1} E\left[\left\|\bar\sJ_c^{N,M} \right\|_{-\mathbf{s}}\right] = 0. 
\]
In particular, $\{\bar\sJ^N, N \in \N\}$ is tight in $\mathbf{H}^{-\mathbf{s}}$ for all $\mathbf{s} \in \sO_d$. 
\end{lemma}

\begin{proof} Fix $\mathbf{s} = (s_1, s_2) \in \sO_d$, and for each $N$ and $1\le j \le N$, let $C^N_{j,\mathbf{s}}$ be the square-integrable random variable from Lemma \ref{JbarReal}, so that $|\bar J^N(\varphi)| \le C^N_{\mathbf{s}}\|\varphi\|_{\mathbf{s}}$ a.s. for all $\varphi \in \sC_c^\infty(U \times \R^d, \R^d)$, where $C^N_{\mathbf{s}} = \frac{1}{N}\sum_{j=1}^N C^N_{j, \mathbf{s}}$. As a consequence of Lemma \ref{gMnormbound}, for some constant $K = K(s_2)<\infty$, we have, for all $\varphi \in \sC_c^\infty(U \times \R^d, \R^d)$ and $M<\infty$,
\begin{align}
	\left\|g_M \varphi \right\|^2_{\mathbf{s}} &= \int_{U} \|g_M\varphi(u, \cdot)\|^2_{s_2}\,du + \int_{U}\int_{U} \frac{\|g_M(\varphi(u, \cdot) - \varphi(v, \cdot))\|^2_{s_2}}{|u-v|^{1+2s_1}}\,du\,dv \le K^2 \|\varphi\|^2_{\mathbf{s}}. \label{eq:eq222r}
\end{align}
Hence,
\[
	\left|\bar J^{N,M}(\varphi) \right| \le C^N_{\mathbf{s}}\|g_M \varphi\|_{\mathbf{s}} \le KC^N_{\mathbf{s}} \|\varphi\|_{\mathbf{s}} \qquad \mbox{a.s.},  
\]
and 
\[
	\left| \bar J_c^{N,M}(\varphi) \right| = \left| \bar J^N((1-g_M)\varphi)\right| \le C^N_{\mathbf{s}} \|(1-g_M)\varphi\|_{\mathbf{s}} \le (1+K)C^N_{\mathbf{s}} \|\varphi\|_{\mathbf{s}} \qquad\mbox{a.s.}
\]
From \cite[Lemma 5]{flagubgiator} it then follows that, for every $M<\infty$, there are 
$\mathbf{H}^{-\mathbf{s}}$-valued random variables  $\bar \sJ^{N,M}$ and $\bar \sJ_c^{N,M}$
such that, for every $\varphi \in \sC_c^\infty(U \times \R^d, \R^d)$ and $M<\infty$,
$$\left\il \bar \sJ^{N,M}(\omega), \varphi \right\ir = \left[J^{N,M}(\varphi)\right](\omega)\qquad \mbox{and}\qquad \left\il \bar \sJ_c^{N,M}(\omega), \varphi \right\ir = \left[J_c^{N,M}(\varphi)\right](\omega), \qquad \mbox{a.e. } \omega \in \Omega.$$
Then, from Lemma \ref{JbarReal}, 
\begin{equation}\label{eq:tightnesscond1}
	\sup_{M<\infty}\sup_{N\ge 1} E \left[\left\|\bar \sJ^{N,M} \right\|_{-\mathbf{s}}^2\right]  \le K^2\sup_{N \ge 1} E\left[ \left(C^N_{\mathbf{s}} \right)^2 \right] < \infty. 
\end{equation}
Let $\bar J^N_j$ be as in \eqref{eq:Jbarj} and define the stopping times $\tau_j^{N,M} = \inf\{t > 0 : |\bar X_j^N(t)| \ge M\}$. Then, 
\[
	\bar J_c^{N,M}(\varphi) = \frac{1}{N}\sum_{j=1}^N \bar J_j^N((1-g_M)\varphi) = \frac{1}{N}\sum_{j=1}^N 1_{\left\{ \tau_j^{N,M} < T \right\}} \bar J^N_j((1-g_M)\varphi), 
\]
and by Lemma \ref{JbarReal}, 
\[
	\left| \bar J_j^N((1-g_M)\varphi) \right| \le C^N_{j, \mathbf{s}} \|(1-g_M)\varphi\|_{\mathbf{s}} \le (1+K) C^N_{j, \mathbf{s}} \|\varphi\|_{\mathbf{s}}. 
\]
Thus, 
\begin{equation}
\label{eq:eq307}	
\left|\bar J_c^{N,M}(\varphi) \right| \le \left(\frac{1+K}{N} \sum_{j=1}^N 1_{\left\{ \tau_j^{N,M} < T \right\}}  C^N_{j, \mathbf{s}} \right) \|\varphi\|_{\mathbf{s}} \doteq \tilde C^N_{\mathbf{s}}\|\varphi\|_{\mathbf{s}}.
\end{equation}
Also, by the Cauchy-Schwarz inequality,
\begin{equation*}
E\left[\left(\tilde C^N_{\mathbf{s}}\right)^2\right] \le \frac{(1+K)^2}{N}\left(\sum_{j=1}^N  P\left( \tau_j^{N,M} < T \right) \right)\left(
\frac{1}{N}\sum_{j=1}^N E\left[\left(C^N_{j, \mathbf{s}} \right)^2\right] \right).
\end{equation*}
By Lemma \ref{L2bound}, Condition \ref{initialcondition}, and the assumption that $\sup_{N\in \N} E\left[\frac{1}{N}\sum_{j=1}^N \int_0^T \left|u_j^N(t)\right|^2\,dt\right] < \infty$, there is a constant $\tilde K < \infty$ such that 
\[
	\sup_{N\ge1} \frac{1}{N}\sum_{j=1}^N P\left( \tau_j^{N,M} < T \right) \le \sup_{N\ge1}\frac{1}{N}\sum_{j=1}^N P\left( \left\| \bar X_j^N\right\|_\infty \ge M \right) \le \frac{\tilde K}{M^2}.
\]
Thus, 
\begin{align*}
	E\left[\left\| \bar \sJ_c^{N,M} \right\|_{-\mathbf{s}}^2\right] \le  E\left[\left(\tilde C^N_{\mathbf{s}}\right)^2\right] \le
	\frac{\tilde K (1+K)^2}{M^2} \sup_{N \ge 1}\frac{1}{N}\sum_{j=1}^N 
		E\left[\left(C^N_{j, \mathbf{s}}\right)^2\right],
\end{align*}
and therefore, from Lemma \ref{JbarReal},
\begin{equation}\label{eq:eq324}\lim_{M\to \infty} \sup_{N\ge 1} E\left[\left\| \bar \sJ_c^{N,M} \right\|_{-\mathbf{s}}^2\right]  =0.\end{equation}

Note that \eqref{eq:tightnesscond1} and \eqref{eq:eq324} are satisfied for every $\mathbf{s} \in \sO_d$. Now for an arbitrary $\mathbf{s}\in \sO_d$, choose $\mathbf{s}' = (s'_1, s'_2)\in \sO_d$ such that
$s'_1< s_1$ and $s'_2 < s_2$. Then applying \eqref{eq:tightnesscond1} for $\mathbf{s}'$ and observing that
$\{\bar \sJ^{N,M}, N \in \N\}$ are compactly supported on $[0,T] \times \{|x| \le M+1\} \subset U \times \R^d$, we see from Lemma \ref {compactembedding} and Markov's inequality that for each fixed $M$, $\{\bar \sJ^{N,M}, N \in \N\}$ is a tight collection of
 $\mathbf{H}^{-\mathbf{s}}$-valued random variables.
Finally, observing that $\bar \sJ^N = \bar \sJ^{N,M} + \bar \sJ_c^{N,M}$ for each $M$
and applying \eqref{eq:eq324} and Lemma \ref{tightnessdecomp}, we obtain that
 $\{\bar \sJ^N, N \in \N\}$ is tight in $\mathbf{H}^{-\mathbf{s}}$. 
\end{proof}

The following general lemma will be useful in proving tightness of  $\{\bar \mu^N\}$.

\begin{lemma}\label{P1tightness}
	Let $(S, d_S)$ be a Polish space. 
	If $\{\gamma_k, k \in \N\}$ is a tight sequence of $\clp(S)$-valued random variables and
	for some $x_0 \in S$ 
	\begin{equation}\label{eq:eq225r}
		\sup_{k\in \N} E\left[\int_S d_S(x,x_0)^2\,\gamma_k(dx) \right] < \infty, 	
	\end{equation}
	then $\{\gamma_k\}$ is tight as a sequence of $\clp_1(S)$-valued random variables. 
\end{lemma}
\begin{proof}
	Suppose that $\gamma_k$ converges in distribution, along a subsequence, in $\clp(S)$ to some $\gamma$, and denote the convergent subsequence once more as $\{\gamma_k\}$. From \eqref{eq:eq225r} it follows
	that each $\gamma_k$ is in $\sP_1(S)$ a.s.
	Furthermore, by lower semicontinuity of the map $\mu \mapsto \int_{S}  d_S(x,x_0)^2\, \mu (dx)$ on $\sP(S)$ and Fatou's lemma, we see that
	$$E\left[\int_{S} d_S(x,x_0)^2 \, \gamma(dx) \right] \le E\left[ \liminf_{k \to \infty} \int_S d_S(x,x_0)^2\,\gamma_k(dx) \right] \le \sup_{k\ge 1} E\left[\int_{S} d_S(x,x_0)^2 \, \gamma_k(dx) \right]  < \infty, $$ 
	and  so in particular  $\gamma \in \clp_1(S)$ a.s.
	By appealing to Skorohod's representation theorem
	we can assume that $\gamma_k \to \gamma$ a.s. in $\clp(S)$. Recalling from Section \ref{sec:notat} 
	the metric $\dbl$ on the space $\sP(S)$, we have that $\dbl(\gamma_k, \gamma) \to 0$ a.s. 
 	
	It suffices now to show that $\gamma_k$ converges in probability in $\clp_1(S)$ to $\gamma$. 
	Take $f \in \sL(S)$ such that $f(x_0)=0$. 
	Fix $1<M<\infty$ and define 
	\[
		f_M(x) \doteq \left(\frac{f(x)}{M} \vee (-1) \right) \wedge 1, 
	\]
	which is a function bounded by $1$ in absolute value whose Lipschitz constant is also bounded by $1$.
	Then,
	\begin{align*}
		&\left|\int_{S} f(x)\,\gamma_k(dx)- \int_{S} f(x)\,\gamma(dx)\right| \\
		&\le
		M\left|\int_{S} f_M(x)\,\gamma_k(dx)- \int_{S} f_M(x)\,\gamma(dx)\right| 
		+\int_{S} \left|Mf_M(x)-f(x)\right|\,\gamma_k(dx) + \int_{S} \left|Mf_M(x)-f(x)\right|\gamma(dx)  \\
		&\le M\dbl(\gamma_k, \gamma) + \int_{S}2 |f(x)|1_{\left\{|f(x)| > M\right\}}\,\gamma_k(dx) + \int_{S} 2|f(x)|1_{\left\{|f(x)| > M\right\}}\,\gamma(dx). 
		\end{align*}
Since the Lipschitz constant of $f$ is bounded by $1$ and $f(x_0)=0$, we have that
$|f(x)| \le d_S(x,x_0)$, and so
$$ \int_{S} |f(x)|1_{\left\{|f(x)| > M\right\}}\,\gamma_k(dx) \le  \frac{1}{M}\int_{S}  d_S(x,x_0)^2\,\gamma_k(dx),
$$
and the equivalent inequality holds for $\gamma$.
Now, since $\il  \mu, f \ir - \il \nu, f \ir = \il \mu,  f - f(x_0) \ir - \il \nu,  f - f(x_0) \ir$ for any $\mu, \nu \in \sP_1(S)$ and $f \in \sL(S)$, the supremum in the definition of $d_1$ can be restricted to $f$ such that $f(x_0) = 0$. Thus, 
\begin{align*}
	E\left[d_1(\gamma_k, \gamma)\right] &= E\left[\sup_{f\in \sL(S), f(x_0) = 0} \left| \int_S f(x)\,\gamma_k(dx) -\int_S f(x)\,\gamma(dx) \right| \right] \\
	&  \le ME\left[\dbl(\gamma_k, \gamma) \right] + \frac{2}{M}\sup_{l \in \N} E\left[\int_{S} d_S(x,x_0)^2 \,\gamma_l(dx) \right]
	+ \frac{2}{M}E\left[\int_{S} d_S(x,x_0)^2 \, \gamma(dx) \right].
\end{align*}
Sending first $k \to \infty$ and then $M \to \infty$, we have that $\lim_{k\to\infty}E\left[d_1(\gamma_k, \gamma) \right] = 0$ which completes the proof.
\end{proof}

We can now complete the proof of Lemma \ref{keytightness}. \\

\subsubsection{Proof of Lemma \ref{keytightness}(i)}

We begin by arguing that $\{\bar \mu^N\}$ is a tight sequence of $\sP(\sX)$-valued random variables.
For this it suffices to show (see \cite[Theorem 2.11]{BudDup_book}) that
$\{\gamma^N, N\in \N\}$ is a relatively compact set in $\sP(\sX)$, where
$$\gamma^N \doteq E\left[\bar \mu^N\right] = \frac{1}{N} \sum_{j=1}^N P\left(\bar X^N_j \in \cdot\right).$$
Note  that 
$$\int_\sX \|\psi\|_\infty^2\,\gamma^N(d\psi) =
\frac{1}{N}\sum_{j=1}^N E\left[ \left\|\bar X_j^N\right\|_\infty^2 \right], 
$$
and so by Lemma \ref{L2bound} and the assumption on the controls in Lemma \ref{keytightness}, we see that
\begin{equation}\label{eq:eq545}
\sup_{N\ge 1} \int_\sX \|\psi\|_\infty^2\,\gamma^N(d\psi)  = \sup_{N\ge 1} \frac{1}{N}\sum_{j=1}^N E\left[ \left\|\bar X_j^N\right\|_\infty^2 \right] < \infty.
\end{equation}

Next, for  $\eps > 0$  let $\sT_\eps$ denote the collection of all $\{\sigma(X(s) : s\le t)\}$-stopping times on $(\sX, \sB(\sX))$ taking values in $[0, T - \eps]$ where
$\{X(t)\}$ is the canonical coordinate process on $\sX$. 
Then for each $N\in \N$, there are $\{\sigma(\bar X^N_j(s) : s\le t)\}$-stopping times $\{\tau_j^N, 1\le j \le N\}$ on $(\Omega, \sF)$ with values in $[0, T-\eps]$, such that 
\[ \int_{\sX} \left|\psi(\tau + \eps) - \psi(\tau)\right|^2\,\gamma^N(d\psi)  = \frac{1}{N}\sum_{j=1}^N E\left[ \left| \bar X_j^N\left(\tau_j^N + \eps\right) - \bar X_j^N\left(\tau_j^N\right) \right|^2\right]. 
\]
Applying Lemma \ref{L2bound}, we then have 
\[
	 \int_{\sX} \left|\psi(\tau + \eps) - \psi(\tau)\right|^2\,\gamma^N(d\psi)  \le c\eps\left( 1 + \sup_{N\ge1}\frac{1}{N}\sum_{j=1}^N \left|x_j^N\right|^2 + \sup_{N\ge1} E\left[ \frac{1}{N}\sum_{J=1}^N \int_0^T \left|u_j^N(t) \right|^2\,dt\right] \right),
\]
and hence  
\begin{equation}\label{eq:eq553}
	\lim_{\eps\to0} \sup_{N\ge1} \sup_{\tau \in \sT_\eps}  \int_{\sX} \left|\psi(\tau + \eps) - \psi(\tau)\right|^2\,\gamma^N(d\psi)  = 0. 
\end{equation}
The relative compactness of $\{\gamma^N, N\in \N\}$ in $\sP(\sX)$ is immediate from
\eqref{eq:eq545} and \eqref{eq:eq553} (see \cite[Theorem D.4]{BudDup_book}), which as noted previously shows $\{\bar \mu^N\}$ is a tight sequence of $\sP(\sX)$-valued random variables.
The tightness of $\{\bar \mu^N\}$ as a sequence of $\sP_1(\sX)$-valued random variables now follows from Lemma \ref{P1tightness} and the uniform moment estimate in \eqref{eq:eq545}.
Note also that since $\bar \mu^N = Q^N_{(1)}$, we have the tightness of the first marginals of $\{Q^N\}$ (as a sequence of $\sP(\sX)$-valued random variables).

That the second marginals $\{Q^N_{(2)}\}$ is a tight  sequence of $\sP(\sR_1)$-valued random variables follows by  an argument similar to \cite[Lemma 5.1]{BudDupFis} however we provide the details.
Note that the function 
\[
	h(r) = \int_{[0,T]\times\R^m}|y|^2\,r(dt, dy)
\]
has compact level sets on $\sR_1$ (recall that $\sR_1$ is equipped with the Wasserstein-1 metric). It then follows that 
\[
	H(\theta) = \int_{\sR_1} h(r)\,\theta(dr)
\]
has relatively compact level sets on $\sP(\sR_1)$ (see  \cite[Lemma 2.10]{BudDup_book}).  It now suffices to show $\sup_{N \ge 1} E[H(Q^N_{(2)})] < \infty$ (see \cite[Lemmas 2.9]{BudDup_book}). However this is immediate as
\begin{equation}\label{eq:supHQfinite}
	\sup_{N\ge1}E\left[ H\left(Q^N_{(2)}\right) \right] = \sup_{N\ge1}E\left[ \frac{1}{N}\sum_{j=1}^N \int_{[0,T]\times\R^m}|y|^2\rho_j^N(dt,dy) \right] 
	= \sup_{N\ge1} E\left[ \frac{1}{N}\sum_{j=1}^N \int_0^T \left|u_j^N(t)\right|^2\,dt \right] < \infty. 
\end{equation}
Thus we have shown that the second marginals of $\{Q^N\}$ are also tight, which in turn shows that
$\{\bar \mu^N, Q^N\}$ is a tight sequence of $\sP_1(\sX)\times \sP(\sZ)$-valued random variables. Together with Lemma \ref{lem:tightjn}, this finishes the proof of Lemma \ref{keytightness}(i). \hfill \qed

\subsubsection{Proof of Lemma \ref{keytightness}(ii)}
 Suppose now that $(\bar \mu^N, Q^N, \bar \sJ^N) \Rightarrow (\bar \mu, Q, \bar \sJ)$ in $\sP_1(\sX)\times\sP(\sZ)\times \mathbf{H}^{-\mathbf{s}}$, where $(\bar \mu, Q, \bar \sJ)$ is defined on some probability space. By appealing to Skorokhod's representation theorem, we can assume that $\{(\bar \mu^N, Q^N, \bar\sJ^N)\}$ and $(\bar\mu, Q, \bar\sJ)$ are defined on a common probability space $(\tilde \Omega, \tilde \sF,  \tilde P)$ and that $(\bar \mu^N, Q^N, \bar \sJ^N) \to (\bar\mu, Q,  \bar\sJ)$ a.s. Let $\tilde E$ denote expectation on this space.
The property $Q_{(1)} = \bar \mu$ is immediate from the identity $Q^N_{(1)} = \bar \mu^N$ for every $N \in \NN$. 
 We will complete the remainder of the proof in three steps: step 1 will establish that $Q \in \sP_2(\sZ)$, step 2 that $Q \in \sS(\sZ)$, and step 3 that $Q \in \sP^*(\bar \sJ)$, from which the result will follow. \\

{\it Step 1.} 
By Fatou's lemma, 
\begin{equation}\label{eq:QinP2}
\begin{aligned}
	 \tilde E\left[E_{Q}\left[ \int_{\R^m\times[0,T]}|y|^2 \rho(dy\,dt) \right] \right] &\le \liminf_{N \to \infty}  \tilde E\left[E_{Q^N}\left[ \int_{\R^m\times[0,T]}|y|^2 \rho(dy\,dt) \right] \right] \\
	&= \liminf_{N \to \infty} E\left[ \frac{1}{N}\sum_{j=1}^N \int_0^T \left|u_j^N(t)\right|^2\,dt \right] < \infty, 
\end{aligned}
\end{equation}
and hence $Q \in \sP_2(\sZ)$ a.s. \\

{\it Step 2.} We now show that a.s. $Q \in \sS(\sZ)$, namely it is  a weak solution to \eqref{eq:LimitContSDE}. 
Define the generator $\sA$ as follows. 
For each $f \in \sC_c^2(\R^d, \R)$, let 
\[
	\sA f(\nu, x, y) = \left( b(x, \nu) + \sigma(x, \nu)y \right) \cdot \nabla f(x), \qquad (\nu, x, y) \in \sP_1\left(\R^d\right)\times \R^d \times\R^m. 
\]
Now fix an $f \in \sC_c^2(\R^d, \R)$ and 
define, for each $V \in \sV$, the $\R$-valued process $\{M^{V}(t), 0\le t \le T\}$ on the measurable space $(\sZ, \sB(\sZ))$ by
\begin{equation}\label{eq:martingale}
	M^{V}(t, (\xi,r)) = f(\xi(t)) - f(\xi(0)) - \int_{[0,t] \times \R^m} \sA f\left(V(s), \xi(s), y\right)\,r(ds, dy), \qquad (\xi, r) \in \sZ. 
\end{equation}
Let $\bar V \doteq \nu_Q$. 
Since $f$ is arbitrary, to establish that $Q \in \sS(\sZ)$ a.s., it suffices to show that for each fixed $0\le t \le T$ and a.e. $\omega \in \tilde \Omega$,   
\begin{equation}
\label{eq:martisz}	
M^{\bar V(\omega)}(t, (\xi,r)) = 0, \qquad Q(\omega) \mbox{-a.e. } (\xi,r) \in \sZ.
\end{equation}
We will supress $\omega$ from the notation for the remainder of the proof.

For each $1\le B < \infty$, let $\psi_B \in \sC_c(\R^m, \R^m)$ be such that $\psi_B(y) = y$ on $\{|y| \le B\}$ and $|\psi_B(y)| \le |y| + 1$ everywhere. Note that since $B \ge 1$, this definition implies that 
\begin{equation}\label{eq:psiBproperty}
	\left|y - \psi_B(y) \right| \le \frac{|y|(2|y| + 1)}{B}1_{\{|y| > B\}} \le \frac{3|y|^2}{B}. 
\end{equation}
Also let $\eta_B \in \sC_c(\R^d, \R^d)$ be such that $\eta_B(x) = x$ on $\{|x| \le B\}$ and $|\eta_B(x)| \le |x| + 1$ everywhere.
As with $\psi_B$, we have that 
\begin{equation}\label{eq:psiBpropertyb}
	\left|x - \eta_B(x) \right| \le \frac{3|x|^2}{B}. 
\end{equation}
Now define the `truncated generator' $\sA_B$
\[
	\sA_{B} f(\nu, x,y) =  \left(\eta_B( b(x, \nu)) + \sigma(x, \nu)\psi_B(y)\right)\cdot \nabla f(x), \qquad (\nu, x,y) \in \sP_1\left(\R^d\right) \times \R^d\times \R^m, 
\]
and for each $V \in \sV$, let $\{M^{V}_{B}(t)\}$ be the corresponding process defined as in \eqref{eq:martingale} with $\sA_{B}$ in place of $\sA$. 
Let 
\[
	K \doteq  \sup_{x \in \R^d} \left(|f(x)| + |\nabla f(x)| + |D^2f(x)|\right) <\infty, 
\]
and note that for all $V \in \sV$,  $0\le s \le t$, and $(x, y) \in \R^d \times \R^m$, 
\begin{equation}\label{eq:eq644}
\begin{aligned}
	\left|\sA f(V(s),x,y)- \sA_{B} f(V(s),x,y)\right| &\le K\left( \frac{3 \left| b(x, V(s)) \right|^2}{B} + \frac{3L|y|^2}{B} \right)  \\
	&\le \frac{12K(L+1)^2}{B} \left( 1+|x|^2 + \int_{\R^d} \left|x'\right|^2\, V(s,dx') + |y|^2 \right). 
\end{aligned}
\end{equation}
Now fix $t$, and define the maps $\Phi$ and $\Phi_B$ on $\sP(\sZ)\times \sV$ by 
\[
	\Phi(\Theta, V) = E_\Theta\left[ \left| M^V(t)  \right| \right], \qquad \Phi_B(\Theta, V) = E_\Theta\left[\left| M^V_B(t)  \right| \right].
\]
Note that $\bar V^N = \nu_{Q^N}$, were $\bar V^N$ is as in Section \ref{sec:varrep}. We proceed by showing that 
\begin{enumerate}[(a)]
\item $\Phi_B$ is bounded and continuous on $\sP(\sZ)\times \sV$, 

\item $\sup_{N\ge1}  \tilde E\left[\left| \Phi(Q^N, \bar V^N) - \Phi_B(Q^N, \bar V^N)\right|\right] \to 0$ and $\left| \Phi(Q, \bar V) - \Phi_B(Q, \bar V)\right| \overset{\tilde P}{\to} 0$ as $B \to \infty$, and

\item $\Phi(Q^N, \bar V^N) \overset{\tilde P}{\to} 0$ as $N \to \infty$. 

\end{enumerate}
The convergence $(Q^N, \bar V^N) \to (Q, \bar V)$ then yields that $\Phi(Q, \bar V)=0$ a.s.,
from which the statement in \eqref{eq:martisz}  is immediate. 

We first show (a). Boundedness of $\Phi_B$ follows from the boundedness of $\eta_B$, $\psi_B$, $\sigma$, $f$, and $\nabla f$. The continuity of $\Phi_B$ follows from the continuity of the map
$(V, z) \mapsto M_B^V(t,z)$ on $\sV \times \sZ$. 

For (b), note from \eqref{eq:eq644} that 
\begin{align}
	&\tilde E \left[\left|\Phi\left(Q^N, \bar V^N \right) - \Phi_B\left(Q^N, \bar V^N\right)\right| \right] 
	\le \tilde E\left[ E_{Q^N} \left[ \left| M^{\bar V^N}(t)-M^{\bar V^N}_{B}(t)\right| \right] \right] \nonumber\\
	 &\le \frac{12K(L+1)^2}{B}\tilde E\left[E_{Q^N} \left[ \int_0^T \left(1+|X(s)|^2 + \int_{\R^d} |x|^2 \, \bar V^N(s,dx) + \int_{\R^m} |y|^2\, \rho_s(dy)\right)\, ds \right] \right]\nonumber\\
	&\le \frac{12K(L+1)^2}{B} \sup_{N\ge 1}  E\left[T +\frac{2T}{N}\sum_{j=1}^N
	\left\|\bar X_j^N\right\|_\infty^2 + \frac{1}{N}\sum_{j=1}^N\int_0^T \left|u_j^N(s)\right|^2\,ds  \right].
	\label{eq:eq237r}
\end{align}
From Lemma \ref{L2bound} and the assumption on the controls in Lemma \ref{keytightness}, we see that the last term in the above display converges to $0$ as $B \to \infty$. 
Similarly, since $Q \in \sP_2(\sZ)$ a.s.,  the estimate
\begin{align*}
&\left| \Phi(Q, \bar V) - \Phi_B(Q, \bar V)\right| 
\le E_Q \left[ \left| M^{\bar V}(t)-M^{\bar V}_{B}(t)\right| \right]\\
&\le \frac{12K(L+1)^2}{B} \left(\int_0^T  \left(1+ 2 \int_{\R^d} |x|^2 \, \bar V(s, dx)\right) ds + E_{Q}\left[\int_{[0,T]\times \R^m}|y|^2 \, \rho(ds, dy) \right]\right)
\end{align*}
implies that 
\begin{equation}
	\label{eq:244}
	\left| \Phi\left(Q, \bar V\right) - \Phi_B\left(Q, \bar V\right)\right|  \to 0 \qquad \mbox{a.s.,} \qquad \mbox{as}\; B\to \infty.
\end{equation}
This completes the proof of (b).

We now turn to (c). Note that 
\begin{align*}
	\Phi\left(Q^N, \bar V^N\right) &= E_{Q^N}\left[\left| M^{\bar V^N}(t)  \right| \right]
	=\frac{1}{N}\sum_{j=1}^N  \left|M^{\bar V^N}\left( t, \left( \bar X_j^N, \rho_j^N\right)\right)  \right| \\
	&= \frac{1}{N}\sum_{j=1}^N \left|f\left(\bar X_j^N(t)\right) - f\left(x_j^N\right) - \int_0^t  \sA f \left(\bar V^N(s), \bar X_j^N(s), u_j^N(s)\right)\,ds\right|. 
\end{align*}
By It\^o's lemma, for each $1\le j \le N$,
\begin{align*}
	f\left( \bar X_j^N(t)\right) - f\left(x_j^N\right) &= \int_{0}^{t} \sA f\left(\bar V^N, \bar X_j^N(s), u_j^N(s) \right)\,ds \\
	&\qquad + \eps_N \int_{0}^{t} \nabla f\left(\bar X_j^N(s) \right) \cdot \sigma\left(\bar X_j^N(s), \bar V^N(s) \right)\,dW_j(s) \\
	&\qquad + \frac{\eps_N^2}{2}  \int_{0}^{t} \mbox{Tr}\left[D^2 f\left(\bar X_j^N(s)\right) \left(\sigma \sigma^{\mathsf{T}}\right)\left(\bar X_j^N(s), \bar V^N(s)\right) \right]\,ds.
\end{align*}
Hence, 
\begin{align*}
	\Phi\left( Q^N, \bar V^N \right) &= \frac{1}{N}\sum_{j=1}^N \left|
	\eps_N \int_{0}^{t} \nabla f\left(\bar X_j^N(s) \right) \cdot \sigma\left(\bar X_j^N(s), \bar V^N(s) \right)\,dW_j(s)  \right. \\
	&\qquad\qquad\qquad + \left. \frac{\eps_N^2}{2}  \int_{0}^{t} \mbox{Tr}\left[D^2 f\left(\bar X_j^N(s)\right) \left(\sigma \sigma^{\mathsf{T}}\right)\left(\bar X_j^N(s), \bar V^N(s)\right) \right]\,ds \right|.
\end{align*}
From the boundedness of $\nabla f$, $D^2f$, and $\sigma$, it follows that
\[
	\tilde E \left[\Phi\left(Q^N, \bar V^N \right)\right] \le  \left(KLT^{1/2}\right)\eps_N + \frac{KL^2T\eps_N^2}{2} \to 0 \qquad \mbox{as}\; N \to \infty. 
\]
This completes (c), which as noted previously proves the statement in \eqref{eq:martisz} and which in turn shows that $Q$ is a.s. a weak solution to \eqref{eq:LimitContSDE}. \\

{\it Step 3.} To complete the proof of Lemma \ref{keytightness}, it only remains to establish that 
\begin{equation}\label{eq:eq242r}
	G_\varphi(Q) = \il \bar\sJ,\varphi\ir \qquad \mbox{for all}\qquad  \varphi \in \sC_c^\infty(U \times\R^d, \R^d), \qquad \tilde P\mbox{-a.s.}
\end{equation}
By considering a countable, dense subset of $\sC_c^\infty(U \times \R^d, \R^d)$, it suffices to show that for each fixed $\varphi \in \sC_c^\infty(U \times \R^d, \R^d)$, we have $G_\varphi(Q) = \bar\sJ(\varphi)$ a.s.

Fix $\varphi$, and let 
\[
	K_\varphi \doteq \sup_{(t,x)\in [0,T] \times \R^d} \left(\left| \varphi(t,x) \right| + \sum_{k,l=1}^d \left| \frac{\partial\varphi_k}{\partial x_l}(t,x) \right| \right) < \infty. 
\]
Then, a.s., 
\begin{align*}
	\left\il\bar \sJ^N, \varphi \right\ir &= \frac{1}{N}\sum_{j=1}^N \int_0^T \varphi\left(t, \bar X_j^N(t)\right)\circ d\bar X_j^N(t) \\
	&= \frac{1}{N}\sum_{j=1}^N \int_0^T \varphi\left(t, \bar X_j^N(t)\right) \cdot d\bar X_j^N(t) + \frac{1}{2N} \sum_{j=1}^N \left\il \varphi\left(\cdot, \bar X_j^N(\cdot)\right), \bar X_j^N(\cdot) \right\ir_T \\
	&= \frac{1}{N}\sum_{j=1}^N \int_0^T \varphi\left(t, \bar X_j^N(t)\right)\cdot d\bar X_j^N(t) \\
	&\qquad + \frac{\eps_N^2}{2N} \sum_{j=1}^N \int_0^T \sum_{k,l=1}^d \frac{\partial \varphi_k}{\partial x_l}\left(t, \bar X_j^N(t)\right) (\sigma\sigma^{\mathsf{T}})_{lk}\left(\bar X_j^N(t), \bar V^N(t)\right)\,dt.
\end{align*}
Define
\[
	G_\varphi^*\left(Q^N\right)  \doteq \frac{1}{N}\sum_{j=1}^N \int_0^T \varphi\left(t, \bar X_j^N(t) \right) \cdot d\bar X_j^N(t). 
\]
Since $|\sigma| \le L$,
\begin{align*}
	 \left| \left\il\bar \sJ^N,\varphi\right\ir - G_\varphi^*\left(Q^N\right) \right| &= \left| \frac{\eps_N^2}{2N} \sum_{j=1}^N \int_0^T \sum_{k,l=1}^d \frac{\partial \varphi_k}{\partial x_l}\left(t, \bar X_j^N(t)\right) (\sigma\sigma^{\mathsf{T}})_{lk}\left(\bar X_j^N(t), \bar V^N(t)\right)\,dt \right| \le \frac{K_\varphi L^2 T \eps_N^2}{2}, 
\end{align*}
and hence $| \il\bar \sJ^N, \varphi \ir - G_\varphi^*(Q^N) | \to 0$ in $L^1$ as $N \to \infty$. 
Also,  by the dominated convergence theorem,  
\[
	\lim_{N \to \infty} E\left[ \left|\left\il\bar \sJ, \varphi \right\ir- \left\il\bar\sJ^N,\varphi \right\ir \right|\wedge 1 \right] = 0. 
\]
Next, writing
\[
	\left| \left\il\bar\sJ, \varphi \right\ir - G_\varphi(Q) \right| \wedge 1 \le \left| \left\il\bar\sJ, \varphi \right\ir - \left\il\bar\sJ^N, \varphi \right\ir \right|\wedge 1 + \left| \left\il\bar \sJ^N, \varphi \right\ir - G_\varphi^*\left(Q^N\right) \right| + \left| G_\varphi^*\left(Q^N\right) - G_\varphi(Q) \right|, 
\]
we see that to prove \eqref{eq:eq242r} and thus to complete the proof it suffices to argue that the third term on the right side of the above display converges to $0$ in probability.

To this end, define the maps $\tilde G_\varphi$ and $\tilde G^B_\varphi$ on $\{\Theta \in \sP_2(\sZ) : \nu_\Theta \in \sV\} \times \sV$ by
\begin{align*}
	\tilde G_\varphi(\Theta, V) &\doteq E_\Theta\left[ \int_0^T \varphi(t, X(t))\cdot b(X(t), V(t))\,dt + \int_{[0,T]\times\R^m} \varphi(t, X(t))\cdot \sigma(X(t), V(t))y\,\rho(dt, dy)\right], \\
	\tilde G^B_\varphi(\Theta, V) &\doteq E_\Theta\left[ \int_0^T \varphi(t, X(t))\cdot \eta_B\left(b(X(t), V(t))\right)\,dt \right] \\
	&\qquad \qquad + E_\Theta\left[
	\int_{[0,T]\times\R^m} \varphi(t, X(t))\cdot \sigma(X(t), V(t))\psi_B(y)\,\rho(dt, dy)\right], 
\end{align*}
for each $1\le B < \infty$.
Note by \eqref{eq:eq457} that $\tilde G_\varphi(\Theta, \nu_\Theta) = G_\varphi(\Theta)$ whenever $\Theta \in \sS(\sZ)$, and hence since $\bar V = \nu_Q$ and $Q \in \sS(\sZ)$ a.s., we have that $\tilde G_\varphi(Q, \bar V) = G_\varphi(Q)$ a.s. 
Also, since
\begin{align*}
	\tilde G_{\varphi}\left(Q^N, \bar V^N\right) &=
	\frac{1}{N}\sum_{j=1}^N  \int_0^T  \varphi\left(t, \bar X_j^N(t) \right)\cdot b\left(\bar X_j^N(t), \bar V^N(t)\right)\,dt \\
	&\qquad + \frac{1}{N}\sum_{j=1}^N\int_0^T \varphi\left(t, \bar X_j^N(t) \right)\cdot \sigma\left(\bar X_j^N(t), \bar V^N(t) \right) u_j^N(t)\,dt
\end{align*}
and $\eps_N \to 0$, we see that $|\tilde G_{\varphi}(Q^N, \bar V^N)-G_{\varphi}^*(Q^N)| \overset{\tilde P}{\to} 0$
as $N\to \infty$. Thus it remains to argue that 
\begin{equation}\label{eq:finalconv}
	\left| \tilde G_\varphi\left(Q^N, \bar V^N\right) - \tilde G_\varphi\left(Q, \bar V\right) \right| \overset{\tilde P}{\to} 0 \qquad \mbox{as}\; N \to\infty. 
\end{equation}
Now, since 
\begin{align*}
	\tilde G_{\varphi}^B\left(Q^N, \bar V^N\right) &=  \frac{1}{N}\sum_{j=1}^N  \int_0^T  \varphi\left(t, \bar X_j^N(t) \right)\cdot \eta_B\left(b\left(\bar X_j^N(t), \bar V^N(t)\right)\right)\,dt \\
		&\qquad +\frac{1}{N}\sum_{j=1}^N \int_0^T \varphi\left(t, \bar X_j^N(t) \right)\cdot \sigma\left(\bar X_j^N(t), \bar V^N(t) \right) \psi_B\left(u_j^N(t)\right)\,dt, 
\end{align*}
and the map
$$(\xi, r, V) \mapsto 
\int_0^T \varphi\left(t, \xi(t) \right) \cdot \eta_B(b(\xi(t),  V(t)))\,dt  +
\int_{[0,T]\times\R^m}\varphi\left(t, \xi(t) \right) \cdot \sigma(\xi(t),  V(t)) \psi_B(y)\, r (dt, dy)$$
is bounded and continuous on $\sZ \times \sV$, the a.s. convergence 
$(Q^N, \bar V^N) \to (Q, \bar V)$ in $\sP(\sZ)\times \sV$ implies that 
\begin{equation}\label{eq:eq249}
	\tilde G_{\varphi}^B\left(Q^N, \bar V^N\right) \to \tilde G_{\varphi}^B\left(Q, \bar V\right) \qquad  \mbox{a.s.,}\qquad  \mbox{as}\; N\to \infty, 
\end{equation}
for each $B$. 
Also, 
using \eqref{eq:psiBproperty} and \eqref{eq:psiBpropertyb}, as in the proof of \eqref{eq:eq237r}, we see
\begin{align*}
&\left|\tilde G_{\varphi}^B\left(Q^N, \bar V^N\right)-\tilde G_{\varphi}\left(Q^N, \bar V^N\right)\right | \\
&\le \frac{18K_{\varphi}L^2T}{B}\left( 1 + \frac{1}{N}\sum_{j=1}^N  \left\|\bar X_j^N\right\|_\infty^2\right) + \frac{3K_{\varphi}L}{BN}\sum_{j=1}^N \int_0^T \left|u_j^N(t)\right|^2 dt, 
\end{align*}
which in view of Lemma \ref{L2bound} and the assumption on the controls in Lemma \ref{keytightness} 
shows that 
\begin{equation}\label{eq:eq249b}
	\sup_{N\ge1} \tilde E\left[\left|\tilde G_{\varphi}^B\left(Q^N, \bar V^N\right)-\tilde G_{\varphi}\left(Q^N, \bar V^N\right)\right | \right] \to 0 \qquad \mbox{as}\; B\to \infty.
\end{equation}
Finally, along the same lines as in the proof of \eqref{eq:244}, 
\begin{align*}
\left|\tilde G_{\varphi}^B\left(Q, \bar V\right)-\tilde G_{\varphi}\left(Q, \bar V\right)\right|	\to 0 \qquad  \mbox{a.s.,} \qquad \mbox{as}\; B\to \infty.
\end{align*}
Combining the above convergence with \eqref{eq:eq249} and \eqref{eq:eq249b}
shows \eqref{eq:finalconv}, 
which as noted previously establishes that $Q \in \sP^*(\bar\sJ)$ a.s.  
and thus completes the proof of the lemma. \hfill \qed

\subsection{Proof of Lemma \ref{lem:lem908}}
\label{sec:pflem908} 

We first prove an estimate similar to that in Lemma \ref{L2bound} for the coordinate process $X(t)$ on the space $(\sZ, \sB(\sZ), \Theta)$ for each $\Theta \in \sP_2(\sZ)\cap \sS(\sZ)$. By the definition of $\sS(\sZ)$, the coordinate maps $(X, \rho)$ satisfy 
\begin{equation}\label{eq:XSDE}
	dX(t) = b\left(X(t), \nu_\Theta(t)\right)\,dt + \int_{\R^m} \sigma\left( X(t), \nu_\Theta(t)\right)y\,\rho_t(dy)\,dt \qquad \Theta\mbox{-a.s.},
\end{equation}
with $X(0) \sim \mu_0$. 
By Condition \ref{maincondition}, 
\begin{equation}\label{eq:bThetagrowth}
\begin{aligned}
	\left|b\left(X(t), \nu_\Theta(t)\right)\right|^2 &\le 3L^2\left( 1 + |X(t)|^2 + \int_{\R^d} |x|^2\,\nu_\Theta(t, dx) \right) \\
	&=  3L^2\left( 1 + |X(t)|^2 + E_\Theta\left[ |X(t)|^2 \right] \right). 
\end{aligned}
\end{equation}
Applying the above bound in \eqref{eq:XSDE}, taking expectation,  using $|\sigma| \le L$, and applying Gronwall's inequality, we have
\begin{equation}\label{eq:XL2bound}
	E_\Theta\left[ \|X\|_\infty^2\right]\le \tilde c\left( 1 + \int_{\R^d}|x|^2\,\mu_0(dx) + E_\Theta\left[\int_{[0,T] \times \R^m} |y|^2\,\rho(dt, dy)\right]\right) < \infty, 
\end{equation}
for some $\tilde c = \tilde c(L,T) < \infty$.

Now fix $\mathbf{s} \in \sO_d$ and let $\{(\mu_k, \Theta_k, \sJ_k)\}$ be a sequence in $\sP_1(\sX) \times \sP(\sZ) \times \mathbf{H}^{-\mathbf{s}}$ that satisfies the hypotheses of the lemma. 
Note that, by \eqref{eq:XL2bound}, 
\begin{equation}\label{eq:supThetakL2}
\begin{aligned}
	\sup_{k\ge1}\int_\sX \|\psi\|^2_\infty \,\mu_k(d\psi) &= \sup_{k\ge1}\int_\sX \|\psi\|^2_\infty \,(\Theta_k)_{(1)}(d\psi) = \sup_{k\ge1} E_{\Theta_k}\left[\|X\|_\infty^2\right] \\
	&\le \tilde c \left( 1 + \int_{\R^d}|x|^2\,\mu_0(dx) + \sup_{k\ge1}E_{\Theta_k}\left[\int_{[0,T]\times\R^m} |y|^2\,\rho(dt, dy)\right]\right) < \infty. 
\end{aligned}
\end{equation}
If $\tau$ is a $\{\sigma(X(s), s \le t)\}$-stopping time on $(\sZ, \sB(\sZ))$ taking values in $[0, T-\eps]$, then for any $\eps > 0$, 
\begin{align*}
	\left|X(\tau+\eps) - X(\tau)\right|^2 &\le 2\left| \int_{\tau}^{\tau+\eps} b(X(t), \nu_{\Theta_k}(t))\,dt \right|^2 + 2\left|\int_{\tau}^{\tau+\eps} \int_{\R^m} \sigma(X(t), \nu_{\Theta_k}(t))y\,\rho_t(dy)\,dt \right|^2 \\
	&\le 6L^2 \eps \int_0^T \left( 1 + |X(t)|^2 + E_{\Theta_k}\left[|X(t)|^2 \right]\right)\,dt + 2L^2\eps \int_0^T \int_{\R^m} |y|^2\,\rho_t(dy)\,dt , 
\end{align*}
$\Theta_k$-a.s. for each $k$. Hence, using the bound in \eqref{eq:XL2bound}, 
\[
	E_{\Theta_k} \left[ \left|X(\tau+\eps) - X(\tau)\right|^2 \right] \le 12 L^2(1+\tilde c)\eps \left( 1 + \int_{\R^d}|x|^2\,\mu_0(dx) + \sup_{k\ge1}E_{\Theta_k}\left[\int_{[0,T]\times \R^m} |y|^2\,\rho(dt, dy)\right]\right). 
\]
If $\sT_\eps$ denotes the collection of all such stopping times $\tau$, it follows that 
\begin{align*}
	 \sup_{k\ge1} \sup_{\tau \in \sT_\eps} \int_{\sX} \left|\psi(\tau+\eps) - \psi(\tau)\right|^2\,\mu_k(d\psi) &=  \sup_{k\ge1} \sup_{\tau \in \sT_\eps} \int_{\sX} \left|\psi(\tau+\eps) - \psi(\tau)\right|^2\,(\Theta_k)_{(1)}(d\psi) \\
	&= \sup_{k\ge 1} \sup_{\tau \in \sT_\eps} E_{\Theta_k} \left[ \left|X(\tau+\eps) - X(\tau)\right|^2 \right] \\
	&\to 0 
\end{align*}
as $\eps \to 0$. This and \eqref{eq:supThetakL2} prove relative compactness of $\{\mu_k\}$ (and hence of $\{(\Theta_k)_{(1)}\}$) in $\sP(\sX)$. By Lemma \ref{P1tightness} and \eqref{eq:supThetakL2}, we  in fact get relative compactness of  $\{\mu_k\}$  in $\sP_1(\sX)$ . 

For the second marginals $\{(\Theta_k)_{(2)}\}$, we recall from the proof of Lemma \ref{keytightness} that 
\[
	H(\theta) = \int_{\sR_1} \int_{[0,T]\times\R^m} |y|^2\,r(dt, dy)\,\theta(dr)
\]
has relatively compact level sets on $\sP(\sR_1)$. Hence, we have relative compactness of $\{(\Theta_k)_{(2)}\}$ in $\sP(\sR_1)$ on observing that
\[
	\sup_{k\ge1} H\left((\Theta_k)_{(2)}\right) = \sup_{k\ge1} E_{\Theta_k} \left[ \int_{[0,T]\times\R^m} |y|^2\rho(dt, dy) \right] < \infty. 
\]
This establishes that $\{\Theta_k\}$ is relatively compact in $\sP(\sZ)$.

For $\{\sJ_k\}$, we employ the following lemma, the proof of which is saved for the Appendix. 
\begin{lemma}\label{Gpathwisereal} 
	Suppose Conditions \ref{maincondition} and \ref{initialcondition} are satisfied.
	Also suppose that for some $\mathbf{s} \in \sO_d$ and $(\mu, \sJ) \in \sP_1(\sX) \times \mathbf{H}^{-\mathbf{s}}$,
	$I(\mu, \sJ)<\infty$. Then, for each $\mathbf{s}' \in \sO_d$, there is a constant $C_{\mathbf{s}'} < \infty$ such that for any $\Theta \in \sP^*(\sJ)$ with
	$\Theta_{(1)}=\mu$, and for all $\varphi \in \sC_c^\infty(U\times\R^d, \R^d)$,  
	\[
		\left| \il \sJ, \varphi \ir \right|^2 \le E_{\Theta} \left[ \left| \int_0^T \varphi(t, X(t))\cdot dX(t) \right|^2\right] \le C_{\mathbf{s}'} \left( 1 + E_\Theta\left[ \int_{[0,T]\times \R^m} |y|^2 \,\rho(dt, dy) \right] \right) \|\varphi\|^2_{\mathbf{s}'}, 
	\]
	where $C_{\mathbf{s}'}$ does not depend on $\sJ$, $\varphi$, or $\Theta$. 
	In particular, 
	$\sJ \in \mathbf{H}^{-\mathbf{s}'}$ for all
	$\mathbf{s}' \in \sO_d$.
\end{lemma}
Recall the collection of test functions $\{g_M, M < \infty\}$ from Definition \ref{mollifiers}, which by Lemma \ref{gMnormbound} (see \eqref{eq:eq222r}) satisfy 
\begin{equation}\label{eq:Kbound}
	\|g_M\varphi\|_{\mathbf{s}} 
	\le K\|\varphi\|_{\mathbf{s}}, 
\end{equation}
 for all $\varphi \in \sC_c^\infty(U\times\R^d, \R^d)$
and $\mathbf{s}\in \sO_d$, with $K < \infty$ depending only on  $\mathbf{s}$. For each $k\ge1$ and $M < \infty$, define $\sJ_k^M, \sJ_k^{M,c} \in \mathbf{H}^{-\mathbf{s}}$ by 
\[
	\left\il \sJ_k^M, \varphi \right\ir \doteq \left\il \sJ_k, g_M\varphi \right\ir, \quad \left\il \sJ_k^{M,c}, \varphi \right\ir \doteq \left\il \sJ_k, \varphi \right\ir - \left\il \sJ_k^M, \varphi \right\ir, \quad \varphi \in \sC_c^\infty(U\times\R^d, \R^d). 
\]
Fix some $\mathbf{s}' \in \sO_d$ such that $s_1' < s_1$ and $s_2' < s_2$. 
Since $\Theta_k \in \sP^*(\sJ_k)$ for each $k$ and \eqref{eq:eq100} holds, $I(\mu_k, \sJ_k) < \infty$ for each $k$, so by Lemma \ref{Gpathwisereal}, $\sJ_k \in \mathbf{H}^{-\mathbf{s}'}$ for each $k$.
Then for each $k$ and $M$, in view of \eqref{eq:Kbound}, $\sJ_k^M$ and $\sJ_k^{M,c}$ are in $\mathbf{H}^{-\mathbf{s}'}$ as well, and furthermore, 
\begin{align*}
	\left| \left\il \sJ_k^M, \varphi \right\ir\right|^2 &\le C_{\mathbf{s}'} \left( 1 + E_\Theta\left[ \int_{[0,T]\times\R^m} |y|^2 \,\rho(dt, dy) \right] \right) \|g_M\varphi\|^2_{\mathbf{s}'} \\
	&\le C_{\mathbf{s}'} K^2 \left( 1 + E_{\Theta_k}\left[ \int_{[0,T]\times\R^m} |y|^2 \,\rho(dt, dy) \right] \right) \|\varphi\|^2_{\mathbf{s}'}, 
\end{align*}
and hence 
\begin{equation}\label{eq:supJM}
	\sup_{M<\infty, k\ge1} \left\|\sJ_k^M \right\|^2_{-\mathbf{s}'} \le C_{\mathbf{s}'} K^2 \left( 1 + \sup_{k\ge1}E_{\Theta_k}\left[ \int_{[0,T]\times \R^m} |y|^2 \,\rho(dt, dy) \right] \right) < \infty. 
\end{equation}
Noting that for each $M$, $\{\sJ_k^M\}$ are all supported on $[0,T] \times \{|x| \le M+1\} \subset U\times \R^d$,  by Lemma \ref{compactembedding}, $\{\sJ_k^M, k \ge 1\}$ is relatively compact in $\mathbf{H}^{-\mathbf{s}}$.  
Now define the collection of stopping times $\{\tau^M, M < \infty\}$ on $(\sZ, \sB(\sZ))$ by $\tau^M \doteq \inf\{t > 0 : |X(t)| \ge M\}$. Note that 
\begin{align*}
	\left\il \sJ_k^{M,c}, \varphi \right\ir &= E_{\Theta_k} \left[ \int_0^T (1-g_M(X(t)))\varphi(t, X(t))\cdot dX(t) \right] \\
	&= E_{\Theta_k} \left[1_{\left\{ \tau^M < T \right\}} \int_0^T (1-g_M(X(t)))\varphi(t, X(t))\cdot dX(t) \right], 
\end{align*}
and so by Lemma \ref{Gpathwisereal} and \eqref{eq:Kbound}, 
\begin{align*}
	\left| \left\il \sJ_k^{M,c}, \varphi \right\ir \right|^2 &\le \Theta_k\left(\tau^M < T \right) E_{\Theta_k}\left[ \left| \int_0^T (1-g_M(X(t)))\varphi(t, X(t)) \cdot dX(t) \right|^2 \right] \\
	&\le \Theta_k\left(\tau^M < T \right)  C_{\mathbf{s}} \left( 1 + E_{\Theta_k}\left[ \int_{[0,T]\times \R^m} |y|^2 \,\rho(dt, dy) \right] \right) \|(1-g_M)\varphi\|^2_{\mathbf{s}} \\
	&\le 2\Theta_k\left(\tau^M < T \right)  C_{\mathbf{s}} \left(1+K^2\right) \left( 1 + E_{\Theta_k}\left[ \int_{[0,T]\times\R^m} |y|^2 \,\rho(dt, dy) \right] \right) \|\varphi\|^2_{\mathbf{s}}, 
\end{align*}
and hence 
\begin{equation}\label{eq:eq307r}
\begin{aligned}
	\sup_{k\ge1} \left\| \sJ_k^{M,c} \right\|^2_{-\mathbf{s}} &\le 2\sup_{k\ge1}\Theta_k\left(\tau^M < T \right)  C_{\mathbf{s}} \left(1+K^2\right) \left( 1 + \sup_{k\ge1} E_{\Theta_k}\left[ \int_{[0,T]\times\R^m} |y|^2 \,\rho(dt, dy) \right] \right) \\
	&\le \frac{2}{M^2}\sup_{k\ge 1} E_{\Theta_k}\left[ \|X\|_\infty^2\right] C_{\mathbf{s}} \left(1+K^2\right) \left( 1 + \sup_{k\ge1} E_{\Theta_k}\left[ \int_{[0,T]\times\R^m} |y|^2 \,\rho(dt, dy) \right] \right) \\
	&\to 0
\end{aligned}
\end{equation}
as $M \to \infty$, by \eqref{eq:XL2bound}. Then by Lemma \ref{tightnessdecomp} (applied to the constant random variables $\sJ_k = \sJ_k^M + \sJ_k^{M,c}$ on $(\sZ, \sB(\sZ))$), we obtain from \eqref{eq:supJM} and \eqref{eq:eq307r} that $\{\sJ_k\}$ is relatively compact in $\mathbf{H}^{-\mathbf{s}}$. 
Lemma \ref{lem:lem908} now follows on combining the above with the relative compactness of
$\{(\mu_k, \Theta_k)\}$ in $\sP_1(\sX)\times \sP(\sZ)$ shown previously.
\hfill \qed

\subsection{Proof of Lemma \ref{lem:weakuniq}}\label{sectionweakuniq}
Recall that we assume that Conditions  \ref{maincondition}, \ref{initialcondition}  and \ref{sigmacondition} hold. In particular, $\sigma(x, \mu) = \sigma(\mu)$.
Let $\Theta_1, \Theta_2 \in \sS(\sZ)\cap\sP_2(\sZ)$ be such that $\Theta_1\circ \theta^{-1} = \Theta_2 \circ \theta^{-1}$, and let $\Lambda = \Theta_1 \circ \theta^{-1}$. Then for  $j =1, 2$, we can disintegrate $\Theta_j$ as
\[
	\Theta_j(dx, dr) = \widetilde \Theta_j(x_0, r, dx) \, \Lambda(dx_0, dr)
\]
for some measurable map  $\widetilde \Theta_j : \R^d\times \sR_1\to \sP(\sX)$. Define the probability measure $\Xi$ on the space $\R^d\times \sR_1\times \sX \times \sX$ as 
\[
	\Xi(dx_0, dr, dx_1, dx_2) = \widetilde \Theta_1(x_0, r, dx_1)\, \widetilde \Theta_2(x_0, r, dx_2)\, \Lambda(dx_0, dr), 
\]
and let $(\xi_0, \rho, X_1, X_2)$ denote the coordinate maps on this space. Then, $X_1(0) = X_2(0) = \xi_0$, and to prove the lemma it suffices to show that $X_1= X_2$ $\Xi$-a.s.  

Letting $u(t) = \int_{\R^m} y\, \rho_t(dy)$ and $V_j(t) = \Xi \circ (X_j(t))^{-1}$, we have that $E_{\Xi}\left[\int_0^T |u(t)|^2\,dt\right] < \infty$ and 
\[
	X_j(t) = \xi_0 +  \int_0^t b\left(X_j(s), V_j(s)\right)\,ds + \int_0^t \sigma\left(V_j(s)\right) u(s)\,ds, \qquad j = 1,2.
\]
By the Lipschitz property of the coefficients and the fact that 
\[
	d_{1}\left(V_1(s), V_2(s)\right)^2 \le \left(E_{\Xi}\left[\left|X_1(s) - X_2(s)\right| \right] \right)^2 \le E_\Xi \left[ \sup_{0\le r\le s} \left|X_1(r) - X_2(r)\right|^2\right], 
\]
it follows from Condition \ref{maincondition} that for every $0 \le t \le T$, 
\begin{align*}
	\left|X_1(t) - X_2(t)\right|^2 &\le 2T\int_0^t \left| b\left( X_1(s), V_1(s)\right) -  b\left( X_2(s), V_2(s)\right) \right|^2\,ds   \\
	&\qquad + 2\left(\int_0^T \left|u(s)\right|^2\,ds\right) \int_0^T \left| \sigma\left(V_1(s) \right) - \sigma\left(V_2(s)\right)\right|^2\,ds \\
	&\le 2L^2 T \int_0^t \left( \left| X_1(s) - X_2(s)\right| + d_{1}\left(V_1(s), V_2(s)\right)\right)^2\,ds \\
	&\qquad + 2L^2 \left(\int_0^t \left|u(s)\right|^2\,ds\right)\int_0^T d_{1}\left(V_1(s), V_2(s)\right)^2\,ds \\
	&\le 4L^2 T \int_0^t \sup_{0\le r \le s} \left|X_1(r) - X_2(r)\right|^2\,ds \\
	&\qquad + 2L^2\left( 2T + \int_0^T |u(t)|^2\,dt\right) \int_0^t  E_\Xi \left[ \sup_{0\le r\le s} \left|X_1(r) - X_2(r)\right|^2\right] \,ds. 
\end{align*}
Then taking expectation with respect to $\Xi$, for all $0 \le t \le T$, 
\begin{align*}
	&E_\Xi\left[ \sup_{0\le s \le t} |X_1(s) - X_2(s)|^2\right] \\
	&\le 2L^2 \left(4T + E_\Xi\left[ \int_0^T |u(s)|^2\,ds \right]\right) \int_0^{t}  E_\Xi \left[ \sup_{0\le r\le s} \left|X_1(r) - X_2(r)\right|^2\right] \,ds.
\end{align*}
Gronwall's inequality now shows that 
$E_\Xi\left[\left\|X_1 - X_2\right\|_\infty^2\right] = 0$, 
which completes the proof. \hfill\qed

\appendix

\section{}
In this section we provide proofs of some Sobolev space results that are used in our work.
It will be convenient to introduce an alternate norm on $\mathbf{H}^{\mathbf{s}}$ equivalent to \eqref{eq:HstNorm}, and which is similar to norms used in \cite{BerButPis} and \cite{Orr}. 
Let $\{e_1, \ldots, e_d\}$ denote the canonical basis in $\RR^d$, recall that $U = (a, b)\supset [0,T]$, let $\sI \doteq \ZZ\times\R^d \times  \{1, \ldots , d\} $, and define the functions $e^k_{n, \xi} : U \times\R^d \to \R^d$ for $(n,\xi, k) \in \sI$ by 
\[
	e^k_{n,\xi}(t,x) = \frac{1}{b-a} e^{2\pi i nt/(b-a)} e^{2\pi i \xi \cdot x} e_k.
\]
Consider the Fourier coefficients of $\varphi \in \sC_c^\infty(U \times\R^d, \R^d)$  given by 
\begin{equation}\label{eq:hatphi}
	\hat \varphi(n,\xi) = \left(\hat \varphi_1(n,\xi), \ldots, \hat \varphi_d(n,\xi) \right), \qquad \hat \varphi_k(n,\xi) = \int_{U} \int_{\R^d} e^k_{-n,-\xi}(t,x) \cdot \varphi(t,x)\,dx\,dt. 
\end{equation}
Then an equivalent norm on $\mathbf{H}^{\mathbf{s}}$, $\mathbf{s} = (s_1, s_2) \in \R_+^2$, is given by 
\begin{equation}\label{eq:equivnorm}
	\|\varphi\|^2_{*,\mathbf{s}} = \sum_{n \in \Z} \int_{\R^d} \left| \hat\varphi(n,\xi) \right|^2\left(1+n^2\right)^{s_1}\left(1+|\xi|^2\right)^{s_2}\,d\xi. 
\end{equation}

\subsection{Proof of Lemma \ref{JbarReal}}
From the equivalence of the norms, it suffices to prove the statement in the lemma with $\|\cdot\|_{\mathbf{s}}$ replaced with
$\|\cdot\|_{*,\mathbf{s}}$. In what follows, we will abuse notation and denote $\|\cdot\|_{*,\mathbf{s}}$ 
once more as $\|\cdot\|_{\mathbf{s}}$.
Recall that for  $N\in \NN$, $1\le j \le N$, and $\varphi \in \sC_c^\infty(U\times\R^d, \R^d)$,
\[
	\bar J_j^N(\varphi) = \int_0^T \varphi\left(t, \bar X_j^N(t) \right)\circ d\bar X_j^N(t). 
\]
Any such $\varphi$ can be written in terms of its Fourier coefficients as 
\[
	\varphi(t,x) = \sum_{k=1}^d \sum_{n \in \Z} \int_{\R^d} \hat \varphi_k(n,\xi) e^k_{n,\xi}(t,x)\,d\xi.
\]
As in \cite[Lemma 8]{flagubgiator}  it follows that
\begin{align*}
	\bar J_j^N(\varphi) 
	&= \sum_{k=1}^d \sum_{n \in \Z} \int_{\R^d} \hat \varphi_k(n,\xi) Z^N_{j,k}(n,\xi)\,d\xi, 
\end{align*}
where 
\[
	Z^N_{j,k}(n,\xi) \doteq \int_0^T e^k_{n,\xi}\left(t, \bar X_j^N(t)\right) \circ d\bar X_j^N(t). 
\]
 Note that 
\begin{align*}
	Z^N_{j,k}(n,\xi) 
	&= \int_0^T e^k_{n,\xi}\left(t, \bar X_j^N(t) \right)\cdot d\bar X_j^N(t) + \frac{1}{2}\left\il e^k_{n,\xi}\left(\cdot, \bar X_j^N(\cdot)\right), \bar X_j^N(\cdot) \right\ir_T \\
	&= \int_0^T e^k_{n, \xi}\left(t, \bar X_j^N(t)\right)\cdot b\left(\bar X_j^N(t), \bar V^N(t)\right)\,dt + \int_0^T e^k_{n, \xi}\left(t, \bar X_j^N(t)\right)\cdot \sigma\left(\bar X_j^N(t), \bar V^N(t)\right)u_j^N(t)\,dt \\
	&\qquad +  \eps_N\int_0^T e^k_{n,\xi}\left(t, \bar X_j^N(t)\right)\cdot \sigma\left(\bar X_j^N(t), \bar V^N(t)\right)\,dW_j(t) \\
	&\qquad + \pi i \eps_N^2\xi_k \int_0^T \left(e^k_{n,\xi}\right)_k\left(t, \bar X_j^N(t)\right) \left(\sigma \sigma^{\mathsf{T}}\right)_{kk}\left(\bar X_j^N(t), \bar V^N(t) \right) \, dt, 
\end{align*}
since the $k$th component $(e^k_{n,\xi})_k$ is the only nonzero component of $e^k_{n,\xi}$. By the Cauchy-Schwarz inequality, for all $\varphi \in \sC_c^\infty(U \times\R^d, \R^d)$,
\begin{equation}\label{eq:eq336r}
	\left| \bar J_j^N(\varphi) \right|^2 \le \|\varphi\|^2_{\mathbf{s}}\sum_{k=1}^d \sum_{n \in \Z} \int_{\R^d} \frac{\left|Z^N_{j,k}(n,\xi) \right|^2}{\left(1+n^2\right)^{s_1}\left(1+|\xi|^2\right)^{s_2}}\,d\xi  
= \|\varphi\|^2_{\mathbf{s}}\left(C^N_{j, \mathbf{s}}\right)^2,\end{equation}
where
\[
	C^N_{j, \mathbf{s}} \doteq \left( \sum_{k=1}^d \sum_{n \in \Z} \int_{\R^d} \frac{\left|Z^N_{j,k}(n,\xi) \right|^2}{\left(1+n^2\right)^{s_1}\left(1+|\xi|^2\right)^{s_2}}\,d\xi \right)^{1/2}.
\]
Since $|e^k_{n,\xi}| \le T^{-1}$ and $|\sigma| \le L$, the Burkholder-Davis-Gundy inequality gives 
\begin{equation}\label{eq:eq511}
\begin{aligned}
	E\left[\left| Z^N_{j,k}(n, \xi) \right|^2\right] &\le 4 E\left[ \int_0^T \left| b\left(\bar X_j^N(t), \bar V^N(t)\right) \right|^2\,dt \right] +   \frac{4 L^2}{T} E\left[ \int_0^T \left| u_j^N(t) \right|^2\,dt \right]   + \frac{4\eps_N^2 L^2}{T} \\
	&\qquad + \frac{4\pi^2\eps_N^4L^4 \xi_k^2}{T}. 
\end{aligned}
\end{equation}
By the linear growth property of $b$ from Condition \ref{maincondition}, 
\[
	\left|b\left(\bar X_j^N(t), \bar V^N(t) \right) \right|^2 \le 3 L^2 \left( 1 + \left| \bar X_j^N(t) \right|^2 + \frac{1}{N} \sum_{l=1}^N \left| \bar X_l^N(t) \right|^2 \right), 
\]
and from Lemma \ref{L2bound}, 
$
	E\left[\|\bar X_j^N\|_\infty^2\right] <\infty
$
for each $N \in \N$ and $1\le j \le N$. 
Using the last two estimates and \eqref{eq:eq511}, we see that
$$\sup_{(n,\xi,k) \in \sI} E\left[\left| Z^N_{j,k}(n, \xi) \right|^2\right]  <\infty.$$
Thus, for each $N \in \NN$ and $1\le j \le N$,
$E[| C^N_{j, \mathbf{s}}|^2] <\infty$ for any $\mathbf{s} \in \sO_d$.
Following \cite{flagubgiator}, we now have from \eqref{eq:eq336r} the existence of a pathwise realization $\bar \sJ^N$ of
$\{\varphi \mapsto \bar J^N(\varphi)\}$ in $\mathbf{H}^{-\mathbf{s}}$ for every $N\in \NN$ and any $\mathbf{s} \in \sO_d$. This proves the first part of the lemma.

For the second part, note that
 by Lemma \ref{L2bound}, 
\begin{align*}
	E\left[ \left|b\left(\bar X_j^N(t), \bar V^N(t) \right) \right|^2 \right] &\le 4L^2(c+1) \left( 1 + \left|x_j^N\right|^2 + E\left[ \int_0^T \left|u_j^N(t)\right|^2\,dt \right] + \frac{1}{N} \sum_{l=1}^N \left|x_l^N\right|^2 \right. \\
	&\qquad +\left.  E\left[ \frac{1}{N}\sum_{l=1}^N \int_0^T \left|u_l^N(t)\right|^2\,dt \right] \right). 
\end{align*}
Thus for some constant $K < \infty$ depending only on $d$, $T$, and $L$,
\[	
	\frac{1}{N}\sum_{j=1}^N \sum_{k=1}^d E\left[\left| Z^N_{j,k}(n,\xi) \right|^2\right] \le K\left( 1 + |\xi|^2 + \frac{1}{N}\sum_{j=1}^N \left|x_j^N\right|^2 + E\left[\frac{1}{N}\sum_{j=1}^N \int_0^T \left|u_j^N(t)\right|^2\,dt\right] \right). 
\]
Letting $C^N_{\mathbf{s}} = \frac{1}{N} \sum_{j=1}^N C^N_{j, \mathbf{s}}$, we have from \eqref{eq:eq336r} that, for all $\varphi \in \sC_c^\infty(U \times\R^d, \R^d)$,
\[
	\left| \bar J^N(\varphi) \right| \le \frac{1}{N}\sum_{j=1}^N \left|\bar J_j^N(\varphi) \right| \le C^N_{\mathbf{s}} \|\varphi\|_{\mathbf{s}}. 
\]
Finally,
\begin{multline*}
	E\left[ \left(C^N_{\mathbf{s}}\right)^2\right] \le \frac{1}{N}\sum_{j=1}^N E\left[ \left(C^N_{j,\mathbf{s}}\right)^2\right] \\
	\le \sum_{n\in \Z}\int_{\R^d} \frac{K}{\left(1+n^2\right)^{s_1}\left(1+|\xi|^2\right)^{s_2}} \left( 1 + |\xi|^2 + \sup_{N\ge1} \frac{1}{N}\sum_{j=1}^N \left|x_j^N\right|^2 + \sup_{N\ge 1} E\left[ \frac{1}{N} \sum_{j=1}^N \int_0^T \left|u_j^N(t) \right|^2\,dt\right] \right)\,d\xi, 
\end{multline*}
which is finite by Condition \ref{initialcondition} and \eqref{eq:eq525} since $\mathbf{s} \in \sO_d$. \hfill \qed

\subsection{Proof of Lemma \ref{Gpathwisereal}}
As in the proof of Lemma \ref{JbarReal}, it suffices to prove the statement in the lemma with $\|\cdot\|_{\mathbf{s}}$ replaced with
$\|\cdot\|_{*,\mathbf{s}}$, and once again, abusing notation, we will denote $\|\cdot\|_{*,\mathbf{s}}$ 
 as $\|\cdot\|_{\mathbf{s}}$.
Suppose that $\mathbf{s} \in \sO_d$ and $(\mu, \sJ) \in \sP_1(\sX)\times \mathbf{H}^{-\mathbf{s}}$ 
are such that
 $I(\mu, \sJ) < \infty$. Then there is some $\Theta \in \sP^*(\sJ)$ such that $\Theta_{(1)} = \mu$ and 
\[
	\il \sJ, \varphi \ir = G_\varphi(\Theta) = E_\Theta\left[ \int_0^T \varphi(t, X(t)) \cdot dX(t) \right], 
\]
for all $\varphi \in \sC_c^\infty(U\times\R^d, \R^d)$. Furthermore, the estimate \eqref{eq:XL2bound} holds for this $\Theta$. 
By an argument as in the proof of Lemma \ref{JbarReal},
\[
	 \int_0^T \varphi(t, X(t)) \cdot dX(t) = \sum_{k=1}^d \sum_{n \in \Z} \int_{\R^d} \hat \varphi_k(n,\xi)Z_k(n,\xi)\,d\xi \qquad \Theta\mbox{-a.s.}, 
\]
where $\hat \varphi_k$ is defined in \eqref{eq:hatphi} and 
\begin{align*}
	Z_k(n,\xi) 
	&\doteq \int_0^T e^k_{n,\xi}\left(t, X(t) \right)\cdot dX(t) \\
	&= \int_0^T e^k_{n,\xi}\left(t, X(t) \right)\cdot  b\left(X(t), \nu_\Theta(t) \right)\,dt + \int_{[0,T]\times\R^m} e^k_{n,\xi}\left(t, X(t) \right) \cdot \sigma\left(X(t), \nu_\Theta(t)\right)y\,\rho(dt, dy)
\end{align*}
$\Theta$-a.s. Since $|e^k_{n, \xi}| \le T^{-1}$, using \eqref{eq:bThetagrowth} we have
\begin{align*}
	\left|Z_k(n, \xi) \right|^2 &\le \frac{6 L^2}{T} \int_0^T \left( 1 + |X(t)|^2 + E_\Theta\left[|X(t)|^2 \right] \right)\,dt + \frac{2L^2}{T} \int_{[0,T]\times\R^m} |y|^2\,\rho(dt, dy), 
\end{align*}
and then the bound in \eqref{eq:XL2bound} gives 
\[
	\sup_{(n,\xi, k)\in \sI} E_\Theta\left[ \left|Z_k(n,\xi)\right|^2\right] \le c'\left(1 + \int_{\R^d}|x|^2\,\mu_0(dx) + E_\Theta\left[\int_{[0,T]\times \R^m} |y|^2\,\rho(dt,dy)\right] \right), 
\]
for some $c' < \infty$. 
Thus by the Cauchy-Schwarz inequality, for any $\mathbf{s}' = (s'_1, s'_2) \in \sO_d$
and $\varphi \in \sC_c^\infty(U\times\R^d, \R^d)$,
\begin{align*}
	&\left|\il \sJ, \varphi \ir \right|^2 \le E_{\Theta} \left[ \left| \int_0^T \varphi(t, X(t)) \cdot dX(t) \right|^2\right] \\
	&\le E_\Theta\left[ \sum_{k=1}^d \sum_{n\in \Z} \int_{\R^d} \frac{\left|Z_k(n,\xi)\right|^2}{\left(1+n^2\right)^{s'_1}\left(1+|\xi|^2\right)^{s'_2}}\,d\xi \right] \|\varphi\|_{\mathbf{s}'}^2 \\
	&\le c' \sum_{n \in \Z} \int_{\R^d} \frac{d\xi}{\left(1+n^2\right)^{s'_1}\left(1+|\xi|^2\right)^{s'_2}} \left(1 + \int_{\R^d}|x|^2\,\mu_0(dx) + E_\Theta\left[\int_{[0,T]\times\R^m} |y|^2\,\rho(dt,dy)\right] \right)  \|\varphi\|_{\mathbf{s}'}^2 \\
	&\le C_{\mathbf{s}'}^2 \left( 1 + E_\Theta\left[ \int_{[0,T]\times\R^m} |y|^2 \,\rho(dt,dy) \right] \right) \|\varphi\|_{\mathbf{s}'}^2
\end{align*}
where 
\[
	C_{\mathbf{s}'}^2 \doteq c' \left( 1 + \int_{\R^d} |x|^2\,\mu_0(dx) \right) \sum_{n \in \Z} \int_{\R^d} \frac{d\xi}{\left(1+n^2\right)^{s'_1}\left(1+|\xi|^2\right)^{s'_2}} < \infty,
\]
since  $\mathbf{s}' = (s'_1, s'_2) \in \sO_d$. The result follows. \hfill \qed

\subsection{Proof of Lemma \ref{gMnormbound}}
We will only consider the case where $s$ is not an integer, the proof for the case when $s$ is an integer is a simpler version of the proof given below. An equivalent norm to $\|\cdot\|_s$  in \eqref{eq:FourierSobolevnorm} can be given as follows (see \cite[page 527]{NezPalVal}): write $s = k + r$ where $k \in \N$ and $r \in (0,1)$. Then, for $h \in H^s(\R^d,  \R^d)$, define
\[
	\|h\|^2_{*,s} \doteq \|h\|^2_{k} + \sum_{|\alpha|=k}\|D^\alpha  h\|^2_{r}, 
\]
where $\|\cdot\|_k$ is the usual integer Sobolev norm 
\[
	\|h\|_k^2 = \sum_{0\le |\alpha| \le k} \|D^\alpha h\|^2_{L^2}, 
\]
and $\|\cdot\|_{r}$ is the fractional Gagliardo-type Sobolev norm 
\begin{equation}\label{eq:GagliardoNorm}
	\|h\|^2_r = \|h\|^2_{L^2} + [h]^2_r = \int_{\R^d} |h(x)|^2\,dx + \int_{\R^d} \int_{\R^d} \frac{|h(x) - h(y)|^2}{|x-y|^{d+2r}}\,dx\,dy. 
\end{equation}
The norm $\|\cdot\|_{*,s} $ is equivalent to the norm $\|\cdot\|_s$  in \eqref{eq:FourierSobolevnorm} and
thus it suffices to prove Lemma \ref{gMnormbound} with  $\|\cdot\|_{s} $ replaced with
$\|\cdot\|_{*,s} $. Henceforth, abusing notation, we will denote this new norm once more as $\|\cdot\|_s$.
Now let $f$ and $g_M$ be as in the statement of the lemma.
With $B(k)$ as in Definition \ref{mollifiers}(iii), 
the Leibniz product formula gives, for a multi-index $\alpha$ with $|\alpha|\le k$,
\[
	\left| D^\alpha \left(g_M(x) f(x)\right) \right| = \left| \sum_{\beta \le \alpha} {\alpha \choose \beta} D^{\alpha - \beta} g_M(x) D^\beta f(x) \right| \le B(k) \sum_{\beta \le \alpha} {\alpha \choose \beta} \left| D^\beta f(x) \right|, 
\]
and hence for all $M<\infty$ 
\begin{align}
	\|g_Mf\|^2_{k} &= \sum_{0 \le |\alpha| \le k} \int_{\R^d} \left| D^\alpha g_M(x) f(x) \right|^2\,dx \le c_1 \sum_{0\le |\beta| \le k}  \int_{\R^d} \left|D^\beta f(x) \right|^2\,dx = c_1 \|f\|^2_{k}, \label{eq:eq346r}
\end{align}
for some $c_1 = c_1(k) < \infty$. For the $r$ term we follow the proof of  \cite[Lemma 5.3]{NezPalVal}. If $\psi \in \sC_c^\infty(\R^d, \R)$ is such that $0 \le \psi \le B_\psi < \infty$ and $h \in H^{r}(\R^d, \R^d)$ for some $0<r<1$, then $\|\psi h\|^2_{L^2} \le B_\psi^2\|h\|^2_{L^2}$. If $L_\psi$ denotes the Lipschitz constant of $\psi$, then
\begin{align*}
	[\psi h]_{r}^2 
	&= \int_{\R^d} \int_{\R^d} \frac{|\psi(x)h(x) - \psi(y)h(y)|^2}{|x-y|^{d+2r}}\,dx\,dy \\
	&\le 2 \int_{\R^d} \int_{\R^d} \frac{|\psi(x)h(x) - \psi(x)h(y)|^2}{|x-y|^{d+2r}}\,dx\,dy + 2 \int_{\R^d} \int_{\R^d} \frac{|\psi(x)h(y) - \psi(y)h(y)|^2}{|x-y|^{d+2r}}\,dx\,dy \\
	&\le 2B_\psi^2 \int_{\R^d} \int_{\R^d} \frac{|h(x) - h(y)|^2}{|x-y|^{d+2r}}\,dx\,dy + 2\int_{\R^d} \int_{\R^d} \frac{|\psi(x) - \psi(y)|^2|h(y)|^2}{|x-y|^{d+2r}}\,dx\,dy \\
	&\le 2B_\psi^2[h]^2_{r} + 2L_\psi^2 \int_{\R^d} \int_{\left\{|x - y| \le 1\right\}} \frac{|h(y)|^2}{|x - y|^{d + 2(r-1)}} \,dx \,dy + 8B^2_{\psi}\int_{\R^d} \int_{\left\{|x-y| > 1\right\}} \frac{|h(y)|^2}{|x-y|^{d+2r}}\,dx\,dy \\
	&\le 2B_\psi^2[h]^2_{r} + 2\left(L_\psi^2+4B^2_{\psi}\right) c_2 \|h\|^2_{L^2}, 
\end{align*}
for $c_2 = c_2(r) < \infty$. In the last line, we used the fact that for some $c_3, c_4 < \infty$ depending on $r$, 
\[
	\int_{\R^d} \int_{\left\{|x - y| \le 1\right\}} \frac{|h(y)|^2}{|x - y|^{d + 2(r-1)}} \,dx \,dy \le \int_{\R^d} \left( \int_{\{|z| \le 1\}} \frac{1}{|z|^{d+2(r-1)}} \,dz \right) |h(y)|^2\,dy \le c_3 \|h\|^2_{L^2},
\]
since $d + 2(r-1) < d$, and
\[
	\int_{\R^d} \int_{\left\{|x - y| > 1\right\}} \frac{|h(y)|^2}{|x - y|^{d + 2r}} \,dx \,dy \le \int_{\R^d} \left( \int_{\{|z| > 1\}} \frac{1}{|z|^{d+2r}} \,dz \right) |h(y)|^2\,dy \le c_4 \|h\|^2_{L^2},
\]
since $d + 2r > d$. 
Thus we have that
\[
	\|\psi h\|^2_{r} \le 8\left(B_\psi^2+L_\psi^2\right)(c_2+1) \|h\|^2_{r}. 
\]
Then, with $B(k)$ as in Definition \ref{mollifiers} and $L(k)$ as in \eqref{eq:eq928}, we obtain that for $|\alpha|=k$,
\begin{align*}
	\left\|D^\alpha g_M f\right\|^2_{r} 
	&= \left\| \sum_{\beta \le \alpha} {\alpha \choose \beta} D^{\alpha - \beta} g_M D^\beta f\right\|^2_{r} \\
	&\le 2^{\alpha!} \sum_{\beta \le \alpha} {\alpha \choose \beta}^2 \left\| D^{\alpha - \beta}g_M D^\beta f\right\|^2_{r} \\
	&\le 2^{\alpha !} 8\left(B(k)^2+L(k)^2\right)(c_2+1) \sum_{\beta \le \alpha} {\alpha \choose \beta}^2 \left\|D^\beta f\right\|^2_{r}. 
\end{align*}
Next, for $|\beta| < k$ and some constant $c_5 = c_5(r) < \infty$, we have that 
\[
	\left\| D^\beta f \right\|^2_{r} \le c_5 \left\|D^\beta f \right\|^2_{1} = c_5\left\|D^\beta f\right\|^2_{L^2} + c_5\sum_{|\alpha| = 1} \left\|D^\alpha D^\beta f\right\|^2_{L^2} \le c_5\left\|D^\beta f \right\|^2_{L^2} + c_5\sum_{|\alpha| = |\beta|+1} \left\|D^\alpha f \right\|^2_{L^2}, 
\]
and hence 
for some $c_6 = c_6(k, r) < \infty$ and all $M<\infty$, 
\begin{align}
	\sum_{|\alpha| = k}\left\|D^\alpha g_M f\right\|^2_{r} 
	&\le  c_6\sum_{|\alpha|=k} \left\|D^\alpha f\right\|^2_{r} + c_6 \|f\|^2_{k}. \label{eq:eq348r}
\end{align}
Finally, from \eqref{eq:eq346r} and \eqref{eq:eq348r}, for all $M<\infty$,
\[
	\|g_Mf\|^2_{s} = \|g_Mf\|^2_{k} + \sum_{|\alpha|=k}\|D^\alpha g_M f\|^2_{r} \le (c_1 + c_6) \|f\|^2_{k} +c_6 \sum_{|\alpha| = k} \left\|D^\alpha f \right\|^2_{r} \le K \|f\|^2_{s}, 
\]
where $K = c_1 + c_6$. \hfill \qed

\subsection{Proof of Lemma \ref{compactembedding}}
\label{sec:pfcptemb}
Let $\mathbf{s}$, $\mathbf{s}'$, $A$ and $K$ be as in the statement of the lemma. In particular
$A \subset \mathbf{H}^{-\mathbf{s}'}$ is such that 
\begin{equation}\label{eq:eq513}
	B \doteq \sup_{F\in A} \|F\|_{-\mathbf{s}'} <\infty, 
\end{equation}
and every $F \in A$ has support contained in $K$. Recall the functions $e^k_{n, \xi}$ for $(n,\xi,k)\in \sI$ introduced above \eqref{eq:hatphi}.
Let $\{F^l\}_{l\in \N}$ be a sequence in $A$, and for $l \in \NN$ and $(n, \xi) \in \Z\times \R^d$, let
\begin{equation}\label{eq:distFourier}
	\hat F^l(n, \xi) \doteq \left( \hat F_1^l(n,\xi), \ldots, \hat F_d^l(n,\xi)\right), \qquad 
	\hat F^l_k(n,\xi) \doteq \left\il F^l, e^k_{-n, -\xi} \right\ir, \qquad  1\le k \le d. 
\end{equation}
Since $F^l$ has compact support, the evaluation on the right side of the second equality above  is indeed meaningful (see e.g. \cite[Theorem 9.8]{folland}) and for each $l\in \NN$ and $n \in \ZZ$, $\xi \mapsto \hat F^l(n, \xi)$ is in $\sC^{\infty}(\R^d, \R^d)$. Also, using \eqref{eq:eq513}  and the compact support property, one can verify (see \cite[Theorem 9.22]{folland}) that for each $n\in \ZZ$,
\[
	\sup_{l\ge1} \sup_{\xi \in \R^d} \left| \hat F^{l}(n, \xi) \right| < \infty \qquad\mbox{and}\qquad \sup_{l\ge1} \sup_{\xi \in \R^d} \left| D_{\xi} \hat F^l(n,\xi) \right| < \infty.
\]
Thus,
for each $n\in \ZZ$, $\{\hat F^l(n, \cdot), l \in \N\}$ is relatively compact in $\sC(\R^d, \R^d)$. By a standard diagonalization procedure, we can pick a subsequence $\{l_j\}$ such that $\{\hat F^{l_j}(n, \cdot), j \in \N\}$ converges in $\sC(\R^d, \R^d)$ for every $n$ to a limit. We will now show that $F^{l_j}$ is Cauchy in $\mathbf{H}^{-\mathbf{s}}$ which will complete the proof. 
 
 By an argument similar to \cite[Proposition 9.16]{folland}, there are constants $c_1(\mathbf{t},K), c_2(\mathbf{t}, K) < \infty$ for $\mathbf{t}= \mathbf{s}, \mathbf{s}'$ such that for any $F \in \mathbf{H}^{-\mathbf{s}'} \subset \mathbf{H}^{-\mathbf{s}}$  supported on the  compact set $K$ and both $\mathbf{t} = (t_1, t_2) = \mathbf{s}, \mathbf{s}'$, 
 \begin{equation}\label{ea:eq529}
 c_1(\mathbf{t},K)\left\|F\right\|^2_{-\mathbf{t}} 	\le \sum_{n\in \Z} \int_{\R^d} \left| \hat F(n,\xi) \right|^2 \left(1+n^2\right)^{-t_1}\left(1+|\xi|^2\right)^{-t_2}\,d\xi \le  c_2(\mathbf{t},K)\left\|F\right\|^2_{-\mathbf{t}},
 \end{equation}
 where $\hat F(n, \xi)$ is defined as in \eqref{eq:distFourier}.
In particular, for $j,m \in \N$,
 $$
 	c_1(\mathbf{s}, K)\left\|F^{l_j} - F^{l_m} \right\|^2_{-\mathbf{s}} \le  \sum_{n\in \Z} \int_{\R^d} \left| \hat F^{l_j}(n, \xi) - \hat F^{l_m}(n,\xi) \right|^2 \left(1+n^2\right)^{-s_1}\left(1+|\xi|^2\right)^{-s_2}\,d\xi . 
$$
 Fix $M \in \N$. Then,  using $(1+|\xi|^2)^{-s_2} \le (1+|\xi|^2)^{-s'_2}$, we have
\begin{align*}
c_1(\mathbf{s},K)	\left\|F^{l_j} - F^{l_m} \right\|^2_{-\mathbf{s}} &\le  \sum_{-M \le n \le M} \int_{\R^d} \left| \hat F^{l_j}(n, \xi) - \hat F^{l_m}(n,\xi) \right|^2 \left(1+n^2\right)^{-s_1}(1+|\xi|^2)^{-s_2}\,d\xi \\
	&\qquad + \sum_{|n| > M} \int_{\R^d} \left| \hat F^{l_j}(n, \xi) - \hat F^{l_m}(n,\xi) \right|^2 \left(1+n^2\right)^{-s_1}\left(1+|\xi|^2\right)^{-s_2}\,d\xi \\
	&\le \sum_{-M \le n \le M} \int_{\R^d} \left| \hat F^{l_j}(n, \xi) - \hat F^{l_m}(n,\xi) \right|^2 \left(1+n^2\right)^{-s_1}\left(1+|\xi|^2\right)^{-s_2}\,d\xi \\
	&\qquad + c_2(\mathbf{s}',K)\left\|F^{l_j} - F^{l_m} \right\|^2_{\mathbf{s}'}  \frac{1}{\left(1+(M+1)^2\right)^{s_1-s_1'}} \\
	&\le \sum_{-M \le n \le M} \int_{\R^d} \left| \hat F^{l_j}(n, \xi) - \hat F^{l_m}(n,\xi) \right|^2 \left(1+n^2\right)^{-s_1}\left(1+|\xi|^2\right)^{-s_2}\,d\xi \\ 
	&\qquad + \frac{4B^2c_2(\mathbf{s}',K)}{\left(1 + (M+1)^2\right)^{s_1-s_1'}}. 
\end{align*}
Next, for each $|n| \le M$ and $R < \infty$, there is a $C(R)<\infty$ such that
\begin{align*}
	&\int_{\R^d} \left| \hat F^{l_j}(n, \xi) - \hat F^{l_m}(n,\xi) \right|^2 \left(1+n^2\right)^{-s_1}\left(1+|\xi|^2\right)^{-s_2}\,d\xi \\
	&= \int_{\left\{|\xi| \le R\right\}} \left| \hat F^{l_j}(n, \xi) - \hat F^{l_m}(n,\xi) \right|^2 \left(1+n^2\right)^{-s_1}\left(1+|\xi|^2\right)^{-s_2}\,d\xi \\
	&\qquad + \int_{\left\{|\xi| > R\right\}} \left| \hat F^{l_j}(n, \xi) - \hat F^{l_m}(n,\xi) \right|^2 \left(1+n^2\right)^{-s_1}\left(1+|\xi|^2\right)^{-s_2}\,d\xi \\
	&\le 
	C(R)\sup_{|\xi| \le R}  \left| \hat F^{l_j}(n,\xi) - \hat F^{l_m}(n,\xi) \right|^2  + \frac{c_2(\mathbf{s}',K)}{\left(1+R^2\right)^{s_2 - s_2'}} \left\|F^{l_j} - F^{l_m} \right\|^2_{\mathbf{s}'} \\
	&\le C(R) \sup_{|\xi| \le R}  \left| \hat F^{l_j}(n,\xi) - \hat F^{l_m}(n,\xi) \right|^2  +  \frac{4B^2c_2(\mathbf{s}',K)}{\left(1+R^2\right)^{s_2 - s_2'}}. 
\end{align*}
Combining the above estimates and sending $j,m \to \infty$, since $\{\hat F^{l_j}(n, \cdot)\}$ converges for every $n$, we get 
\[
	\limsup_{j,m \to \infty} \left\|F^{l_j} - F^{l_m} \right\|^2_{\mathbf{s}} \le \frac{4B^2(2M+1)c_2(\mathbf{s}',K)}{c_1(\mathbf{s},K)\left(1+R^2\right)^{s_2 - s_2'}} +  \frac{4B^2c_2(\mathbf{s}',K)}{c_1(\mathbf{s},K)\left(1 + (M+1)^2\right)^{s_1-s_1'}}. 
\]
The result now follows on first sending $R \to \infty$ and then $M \to \infty$. \hfill \qed


\begin{thebibliography}{9}

\bibitem{adafou}
	R.~A.~ Adams and J.~F.~ Fournier, 
	``Sobolev Spaces'',
	Academic Press,
	Elsevier, 
	Oxford,
	2003. 
	
	

\bibitem{BouDup}
M.~Bou\'{e} and P.~Dupuis.
\newblock {\em A variational representation for certain functionals of {B}rownian
  motion},
\newblock Ann. Probab., {\bf 26} (1998), 1641--1659. 


\bibitem{BerButPis}
	L.~Bertini, P.~Butt\`a, and A.~Pisante, 
	{\em Stochastic Allen-Cahn approximation of the mean curvature flow: large deviations upper bound}, 
	Arch. Ration. Mec. Anal., 
	{\bf 224} (2017),
	659--707.


\bibitem{BraHep}
	W.~Braun and K.~Hepp, 
	{\em The Vlasov dynamics and its fluctuations in the $1/N$ limit of interacting classical particles}, 
	Comm. Math. Phys., {\bf 56} (1977), 101--113.


\bibitem{BudCon}
	A.~Budhiraja and M.~Conroy, 
	{\em Empirical measure and small noise asymptotics under large deviation scaling for interacting diffusions}, 
	To appear in J. Theoret. Probab. (2021).

	
\bibitem{BudDup}
	A.~Budhiraja and P.~Dupuis,
	{\em A variational representation for positive functionals of infinite dimensional Brownian motion},
	Probab. Math. Statist., 
	{\bf 20} (2000), 39--61.
	
	
	
\bibitem{BudDup_book}
	A.~Budhiraja and P.~Dupuis
	\newblock ``Analysis and Approximation of Rare Events: Representations and
	  Weak Convergence Methods'', vol.~94, 
	\newblock Springer, 2019.



\bibitem{BudDupFis}
	A.~Budhiraja, P.~Dupuis, and M.~Fischer,
	{\em Large deviation properties of weakly interacting processes via weak convergence methods},
	Ann. Probab., 
	{\bf 40} (2012), 74--102. 

\bibitem{Daw}
	D.~A.~Dawson, 
	{\em Critical dynamics and fluctuations for a mean-field model of cooperative behavior}, 
	J. Stat. Phys., {\bf 31} (1983), 29--85.
	
\bibitem{DawGar}
	D.~A.~Dawson and J.~G\"artner, 
	{\em Large deviations from the McKean-Vlasov limit for weakly interacting diffusions},
	Stochastics, {\bf 20} (1987), 247--308.
	
	
\bibitem{DupEll}
	P.~Dupuis and R.~S.~Ellis, 
	``A Weak Convergence Approach to the Theory of Large Deviations'', 
	John Wiley \& Sons, New York, 1997. 


\bibitem{EthKur}
	S.~N.~Ethier and T.~G.~Kurtz, 
	``Markov Processes: Characterization and Convergence'', 
	Wiley, 
	New York, 
	1986. 


\bibitem{flagubgiator} 
	F.~Flandoli, M.~Gubinelli, M.~Giaquinta, and V.~M.~Tortorelli, 
	{\em Stochastic currents},
	Stochastic Process. Appl.,
	{\bf 115} (2005), 1583--1601. 
	
	
\bibitem{FlaTud}
	F.~Flandoli and C.~A.~Tudor, 
	{\em Brownian and fractional Brownian stochastic currents via Malliavin calculus}, 
	J. Funct. Anal., 
	{\bf 258} (2010), 279--306. 
	
	
\bibitem{folland}
	G.~B.~Folland, 
	``Real Analysis: Modern Techniques and Their Applications'',
	2nd ed., 
	John Wiley \& Sons, 
	1999. 

\bibitem{FLOS}
M.~Fornasier, S.~Lisini, C.~Orrieri, and G.~Savar\'{e},  {\em Mean-field optimal control as Gamma-limit of finite agent controls}, European J. Appl. Math., {\bf 30} (2019), 1153--1186. 

	
\bibitem{GMS}
M.~Giaquinta, G.~Modica, and J. ~Sou\v{c}ek, ``Cartesian Currents in the Calculus of Variations I'', Springer,
Berlin, 1998.

\bibitem{gub}
M.~Gubinelli, 
{\em Controlling rough paths}, J. Funct. Anal.,  {\bf 216} (2004), 86--140. 

\bibitem{HIP}
S.~Herrmann, P.~Imkeller, and D.~Peithmann, {\em Large deviations and a Kramers’ type law for self-stabilizing diffusions}, Ann. Appl. Probab., {\bf 18} (2008), 1379--1423.

\bibitem{HerTug} S.~Herrmann and J.~Tugaut, {\em Mean-field limit versus small-noise limit for some interacting particle systems}, Commun. Stoch. Anal., {\bf 10} (2016), 39–55.


	
\bibitem{Kle}
	A.~Klenke, 
	``Probability Theory: A Comprehensive Course'', 
	2nd ed., 
	Springer, 
	Berlin, 
	2013. 
	
	

\bibitem{KurXio}
	T.~G.~Kurtz and J.~Xiong, 
	{\em Particle representations for a class of nonlinear SPDEs}, 
	Stochastic Process. Appl., {\bf 83} (1999), 103--126.

\bibitem{lyo}
T.~J.~Lyons, 
{\em Differential equations driven by rough signals}, Rev. Mat. Iberoam., 
{\bf 14} (1998), 215--310. 


\bibitem{mar}
M.~Mariani,
 {\em A $\Gamma$-convergence approach to large deviations}, 
 Ann. Sc. Norm. Super. Pisa Cl. Sci., 
 {\bf 18} (2018), 951--976.

\bibitem{McK}
	H.~P.~McKean, 
	{\em A class of Markov processes associated with nonlinear parabolic equations}, 
	Proc. Natl. Acad. Sci. USA, 
	{\bf 56} (1966), 1907--1911.

\bibitem{Mel}
	S.~M\'el\'eard,
	{\em Asymptotic behaviour of some interacting particle systems; McKean-Vlasov and Boltzmann models},
	In: ``Probabilistic models for nonlinear partial differential equations (Montecatini Terme, 1995)'', vol. 1627 of 	``Lecture Notes in Math'', Springer, Berlin, 1996, 42--95. 
	
\bibitem{Oel}
	K.~Oelschl\"ager, 
	{\em A martingale approach to the law of large numbers for weakly interacting stochastic processes}, 
	Ann. Probab., {\bf 12} (1984), 458--479.

\bibitem{NezPalVal}
	E.~Di Nezza, G.~Palatucci, and E.~Valdinoci, 
	{\em Hitchhiker's guide to fractional Sobolev spaces}, 
	Bull. Sci. Math, {\bf 136} (2012), 521--573. 


\bibitem{Orr}
	C.~Orrieri, 
	{\em Large deviations for interacting particle systems: joint mean-field and small-noise limit}, 
	Electron. J. Probab., {\bf 25} (2020), 1--44. 


\bibitem{RST}
 G.~D.~Reis, W.~Salkeld, and J.~Tugaut, {\em Freidlin-Wentzell LDP in path space for Mckean-Vlasov equations and the functional iterated logarithm law}, Ann. Appl. Probab., {\bf 29} (2019), 1487--1540. 

\bibitem{ShiTan}
	T.~Shiga and H.~Tanaka,
	{\em Central limit theorem for a system of Markovian particles with mean field interactions},
	Probab. Theory Related Fields, {\bf 69} (1985), 439--459.





\bibitem{Vil}
	C.~Villani, 
	``Optimal Transport, Old and New'', 
	Springer, 
	New York, 
	2009. 
	



\end{thebibliography}
\end{document}